\documentclass[11pt,twoside]{article}
\usepackage{latexsym}
\usepackage{amssymb,amsbsy,amsmath,amsfonts,amssymb,amscd,color}
\usepackage{epsfig, graphicx}
\usepackage{hyperref}

\newcommand{\commentout}[1]{}

\newcommand {\Chi} {{\bf \raise 2pt \hbox{$\chi$}} }
\newcommand {\f}   {\frac}
\newcommand {\p}   {\partial}
\newcommand{\dis}{\displaystyle}

\newcommand{\beq}{\begin{equation}}
\newcommand{\eeq}{\end{equation}}
\newcommand{\bea} {\begin{array}{rl}}
\newcommand{\eea} {\end{array}}
\newcommand{\bepa}{\left\{ \begin{array}{l}}
\newcommand{\eepa} {\end{array}\right.}

\newcommand{\vitro}{{\it vitro}}
\newcommand{\vivo}{{\it vivo}}

\newtheorem{theorem}{Theorem}[section]

\newcommand{\qed}{{ \hfill
                       {\unskip\kern 6pt\penalty 500 \raise -2pt\hbox{\vrule\vbox to 6pt{\hrule width 6pt
                       \vfill\hrule}\vrule} \par}   }}
\title{\Large \bf An accurate front capturing scheme for tumor growth models with a free boundary limit}
\author{Jian-Guo Liu\thanks{Department of Mathematics and Department of Physics, Duke University} \and
  Min Tang\thanks{Department of mathematics, Institute of Natural Sciences and MOE-LSC. Shanghai Jiao Tong University}
\and  Li Wang\thanks{Department of Mathematics, Computational and Data-Enabled Science 
and Engineering Program, State University of New York at Buffalo} 
\and Zhennan Zhou \thanks{Department of Mathematics, Duke University}
}

\date{\today}

\begin{document}
\maketitle
\pagestyle{plain}
\begin{abstract}
We consider a class of tumor growth models under the combined
effects of density-dependent pressure and cell multiplication, with a free boundary model as its singular limit when the pressure-density relationship becomes highly nonlinear. In particular, the constitutive law connecting pressure $p$ and density $\rho$ is $p(\rho)=\frac{m}{m-1} \rho^{m-1}$, and when $m \gg 1$, the cell density $\rho$ may evolve its support due to a pressure-driven geometric motion with sharp interface along the boundary of its support. The nonlinearity and degeneracy in the diffusion bring great challenges in numerical simulations, let alone the capturing of the singular free boundary limit. Prior to the present paper, there is lack of standard mechanism to numerically capture the front propagation speed as $m\gg 1$. In this paper, we develope a numerical scheme based on a novel prediction-correction reformulation that can accurately approximate the front propagation even when the nonlinearity is extremely strong. We show that the semi-discrete scheme naturally connects to the free boundary limit equation as $m \rightarrow \infty$, and with proper spacial discretization, the fully discrete scheme has improved stability, preserves positivity, and implements without nonlinear solvers. Finally, extensive numerical examples in both one and two dimensions are provided to verify the claimed properties and showcase good performance in various applications.  
\end{abstract}


\pagenumbering{arabic}

\section{Introduction} 

In this paper, we are concerned with a tumor growth model for the cell density function $\rho(\mathbf x,t)$, whose governing equation is given by:
\begin{equation}\label{eq:n}
\frac{\partial}{\partial t} \rho - \nabla \cdot \left( \rho \nabla p(\rho) \right)=\rho G(c,p), \quad \mathbf x\in \mathbb R^d, \, t\ge0,
\end{equation}
where $c(\mathbf x,t)$ is a nutrient concentration which may depend on the cell density function \cite{BOBAM,PTV}. The tumor cells are transported due to the spacial availability, which is described by the negative gradient of a pressure function, and  they proliferate with a growth rate that depends on both the nutrient concentration and the pressure. The pressure depends on the cell density through the state equation \begin{equation}
\dis p(\rho) = \frac{m}{m-1} \rho^{m-1}\,.
\label{eq:p}
\end{equation}and other forms as in \cite{TVCVDP} are possible. 
 $G(c,p)$ represents the growth factor, which is assumed to be monotonic increasing in the $c$ variable, monotonic decreasing in the $p$ variable and may take negative values, i.e.,
\[
\frac{\partial}{\partial c}G(c,p)\ge 0,\quad \frac{\partial}{\partial p}G(c,p)\le 0, \quad  G(\bar c,p)=0, \quad \mbox{for}\, \,\mbox{some} \, \,\bar c>0.
\]
Without loss of generality, we denote by $c_b$ the maximum nutrient concentration that the environment can provide, and we assume that the following condition is satisfied
\begin{equation}\label{est:cbound}
0 \le c(\mathbf x,t) \le c_b, \quad \quad \mathbf x\in \mathbb R^d,  \quad \forall t \ge 0.
\end{equation}
In Section \ref{sec:num}, the nutrient models in numerical tests will be specified, which satisfy the boundedness condition \eqref{est:cbound} according to \cite{LTWZ,PQV}.

With the uniform boundedness assumption on $C$, we further assume there is there exists a constant $P$ such that, $G(c_b,P)=0$, which prevents the pressure $p$ exceeding this limit.
Accordingly, we complement this system with an initial  condition that satisfies
\begin{equation} \label{hypInit} \left\{�\begin{array}{l}
\rho(0,x)=\rho^{ini}(x) > 0, \qquad  \rho^{ini}\in L^1\cap L^\infty,
\\[6pt]
p^{ini}=\f{m}{m-1}(\rho^{ini})^{m-1} \leq P.
\end{array} \right.\end{equation}


It has been proved in \cite{PQV} that, when $m \to \infty$, the solution of \eqref{eq:n} converges to the solution of a free boundary problems supported on $\Omega(t)$, which is a set on which $p_\infty >0$. The geometric motion of $\Omega(t)$ is governed by the limiting pressure $p_\infty$. To understand this, we first write down the equation for $p$. Multiplying equation \eqref{eq:n} by $p'(\rho)$, it becomes
\[
\f{\p}{\p t} p= \rho p'(\rho) \Delta p + |\nabla p|^2  
+p'(\rho)\rho G(c,p)\,,
\]
and since $p= \frac{m}{m-1} \rho^{m-1}$, we find
\begin{equation}    \label{eq:pk}
\f{\p}{\p t} p= (m-1) p \Delta p + |\nabla p|^2 + (m-1)p G(c,p)\,.
\end{equation}
Sending $m \rightarrow \infty$ in the above equation, we formally have the following `complementary relation':
\beq 
p_\infty \big(\Delta p_\infty+G(c,p_\infty) \big)= 0\,.
\label{eqp1}
\eeq
The normal velocity of the boundary $\partial \Omega$ is $v=-\nabla p_\infty \cdot \hat n$, where $\hat n (x,t)$ is the unit outer normal direction on the boundary. 

We point out that in the limit of $m \rightarrow \infty$, the constitutive law \eqref{eq:p} gets lost, that is, $\rho_\infty$ and $p_\infty$ no longer satisfy the relation \eqref{eq:p}. Instead, $\rho_\infty$ jumps from $0$ to $1$ across the boundary $\partial \Omega$, and if starts with a characteristic function, it maintains as a characteristic function of $\Omega (t)$ along dynamics, and $p_\infty$ satisfies 
\beq
p_\infty \in P_\infty(\rho) = \left\{\begin{array}{ll} 
0, & 0\leq \rho_\infty < 1, \\[2mm]
[0,\infty),\qquad & \rho_\infty=1. 
\end{array}
\right.
\label{eq:graph}
\eeq
This kind of solution behavior are also considered in the congested  crowd  transport models, or the congested aggregation models \cite{AKY,Degond,CKY}.


In term of numerics, the difficulties are two-fold. On one hand, it is well known that the solutions of degenerate parabolic equations may have no classical spatial derivatives at a subset of the domain, the profile of the solution may have sharp interfaces near its support and the boundary of the support propagates with a finite speed \cite{Vazquez}. Moreover, the solution to the reaction-diffusion equation \eqref{eq:n} have a {\it time-dependent} support. Due to the lack of smoothness at the sharp interfaces, prevailing parabolic solvers may lose the convergence order for degenerate problems, which may even result in incorrect propagation speed of the boundary of the density support.  Many numerical methods have been proposed for the simulations of degenerate parabolic equations, including the finite element method \cite{AWZ,BCW}, finite volume scheme \cite{BF,EGHM}, finite difference method \cite{KRT,LSZ}, relaxation scheme which exhibits the merit of the Jin-Xin relaxation model \cite{JinXin,NPT}, discontinuous galerkin method \cite{ZW}, or some approach based on perturbation and regularization \cite{PY}. However, to the best of our knowledge, no existing numerical methods have ever investigate the possibility of preserving the free boundary limit of the degenerate reaction-diffusion equation.

 On the other hand, the nonlinearity introduced by large $m$ also brings severe numerical challenges. In theory, there are two ways to handle the nonlinearity. One is to use a fully implicit scheme for the nonlinear terms and solve the resulting discrete nonlinear system by some iterative method. However, as $m$ increases, the growing multiplicity and the stiffness of the Jacobi matrix of the resulting algebraic equation makes the implementation of iterative methods infeasible. The other one is to treat the nonlinear diffusion term semi-implicitly as in \cite{LTWZ}. With refined spatial grids and time steps, it has been shown that the numerical method in \cite{LTWZ} gives accurate approximation to the profile of the density. However, if the spacial grid and the time steps are not resolved enought, either the scheme is not stable or the numerical front position deviates from its true location. In particular, it has been observed that, a parabolic type of CFL conditions are necessary to guarantee accurate numerical approximation, and the constraints are more severe as $m$ increases.

The goal of this present paper is to develop an efficient numerical method for the nonlinear reaction-diffusion equation \eqref{eq:n} to overcome those difficulties. Especially, we want to show that the proposed scheme has improved stability constraints and, more importantly,  it effectively preserves the free boundary limit as $m\rightarrow \infty$.


Since the nonlinear diffusion term causes the primary challenge in numerical simulations,
to illustrate the novelty of our scheme, we consider a family of  diffusion equations of the form 
\beq\label{eq:hrho}
\p_t\rho=\Delta h_{m}(\rho)\,.
\eeq
Here, $\{h_{m}(\cdot)\}_{m \in I}$ stand for a family of nonlinear functions, and $I$ is the index set.  For example, when $h_m(\rho)=\rho^m$ and $I=\{m \in \mathbb R,m>1\}$, it recovers the porous media equation \cite{Vazquez}, or the tumor growth model \eqref{eq:n} when $G\equiv 0$. The nonlinearity changes in $h_m$ as $m$ varies.

Our new perspective starts with deriving the {\it equivalent velocity equation} for the cell density model. Specifically, we rewrite the original nonlinear diffusion equation \eqref{eq:hrho} as
\[
\p_t\rho=\nabla\cdot \left(\rho\frac{h'_m(\rho)}{\rho}\nabla\rho \right) :=-\nabla\cdot(\rho \mathbf{u})\,,
\]
where $\mathbf{u}$ satisfies  
\[
\mathbf{u}=- g_m(\rho) \nabla_\mathbf{x} \rho\,, \quad g_m(\rho)= \frac{h'_m(\rho)}{\rho}\,
\]
can be considered as velocity. 

As written, the new variable $\mathbf{u}$ suggests a way to single out the velocity from the density equation. Given the fact that velocity field plays the major role in expanding the front at the correct speed, we come up with the following new system that evolves the cell density and the velocity field simultaneously if they are initialized with compatible profiles:
\begin{align*}
\begin{cases}{}
&\partial_t\rho+\nabla\cdot(\rho \mathbf{u}) =0, \\
&\partial_t \mathbf{u}+\nabla \big(g_m(\rho)\nabla \cdot (\rho \mathbf{u})\big)=0.
\end{cases}{}
\end{align*}
Here the equation for $\mathbf{u}$ is obtained by
multiplying both sides of the equation for $\rho$ by $g_m(\rho)=f_m'(\rho)$ and taking the derivative with respect to $\mathbf{x}$.

However, as we shall show in Section~\ref{reform}, this system is not stable in the sense that perturbations in the constitutive relationship $\mathbf{u}= - g_m(\rho) \nabla_\mathbf{x} \rho $ does not decay in time. For this reason, we propose the following relaxation system 
\begin{equation} \label{eq:relax}
\left\{ 
\begin{aligned}
&\partial_t\rho+\nabla\cdot(\rho \mathbf{u}) =0,  \\
&\partial_t \mathbf{u}+\nabla \big(g_m(\rho)\nabla \cdot (\rho \mathbf{u})\big)=- \frac{1}{\epsilon^2} (\mathbf{u}+ g_m(\rho) \nabla \rho), 
\end{aligned} \right.
\end{equation}
where $\epsilon \ll 1$ is the relaxation parameter. Since we artificially expand the size of the system of equation, and the extra relaxation term in the above system will reinforce the constitutive relation and thus  helps to stabilize the discrepancy between the auxiliary quantity $\mathbf{u}$ and its consistent representation $-\nabla_\mathbf{x} g_m(\rho) $.

We shall show in the paper that, as $\epsilon \rightarrow 0$, the relaxation system \eqref{eq:relax} leads to a system of augmented differential algebraic equations (DAEs), whose numerical method can be understood as a prediction correction method. And for the tumor growth model, the proposed numerical method for the augmented DAEs is compatible with the free boundary limit of the cell density model, and is shown to capture the correct front speed of the moving boundaries.

The rest of the paper is organized as follows. The relaxation system of the cell density model are formulated and analyzed in Chapter 2. We also show augmented DAEs as the limiting system of the relaxation system. In Chapter 3, we propose the prediction correction method of the augmented DAEs in the semi-discrete level, and proves how it connects to the free boundary limit---the Hele-Shaw model.  The fully discrete scheme is presented in Chapter 4, along with some further numerical analysis results. At last, in Chapter 5 we confirm the properties of our numerical method with numerous test examples, and we also explore the applications of the scheme for a few extended models.



\section{Relaxation system and the prediction correction formulation} \label{reform}

\subsection{Reformulating the $\mathbf u$ equation}

In this section, we  focus on the density equation \eqref{eq:n}, and formulate a relaxation system out of it, which provides a superior framework for accurate numerical approximations. 
As briefly mentioned in the introduction, let us first introduce the velocity field that transports the cell population: 
\begin{equation}\label{eqn:urho}
 \mathbf{u}=-\nabla p(\rho)= -\f{m}{m-1}\nabla \rho^{m-1}\,,
\end{equation}
then the equation for $\rho$ reads
\begin{equation}\label{eqn:conv}
\p_t \rho+\nabla \cdot (\rho  \mathbf{u} )= \rho G(c)\,.
\end{equation}
Obviously, together with \eqref{eqn:urho}, the reformulation \eqref{eqn:conv} is as difficult to solve numerically as the original one. However, as discussed in the introduction, one can derive an equation for $\mathbf{u}$ from the equation for $\rho$ and evolve the cell density and the velocity field simultaneously. When $\rho$ and $\mathbf{u}$ are initialized with compatible profiles, the following transport system is considered instead: 
\begin{equation} \label{eq:convection}
\left\{\begin{aligned} 
&\p_t \mathbf{u}=m\nabla \Big(\rho^{m-2}\big(\nabla \cdot (\rho \mathbf{u})- \rho G(c)\big)\Big),\\
&\p_t \rho+ \nabla \cdot (\rho  \mathbf{u} )= \rho G(c),\\
&\mathbf{u}(x,0)=-\f{m}{m-1}\nabla \rho^{m-1}(x,0).\\
\end{aligned}
\right.	
\end{equation}
Note that here the first two equations are equivalent to the original cell density only when the last equation is satisfied. Indeed, if we introduce the discrepancy term $$W=\mathbf{u}+\f{m}{m-1}\nabla \rho^{m-1}\,,$$ it is straightforwardly to verify that
\[
\p_t W(x,t) =0, \quad W(x,0)=0.
\]
Therefore, we conclude $W(x,t)=0$, and the consistency condition \eqref{eqn:urho} is always satisfied. 

However, when solving the system \eqref{eq:convection} numerically, the discretization error destroys the relation \eqref{eqn:urho}. In fact, $W$ deviates from $0$ at every time step, and since there is no damping in the $W$ equation, the local truncation error accumulates and leads to a large global error. It is worth emphasizing that, $W$ is not a direct quantity of interest, however, it determines whether the transport system is an equivalent reformulation of the original density equation  \eqref{eq:n}. 

To illustrate the propagation of the discrepancy when numerical error is present, we add a small perturbation to \eqref{eq:convection}:
\begin{equation} \label{eq:convectiondelta}
\left\{\begin{aligned} 
&\p_t \mathbf{u}=m\nabla \Big(\rho^{m-2}\big(\nabla \cdot (\rho \mathbf{u})- \rho G(c)\big)\Big)+\delta_1,\\
&\p_t \rho+ \nabla \cdot (\rho  \mathbf{u} )= \rho G(c)+\delta_2,\\
&\mathbf{u}(x,0)=-\f{m}{m-1}\nabla \rho^{m-1}(x,0).\\
\end{aligned}
\right.	
\end{equation} 
Here, $\delta_1$, $\delta_2$ are small perturbation functions (in numerical schemes, they are just local truncation errors). Then by direct calculation, we get 
\[
\p_tW(x,t)=\delta_1+m\nabla(\rho^{m-2}\delta_2).
\]
Clearly,  $\|W(\cdot,t) \|_{L^\infty}$ may increase linearly in time, which breaks the consistency condition \eqref{eqn:urho} in the transport system \eqref{eq:convection}. In particular, $\|W(\cdot,t) \|_{L^\infty}$ may increase faster when $m$ is large.

\subsection{The relaxation system and its properties}

To ensure the consistency condition, we instead consider the following equation for the discrepancy with an artificial damping
\[
\p_t W(x,t) =- \frac{1}{\epsilon^2} W(x,t), 
\]
where $0<\epsilon \ll 1$ is the relaxation constant. 
This discrepancy equation leads to the following relaxation system
\begin{equation} \label{eq:convection2}
\left\{\begin{aligned} 
&\p_t \mathbf{u}=m\nabla \Big(\rho^{m-2}\big(\nabla \cdot (\rho \mathbf{u})- \rho G(c)\big)\Big) - \frac{1}{\epsilon^2} \Big( \mathbf{u}+\f{m}{m-1}\nabla \rho^{m-1}\Big), \\
&\p_t \rho+ \nabla \cdot (\rho  \mathbf{u} )= \rho G(c),\\
&\mathbf{u}(x,0)=-\f{m}{m-1}\nabla \rho^{m-1}(x,0).\\
\end{aligned}
\right.	
\end{equation}
Correspondingly, the relaxation system with perturbations is given by
\begin{equation} \label{eq:convection2p}
\left\{\begin{aligned} 
&\p_t \mathbf{u}=m\nabla \Big(\rho^{m-2}\big(\nabla \cdot (\rho \mathbf{u})- \rho G(c)\big)\Big) - \frac{1}{\epsilon^2} \Big( \mathbf{u}+\f{m}{m-1}\nabla \rho^{m-1}\Big)+\delta_1, \\
&\p_t \rho+ \nabla \cdot (\rho  \mathbf{u} )= \rho G(c)+\delta_2, \\
&\mathbf{u}(x,0)=-\f{m}{m-1}\nabla \rho^{m-1}(x,0),\\
\end{aligned}
\right.	
\end{equation}
where again, $\delta_1$, $\delta_2$ are small perturbation functions that may represent local truncation errors from numerical schemes.

As we shall see below, the advantage of the relaxation system \eqref{eq:convection2} can be seen in two ways. First of all, when $\epsilon$ is small, the consistency condition \eqref{eqn:urho} is effectively preserved along the dynamics even in the presence of perturbations. Secondly, the relaxation term in \eqref{eq:convection2p} can be shown of the same order as local perturbations, so that even with perturbations the $\mathbf u$ equation in \eqref{eq:convection2p} still well approximates the $\mathbf u$ equation in \eqref{eq:convection}. This property is important as the accuracy in approximating  $\mathbf{u}$ determines the accuracy of the front position. In particular, in the absence of perturbation, the relaxation term is exactly $0$.

More specifically, we show in the following that 1) when $\epsilon\to 0$, $W(x,t)\to 0$ for $t \ge0$, which indicates the preservation of the consistency condition \eqref{eqn:urho}; 2) the relaxation terms in \eqref{eq:convection2} $\frac{W(x,t)}{\epsilon^2}$ is of the same order as the perturbations.

For \eqref{eq:convection2p}, we find that $W(x,t)$ satisfies 
\begin{equation*} \label{eq:Wrelax}
\p_tW(x,t)=-\frac{1}{\epsilon^2}W(x,t)+\delta_1+m\nabla(\rho^{m-1}\delta_2)\,,
\end{equation*} 
whose solution can be represent as follows
\[
W(x,t)=\exp\left(-\frac{t}{\epsilon^2}\right)W(x,0)+\int_0^t\exp\left(-\frac{t-\tau}{\epsilon^2}\right)\big(\delta_1(\tau)+m\nabla(\rho^{m-1}\delta_2)(\tau)\big)\,d\tau \,.
\]
Then, with integration by parts, we get
$$\begin{aligned}
\frac{W(x,t)}{\epsilon^2}=&\frac{1}{\epsilon^2}\exp\left(-\frac{t}{\epsilon^2}\right)\big(W(x,0)-\epsilon^2 \delta_1(0)-\epsilon^2 m\nabla(\rho^{m-1}\delta_2)(0)\big)+\delta_1(t)+m\nabla(\rho^{m-1}\delta_2)(t)\\
&-\int_0^t\exp\left(-\frac{t-\tau}{\epsilon^2}\right)\p_\tau\big(\delta_1(\tau)+m\nabla(\rho^{m-1}\delta_2)(\tau)\big)\,d\tau,
\end{aligned}
$$which indicates that when $\epsilon\to 0$, 
\[
\f{W(x,t)}{\epsilon^2}= \delta_1(t)+m\nabla(\rho^{m-1}\delta_2)(t)+O(\epsilon^2).
\]
Obviously, $\f{W(x,t)}{\epsilon^2}$ is at the same order of $\delta_1$ and $\delta_2$ for small $\epsilon$, and when $\epsilon\to 0$, $W(x,t)\to 0$ for $t \ge0$.

 
We shall further explore the role of the relaxation term by the Chapman-Enskog expansion below similar to what is done in \cite{JinXin}.
Conducting the Chapman-Enskog expansion to the relaxation system \eqref{eq:convection2}, we get
$$\begin{aligned}
\mathbf{u}&=-\frac{m}{m-1}\nabla\rho^{m-1}-\epsilon^2\p_t \mathbf{u}+\epsilon^2 m\nabla \Big(\rho^{m-2}\big(\nabla \cdot (\rho \mathbf{u})- \rho G(c)\big)\Big) \\
&=-\frac{m}{m-1}\nabla\rho^{m-1}-\epsilon^2\p_t\Big(-\frac{m}{m-1}\nabla\rho^{m-1}-\epsilon^2\p_t \mathbf{u}+\epsilon^2 m\nabla \Big(\rho^{m-2}\big(\nabla \cdot (\rho \mathbf{u})- \rho G(c)\big)\Big)\Big)\\
&\qquad +\epsilon^2 m\nabla \Big(\rho^{m-2}\big(\nabla \cdot (\rho \mathbf{u})- \rho G(c)\big)\Big)\\
&=-\frac{m}{m-1}\nabla\rho^{m-1}+\epsilon^4\p_t\Big(\p_t\mathbf{u}-m\nabla \Big(\rho^{m-2}\big(\nabla \cdot (\rho \mathbf{u})- \rho G(c)\big)\Big).
\end{aligned}
$$
If we continue the calculations, formally, all higher order terms will cancel each other, which implies that \eqref{eq:convection2} is equivalent to \eqref{eq:n}, not only to the leading order, but also at each higher order. This is not surprising since analytically we always have the consistency condition \eqref{eqn:urho} and the extra relaxation term is simply zero, and thus \eqref{eq:convection2} is equivalent to \eqref{eq:convection}. On the other hand, when solving \eqref{eq:convection} numerically, the consistency condition is no longer satisfied with full accuracy, and the relaxation term in \eqref{eq:convection2} will add extra contributions to counteract the perturbations caused by numerical approximations so that it reinforces the consistency relation between $u$ and $\rho$ without adding any dissipation. 
Therefore, the relaxation term in \eqref{eq:convection}, although remaining zero analytically, plays an crucial role in  numerical approximation.  



\subsection{The prediction-correction method}
Based on the analysis in the previous sections, we choose to solve \eqref{eq:convection2} instead and propose the following time splitting method:
\begin{equation}\label{eqn:dae1}
\left\{\begin{aligned}
&\p_t \rho+\nabla \cdot (\rho  \mathbf{u} )=\rho G(c),\\
&\p_t u=m\nabla  \Big(\rho^{m-2}\big(\nabla \cdot (\rho \mathbf{u})-\rho G(c)\big)\Big),
\end{aligned}
\right.	
\qquad
\left\{\begin{aligned}
&\p_t\rho=0,\\
&\p_t \mathbf{u}=-\frac{1}{\epsilon^2}(\mathbf{u}+\f{m}{m-1}\nabla \rho^{m-1})\,.
\end{aligned}
\right.
\end{equation}
At every time step, given $(\rho^n,\mathbf{u}^n)$, one solves the left system in \eqref{eqn:dae1} for one time step $\Delta t$ and obtains the intermediate values $(\rho^{*}, \mathbf{u}^{*})$, and then solve the second system of equations in \eqref{eqn:dae1} to get $(\rho^{n+1},\mathbf{u}^{n+1})$.

When $\epsilon\to 0$, the second step in \eqref{eqn:dae1} reduces to 
\begin{equation}\label{eqn:dae2}
\p_t\rho=0,\qquad
\mathbf{u}(x,t)=-\f{m}{m-1}\nabla \rho^{m-1}(x,t),
\end{equation}
which can be understood as a projection step. Note that in the projection step, $\rho$ is a constant in time, namely $\rho^*=\rho^{n+1}$.   Therefore, the time splitting method for the fully relaxed system ($\epsilon=0$) becomes: 
\begin{equation}\label{eqn:dae3}
\left\{\begin{aligned}
&\p_t \rho+\nabla \cdot (\rho  \mathbf{u} )=\rho G(c),\\
&\p_t \mathbf{u}=m\nabla  \Big(\rho^{m-2}\big(\nabla \cdot (\rho \mathbf{u})-\rho G(c)\big)\Big),
\end{aligned}\qquad
\mathbf{u}(x,t)=-\f{m}{m-1}\nabla \rho^{m-1}(x,t).
\right.		
\end{equation}
We use the first two equations to numerically evolve the systems  while taking the last equation as a constraint, which induces the following prediction-correction method: at each time step, we first solve the left two equations in \eqref{eqn:dae3}  from $\left(\rho^n, \mathbf{u}^n\right)$ to obtain $\left(\rho^{n+1}, \mathbf{u}^*\right)$ (here in $\rho$ equation, the convection term is treated semi-implicitly, i.e., $\nabla\cdot (\rho^n u^*)$), and then use the right equation to enforce $\mathbf{u}^{n+1}=-\f{m}{m-1}\nabla (\rho^{n+1})^{m-1}$. Here $\mathbf{u}^*$  is an intermediate value which is only used to help compute $\rho^{n+1}$ better. Actually, $\mathbf{u}^*$ can be a good approximation to $\mathbf{u}(t^{n+1})$, but, $\mathbf{u}^*$ and $\rho^{n+1}$ may not satisfy the consistency condition \eqref{eqn:urho}. 

We remark that, the introduction of $\mathbf{u}^*$ gives us the freedom to solve for $\rho^{n+1}$ stably and accurately without worrying about the  constraint \eqref{eqn:urho}. Hence, $\mathbf{u}^*$ can be viewed as a prediction, which is corrected by the consistency condition.  The correction is essentially a projection onto the solution manifold (as shown in Figure~\ref{fig:scheme}), which is not carried out on $\mathbf{u}^*$ directly, but by using $\rho^{n+1}$ and the explicit consistency condition \eqref{eqn:urho}. It is worth mentioning that similar projection ideas have been introduced to other equations in numerical simulations, like the incompressible flows (\cite{Chorin1,Chorin2,HJL}) and the Landau-Lifshitz equation (\cite{CLZ,WGE}).
\begin{figure}[ht]
\centering 
\includegraphics[scale=0.6]{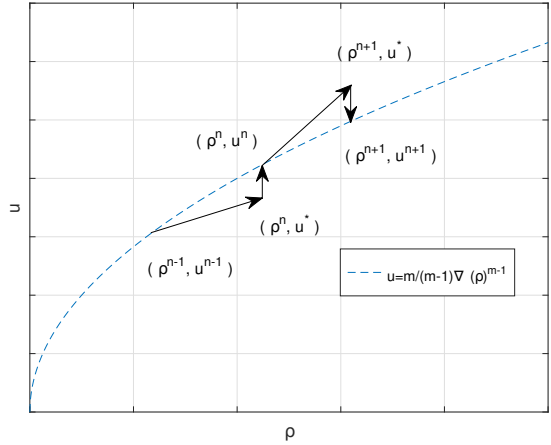}
\caption{Schematic plot of the prediction-correction numerical simulation of the augmented differential algebraic equations \eqref{eqn:dae3}. The correction is only done on the u component by projecting onto the solution manifold.}
\label{fig:scheme}
\end{figure}
%

\section{Time discretization and connections to the Hele-Shaw model}
In this section, we aim to propose a semi-discrete form of \eqref{eqn:dae3} that yields a good stability condition. Note that 
although in the prediction step of \eqref{eqn:dae3}, the two equations are linear respectively in $\rho$ and $u$, they are still nonlinear equations. If we naively update $\rho$ and $u$ explicitly there, it is equivalent to explicitly updating original equation \eqref{eq:n}, which is subject to severe stability constraints that highly depend on the strength of the nonlinearity encoded by $m$. Instead, we consider an implicit-explicit discretization for \eqref{eqn:dae3} as follows: 
\begin{align}
\f{\mathbf{u}^{n*}-\mathbf{u}^n}{\Delta t}&=m\nabla \Big((\rho^{n})^{m-2}\big(\nabla \cdot  (\rho^n \mathbf{u}^{n*})-\rho^n G(c^{n},p(\rho^n))\big)\Big) \,, \label{eq:ustar}\\
\f{\rho^{n+1}-\rho^n}{\Delta t}&=-\nabla \cdot (\rho^n \mathbf{u}^{n*})+\rho^{n+1} G(c^{n},p(\rho_n)) \,,\label{eq:rhonew} \\
\mathbf{u}^{n+1}&=-\f{m}{m-1}\nabla (\rho^{n+1})^{m-1} \,. \label{eq:ustar0}
\end{align}
In this scheme, each equation  can be solved consecutively, and each of them is only semi-implicit, which means that no nonlinear solver is needed in implementing the scheme.  For the ease of analysis, we assume $G$ is a constant in this section to conduct the numerical analysis. In the numerical examples, we shall instead consider more general time and spatial dependent $G$ in various examples.

\subsection{The connection to the Hele-Shaw model} \label{sec:heleshaw}


In the free boundary limit, $m\to\infty$, with properly prepared initial conditions, the density function $\rho$ behaves as a time dependent characteristic function with a sharp front, whose geometric motion is governed by the Hele-Shaw model for the pressure (see e.g.  \cite{PQV}). 


In order to examine the connection of the semi-discrete scheme to the Hele-Shaw model, we aim to show that $u^*$, computed from \eqref{eq:ustar} and used in \eqref{eq:rhonew} to propagate $\rho^{n+1}$, satisfies the Hele-Shaw model when $m \rightarrow \infty$, with $\Delta t$ fixed. 
 For $m\gg 1$, the right hand side of \eqref{eq:ustar} is dominating, and the leading order with respect to $m$ reads
\[
\vec 0=\nabla \Big((\rho^{n})^{m-2}\big(\nabla \cdot (\rho^n \mathbf{u}^{n*})-\rho^n G\big)\Big)= \nabla \Big((\rho^{n})^{m-1}  \nabla \cdot \mathbf{u}^{n*}+(\rho^{n})^{m-2} (\nabla \rho^n) \cdot \mathbf{u}^{n*}-(\rho^n)^{m-1} G \Big).
\]
By the definition of the pressure $p(\rho)$, and since $\mathbf{u}^{n}=-\f{m}{m-1}\nabla(\rho^{n})^{m-1} $ from the previous step, we obtain
\[
\vec 0= \nabla \Big( \frac{m-1}{m} p(\rho^n) \nabla \cdot \mathbf{u}^{n*} + \frac{1}{m} \mathbf{u}^n \cdot \mathbf{u}^{n*} - \frac{m-1}{m} p(\rho^n) G\Big).
\]
For $m \gg 1$, this effectively reduces to 
\[
\vec 0= \nabla \Big(  p(\rho^n) \big(\nabla \cdot  \mathbf{u}^{n*} - G\big) \Big),
\]
which indicates that there exists a scalar constant $a$ that is independent of space, such that
\[
p(\rho^n) \big(\nabla \cdot \mathbf{u}^{n} - G\big) = a.
\]
Note that, due to the degenerate diffusion, if we initially start with a density that has compact support, then $\rho$ as well as the pressure $p(\rho^n)$ remain compactly supported, which implies
\[
p(\rho^n) \big(\nabla \cdot \mathbf{u}^{n*} - G\big) = 0.
\]
Therefore, if we denote the support of $p(\rho^n)$ by $\Omega_n$, then the velocity field $\mathbf{u}^{*}$ satisfies
\begin{equation} \label{eq:ap}
\nabla \cdot  \mathbf{u}^{n*} - G=0, \quad x \in \Omega_n.
\end{equation}
%

{From \eqref{eq:ustar}, since $p(\rho^n)=0$ outside of $\Omega_n$, we have ${\mathbf{u}^{n*}}=\mathbf{u}^n$ for $x\in\mathbb{R}^d/\bar\Omega_n$. From \eqref{eq:ustar0} of the previous step, $\mathbf{u}^n$ is supported in $\Omega_n$, so it is with $\mathbf{u}^{n*}$.} The above limit is for the velocity field $\mathbf{u}^{n*}$, which is closely related to the limiting pressure. We further show that the velocity field can be written as the gradient of a scalar function, which plays the role of the pressure.

We observe that,  \eqref{eq:ustar} can be rewritten as
\[
{u^{n*}}=u^n +\nabla \Big( m{\Delta t}  (\rho^{n})^{m-2}\big(\nabla \cdot  (\rho^n \mathbf{u}^{n*})-\rho^n G\big)\Big), 
\]
and recall $\mathbf{u}^n = -\f{m}{m-1}\nabla (\rho^{n})^{m-1}$, then $\mathbf{u}^{n*}$ can also be written as the gradient of a scalar function, which we denote by $-p^{n*}$, namely $\mathbf{u}^{n*} =- \nabla p^{n*}$. Thus, \eqref{eq:ap} becomes
\[
-\Delta p^{n*} =G, \quad x \in \Omega_n.
\]

Next, we try to find the boundary condition for $p^{n*}$ on  $\partial\Omega_n$ by assuming that $p^{n*}$ is continuous across $\Omega_n$.  Due to the definition of $p^{n*}$ such that $\mathbf{u}^{n*} =- \nabla p^{n*}$ and the fact that $\mathbf{u}^{n\ast}$ is supported in $\Omega_n$,
we conclude that $p^{n*}$ is constant along $\partial \Omega_n$.  Then we have for some constant $d$,
\begin{equation} \label{eq:pstar}
\left\{ \begin{split}-\Delta p^{n*} =G, & \quad x \in \Omega_n, \\
p^{n*}=d,& \quad x \in \partial \Omega_n.
\end{split} \right.
\end{equation}
The solutions to \eqref{eq:pstar} yield the same $\mathbf{u}^{n\ast}=-\nabla p^{n\ast}$ for different constant $d$.  Especially, we can view $u^{n*}$ as $-\nabla p^{n*}_0$, where $p^{n*}_0$ solves the Hele-Shaw model
\begin{equation} 
\left\{ \begin{split}-\Delta p^{n*}_0 =G, & \quad x \in \Omega_n, \\
p^{n*}_0=0,& \quad x \in \partial \Omega_n.
\end{split} \right.
\end{equation}
Due to the convex profile of $p_0^{n*}$ within $\Omega_n$, although $p_0^{n*}$ is continuous in the whole space, $\nabla p^{n*}$ may have jumps across  $\Omega_n$, which reflects the jumps in $u^{n*}$ across  $\Omega_n$.  

 To sum up, for $m\gg 1$ and fixed $\Delta t$, due to \eqref{eq:ustar0} from the previous time step, and \eqref{eq:ustar}, the velocity field of the semi-discrete numerical scheme is the same as the velocity field given by the limiting Hele-Shaw model. Then the front propagates according to the velocity field $\mathbf{u}^{n*}$ from \eqref{eq:rhonew}, which is the same as the free boundary model.

\subsection{Stability analysis}
When $m>1$, it is difficult to analyze the stability of the semi-discrete scheme due to nonlinearity. However, for $m=2$, the nonlinearity in the degenerate diffusion is only quadratic, and the introduction of the velocity $\mathbf u$ makes stability analysis of the augmented DAEs \eqref{eq:ustar}--\eqref{eq:ustar0} feasible. Similar observation has been  made in a recent work \cite{CYZ}.
For simplicity, we still assume $G$ is a constant, and the extensions to more general $G$ functions will be discussed later in this section.

\textbf{Heuristic argument for $m>1$:} In obtaining $\mathbf{u}^{n*}$ from equation  \eqref{eq:ustar}, the right hand side is treated semi-implicitly: only the velocity on the right hand side is taken implicitly while the rest of the terms are still taken explicitly. Therefore, the stability may be compromised for efficiency. Especially, since $m\nabla(\rho^n G)$ is explicit, if $G$ is independent of space, we can rewrite \eqref{eq:ustar} into
$$\f{\mathbf{u}^{n*}-\mathbf{u}^n}{\Delta t}=m\nabla \big((\rho^{n})^{m-2}\nabla \cdot  (\rho^n \mathbf{u}^{n*})\big)-(m-1)G\mathbf{u}^n.  
$$
{Therefore, a necessary stability requirement for the equation to update $\mathbf{u}^{n*}$ \eqref{eq:ustar} is that the required $\Delta t$ linearly depends on $1/(m-1)$. }
Besides, in solving for $\rho^{n+1}$ from equation \eqref{eq:rhonew} with $\mathbf{u}^{n*}$ and $c^{n+1}$ given, the convection term on the right hand side is treated explicitly. Therefore, the time step $\Delta t$ needs to satisfy the regular hyperbolic CFL constraint.

However, due to the high nonlinearity of the system, it is impossible for us to derive the sufficient stability condition, but we shall numerically verify that, with proper spatial discretization, the resulting scheme is uniformly accurate in $m$ and the stability condition depends on $1/(m-1)$ linearly.

\textbf{Rigorous stability analysis for $m=2$: }
When $m=2$, the semi-discretization \eqref{eq:ustar}-\eqref{eq:ustar0} reduces to:
\begin{align}
\f{\mathbf{u}^{n*}-\mathbf{u}^{n}}{\Delta t}&=2\nabla \Big(\big(\nabla \cdot  (\rho^n \mathbf{u}^{n*})-\rho^n G\big)\Big)\,, \label{eq:us1}\\
\f{\rho^{n+1}-\rho^n}{\Delta t}&=-\nabla \cdot (\rho^n \mathbf{u}^{n*})+\rho^{n+1} G\,, \label{eq:rho1}\\
\mathbf{u}^{n+1}&=-2\nabla \rho^{n+1}\,. \label{eq:u1}
\end{align}
 Clearly, \eqref{eq:us1} is equivalent to
\begin{equation} \label{eq:us2}
\f{\mathbf{u}^{n*}-\mathbf{u}^n}{\Delta t} =2\nabla \nabla \cdot (\rho^n \mathbf{u}^{n*}) + \mathbf{u}^{n} G,
\end{equation}
where we have used $u^{n}=-2\nabla \rho^{n}$. Take the spatial gradient of \eqref{eq:rho1} and multiply by $-2$, we get
\begin{equation} \label{eq:us3}
\f{\mathbf{u}^{n+1}-\mathbf{u}^n}{\Delta t} =2\nabla \nabla \cdot   (\rho^n \mathbf{u}^{n*}) + \mathbf{u}^{n+1} G.
\end{equation}
Subtracting \eqref{eq:us2} from \eqref{eq:us3}, we see
\begin{equation} \label{eq:us4}
\f{\mathbf{u}^{n+1}-\mathbf{u}^{n*}}{\Delta t} =(\mathbf{u}^{n+1}-\mathbf{u}^n) G,
\end{equation}
which implies
\begin{equation} \label{rel:us}
\mathbf{u}^{n*}=\mathbf{u}^{n+1}-\Delta t G (\mathbf{u}^{n+1}-\mathbf{u}^n).
\end{equation}
It is important to note that this relation tells that $\mathbf{u}^{n*}$ is indeed a good prediction of $\mathbf{u}^{n+1}$ since their discrepancy is of order $O(\Delta t^2)$ when $\mathbf{u}^{n+1}-\mathbf{u}^n$ is at order $\Delta t$. 

Now multiplying both sides of \eqref{eq:rho1} by $\rho^{n+1}\Delta t$, integrating in space, one gets
\begin{equation} \label{est1}
\frac{1}{2} \| \rho^{n+1}\|^2- \frac{1}{2} \| \rho^{n}\|^2+\frac{1}{2} \| \rho^{n+1}-\rho^n\|^2 =-\Delta t \int \nabla \cdot (\rho^n \mathbf{u}^{n*}) \rho^{n+1} dx +  \Delta t G \| \rho^{n+1}\|^2.
\end{equation}Here $\|\cdot\|$ represents the $L^2$ norm.
If we assume that the scheme is positivity preserving, i.e., $\rho^n \ge 0$, then with integration by parts and \eqref{eq:u1}, we learn that
\[
-\int \nabla \cdot (\rho^n \mathbf{u}^{n*}) \rho^{n+1} dx = - \frac{1}{2} \int \rho^n \mathbf{u}^{n*}\cdot \mathbf{u}^{n+1} dx .
\]
Therefore, from \eqref{rel:us}, we obtain
\begin{equation} 
- \frac{1}{2} \int \rho^n \mathbf{u}^{n*}\cdot \mathbf{u}^{n+1} dx = -\frac{1}{2} \int \rho^n | \mathbf{u}^{n+1}|^2 d x + O(\Delta t) .
\end{equation}
Hence, we conclude that
\begin{equation} \label{est2}
\left(\frac{1}{2}-G \Delta t \right) \| \rho^{n+1}\|^2 \le \frac{1}{2} \| \rho^{n}\|^2 +O(\Delta t^2).
\end{equation}
This implies the following stability estimate for $\frac{1}{2}-G \Delta t>0$, 
\[
 \| \rho^{n}\|^2 \le \bar C e^{2GT}  \| \rho^{0}\|^2, \quad n\Delta t <T\,,
\]
where $\bar C$ only depends on the final time $T$.

In conclusion, we have proved the following  stability estimate
\begin{theorem}
For some $T>0$ and $G_{\max}>0$, when $G_{\max} \Delta t < \frac 1 2$, the solutions to the semi-discrete scheme \eqref{eq:us1} \eqref{eq:rho1} \eqref{eq:u1}  satisfy
\begin{equation}
 \| \rho^{n}\|_{L^2} \le C e^{G_{\max}T}  \| \rho^{0}\|_{L^2}, \quad \quad \forall \, \, n\Delta t <T,
\end{equation}
where the constant $C$ only depends on $T$ and $G_{\max}$.
\end{theorem}
 Obviously, the proof also works for spatially inhomogeneous $G$ as long as there exists a constant $G_{\max}$ such that
\[
0 \le G \le G_{\max} .
\] Although the theorem above does not necessarily imply the stability constraint of a fully discrete scheme, it indicates that our time discretization respects the stability estimate of the  original density model \eqref{eq:n}.
 


\section{The fully discrete scheme}
In this section, we introduce the spatial discretization to obtain the fully discrete scheme. We first construct the discretization for \eqref{eq:ustar}--\eqref{eq:ustar0} in one spatial dimension, and then extend it to 2D. Higher dimensional extensions should be similar. We will focus our attention exclusively on the discretization of $\rho$ and $\mathbf u$ by assuming that $G$ is given, and discretization for $c$ is straightforward, see \cite{LWZ}.  We also look into the scheme's property on positivity and establish a formal energy estimate.

\subsection{The fully discrete scheme in 1D}

Let $[a,b]$ be the computational domain, $\Delta x=(b-a)/N_x$ be the mesh size, and the grid points be 
$$
x_i=a+i\Delta x,\quad x_{i+1/2}=a+(i+1/2)\Delta x,\qquad
i\in\{0,1,\cdots,N_x-1\}\,.
$$
We use staggered grid for $u$ and regular grid for $\rho$. More precisely, $\rho$ takes cell averages on $[x_{i-\frac{1}{2}},x_{i+\frac{1}{2}}]$ and $u$ takes value at $x_{i+1/2}$, i.e., 
$$ \rho_i(t)=\frac{1}{\Delta x}\int_{x_{i-1/2}}^{x_{i+1/2}}\rho(x,t)\,dx,
\qquad u_{i+1/2}(t)= u(x_{i+1/2},t)\,.
$$  

First of all, the space discretization for $u^{n\ast}$ in \eqref{eq:ustar} is via the centered finite difference method, i.e., 
\begin{eqnarray}\label{sch:u1D}
\f{u_{i+1/2}^{n*}-u_{i+1/2}^n}{\Delta t}
&=&\frac{m}{\Delta x}\left\{ \left((\rho^{n}_{i+1})^{m-2}\left(\frac {\rho_{i+3/2}^n u_{i+3/2}^{n*}-
\rho_{i+1/2}^n u_{i+1/2}^{n*}}{\Delta x}-\rho_{i+1}^n G^n_i \right)\right) \right.  \nonumber \\
&& \left. \qquad-\left((\rho^{n}_i)^{m-2}\left(\frac {\rho_{i+1/2}^n u_{i+1/2}^{n*}-
\rho_{i-1/2}^n u_{i-1/2}^{n*}}{\Delta x}-\rho_{i}^{n} G^n_i \right)\right) \right\}  \,,
\end{eqnarray}
where  $G^n_i\approx G(x_i,n\Delta t)$ and the half grid values of $\rho$ are taken as the average of their two neighboring cells, i.e.
\[
\rho^n_{i+1/2}= \frac{ \rho^n_i +  \rho^n_{i+1}}{2}.
\]


After obtaining $u^{n*}$, equation \eqref{eq:rhonew} is just a linear hyperbolic equation for $\rho$ with a growth term, and we use central scheme to discretize it as in \cite{BF}. More specifically, we have 
$$
\frac{\rho^{n+1}_i-\rho^n_i}{\Delta t}+\frac{F^n_{i+1/2}-F^n_{i-1/2}}{\Delta x}= \rho_i^{n+1}G^n_i\,,
$$
where the flux $F^n_{i\pm 1/2}$ is given by
\begin{equation}\label{eq:flux1d}
F^n_{i\pm1/2} = \frac{1}{2} \left[ \rho^{Ln} u^{n\ast } + \rho^{Rn} u^{n\ast } - |u^{n\ast }|(\rho^{Rn} - \rho^{Ln}) \right]_{i\pm1/2}\,,
\end{equation}
and $\rho_{i\pm1/2}^{L/Rn}$ are edge values constructed as follows. 

On the cell $[x_{i-1/2},x_{i+1/2}]$, let 
\begin{equation}\label{eqn:10}
\rho_i^n(x)\approx \rho^n_i+(\partial_x\rho)^n_i(x-x_i).
\end{equation} 
Then at the interface $x_{i+1/2}$, there are two approximations from the left and from the right, i.e., 
\begin{equation}\label{eqn:rhoLR}
\rho_{i+1/2}^{Ln}= \rho^n_i+\frac{\Delta x}{2}(\partial_x\rho)^n_i,\qquad
\rho_{i+1/2}^{Rn}= \rho^n_{i+1}-\frac{\Delta x}{2}(\partial_x\rho)^n_{i+1}.
\end{equation}
Here $(\partial_x\rho)_i$ is given by the minmod limiter:
\begin{equation} \label{eqn:slopelimiter}
(\partial_x\rho)^n_i=\left\{\begin{array}{cc}\min\{\frac{\rho^n_{i+1}-\rho^n_i}{\Delta x},\frac{\rho^n_{i+1}-\rho^n_{i-1}}{2\Delta x},\frac{\rho^n_{i}-\rho^n_{i-1}}{\Delta x}\},&\qquad\mbox{if all are positive,}\\
\max\{\frac{\rho^n_{i+1}-\rho^n_i}{\Delta x},\frac{\rho^n_{i+1}-\rho^n_{i-1}}{2\Delta x},\frac{\rho^n_{i}-\rho^n_{i-1}}{\Delta x}\},&\qquad\mbox{if all are negative},\\
0,&\qquad\mbox{otherwise.}\end{array}\right.
\end{equation}
The reconstruction of $\rho$ based on cell averages $\rho_j$ via \eqref{eqn:10} has the following property (see e.g. \cite{CCH}).
The minmod limiter is used so that the reconstruction is second-order accurate and positive preserving. Other limiter such as Van Leer limiter can also be used \cite{BF}.


In the correction step, we simply employ the centered difference approximation, namely
\begin{equation} \label{sch:u1D2}
u^{n+1}_{i+1/2} =  - \frac{m}{m-1}  \frac{ ( \rho^{n+1}_{i+1})^{m-1} -( \rho^{n+1}_{i})^{m-1} }{\Delta x}.
\end{equation}
Note that at the propagating front, the accuracy of this approximation may degrade. However, it ensures the consistency condition between $\rho$ and $u$. The fully discrete scheme is first order accurate in time and second order accurate in space. We observe in numerical tests that, the convergence order is $1$ when the hyperbolic CFL condition is satisfied, and the convergence order becomes 2 if the parabolic CFL condition is used instead. In the future, we may investigate a second order discretization in time.



\subsection{The fully discrete scheme in 2D}
In 2D, the velocity has two components $\mathbf{u} = (u,v)$, with $u$ and $v$ being the velocities along $x$ and $y$ directions, respectively. Then \eqref{eq:ustar} becomes  
\begin{eqnarray*}
u_t = m \partial_x \left[ \rho^{m-2}  \left(  (\rho u)_x + (\rho v)_y - \rho G   \right)  \right]\,, \nonumber
\\ 
v_t = m \partial_y \left[  \rho^{m-2} \left(  (\rho u)_x + (\rho v)_y - \rho G   \right)   \right] \,, 
\end{eqnarray*}
and the discretization is essentially the same as in 1D except a slight difference in choosing half grid points or grid points. We assume that the computational domain is $(x,y)\in[a,b]\times[a,b]$, similar to the grids in $x$-direction, denote  
\[
y_j=a+j\Delta y,\quad y_{j+1/2}=a+(j+1/2)\Delta y,\qquad \Delta y = \frac{b-a}{N_y}, \qquad
j\in\{0,1,\cdots,N_y-1\}\,.
\] 
At variance with the 1D case, we compute both $\rho$ and $(u,v)$ on regular grids, and specify the half grid values if needed. To be concrete, 
let
\[
 \rho_{ij}(t)=\frac{1}{\Delta x\Delta y}\int_{x_{i-1/2}}^{x_{i+1/2}} \int_{y_{i-1/2}}^{y_{i+1/2}}   \rho(x, y, t)\,dx dy,
\qquad u_{ij}(t) \approx u(x_{i}, y_j, t)\, \qquad v_{ij}(t) \approx v(x_{i}, y_j, t) \,,
\]
then the discretization for $u$ takes the following form
\begin{eqnarray} \label{sch:u2D}
\frac{u_{ij}^{n*} - u_{ij}^n }{\Delta t} &=&  \frac{m}{\Delta x} 
\left\{    (\rho^n_{i+1/2,j})^{m-2} \frac{ (\rho^n u^{n*})_{i+1,j} - (\rho^n u^{n*})_{ij}  }{\Delta x}  -   
            (\rho^n_{i-1/2,j})^{m-2}\frac{ (\rho^n u^{n*})_{i j} - (\rho^n u^{n*})_{i-1,j}  }{\Delta x}   \right. \nonumber 
 \\ && \left.  +\frac{1}{2}( \rho_{i+1,j}^n)^{m-2}  \frac{ (\rho^n v^{n*})_{i+1,j+1} - (\rho^n v^{n*})_{i+1,j-1} }{2\Delta y} 
    -  \frac{1}{2} (\rho_{i-1,j}^n)^{m-2}  \frac{ (\rho^n v^{n*})_{i-1,j+1} - (\rho^n v^{n*})_{i-1,j-1} }{2\Delta y}      \right.    \nonumber 
 \\&&  \left.  - \frac{[(\rho^n)^{m-1} G^{n+1}]_{i+1,j} - [(\rho^n)^{m-1} G^{n+1}]_{i-1,j }}{2}  \right\}  \,.
\end{eqnarray}
Likewise, the discretization for $v$ reads
\begin{eqnarray}\label{sch:v2D}
\frac{v_{ij}^{n*} - v_{ij}^n }{\Delta t} &=&  \frac{m}{\Delta y} 
\left\{    (\rho^n_{i,j+1/2})^{m-2} \frac{ (\rho^n v^{n*})_{i,j+1} - (\rho^n v^{n*})_{ij}  }{\Delta y}  -   
            (\rho^n_{i,j-1/2})^{m-2}\frac{ (\rho^n v^{n*})_{i j} - (\rho^n v^{n*})_{i,j-1}  }{\Delta y}   \right. \nonumber 
 \\ && \left.  +\frac{1}{2}( \rho_{i,j+1}^n)^{m-2}  \frac{ (\rho^n u^{n*})_{i+1,j+1} - (\rho^n u^{n*})_{i-1,j+1} }{2\Delta x} 
    - \frac{1}{2}  (\rho_{i,j-1}^n)^{m-2}  \frac{ (\rho^n u^{n*})_{i+1,j-1} - (\rho^n u^{n*})_{i-1,j-1} }{2\Delta x}      \right.    \nonumber 
 \\&&  \left.  - \frac{[(\rho^n)^{m-1} G^{n+1}]_{i,j+1} - [(\rho^n)^{m-1} G^{n+1}]_{i,j-1 }}{2}  \right\}  \,.
\end{eqnarray}
Here again the half grid value of $\rho$ is taken as the average of grid values:
\[
\rho_{i+1/2,j}=\frac{1}{2} (\rho_{ij}+\rho_{i+1,j}), \quad \rho_{i,j+1/2}=\frac{1}{2} (\rho_{ij}+\rho_{i,j+1})\,.
\]
Putting \eqref{sch:u2D} and \eqref{sch:v2D} together, we end up with a linear system for $u^*$ and $v^*$ which can be solved using GMRES. 

Upon getting $u^{n*}$ and $v^{n*}$, $\rho^{n+1}$ can be obtained by virtue of the central scheme again:
\begin{equation} \label{sch:rho2D}
\frac{\rho^{n+1}_{ij} - \rho^n_{ij}}{\Delta t} + \frac{1}{\Delta x} \left[ (F_1)^n_{i+1/2,j} - (F_1)^n_{i-1/2,j}\right] 
+ \frac{1}{\Delta y} \left[ (F_2)^n_{i,j+1/2} - (F_2)^n_{i,j-1/2}\right]  = (\rho G)^{n+1}_{ij} \,,
\end{equation}
where 
\begin{eqnarray*}
&& (F_1)^n_{i\pm 1/2, j} = \frac{1}{2} \left[ \rho^{L_xn} u^{n*} + \rho^{R_xn} u^{n*} - |u^{n*}| (\rho^{R_xn} - \rho^{L_xn})\right]_{i\pm 1/2, j} \,,
\\ && (F_2)^n_{i, j\pm 1/2} = \frac{1}{2} \left[ \rho^{L_yn} u^{n*} + \rho^{R_yn} u^{n*} - |u^{n*}| (\rho^{R_yn} - \rho^{L_yn})\right]_{i, j \pm 1/2}\,.
\end{eqnarray*}

Here $\rho^{L_x}$, $\rho^{R_x}$, $\rho^{L_y}$, $\rho^{R_y}$ are half grid values obtained by linear reconstruction with a slope limiter:  
\begin{eqnarray*}
&& \rho_{i+1/2,j}^{L_xn}= \rho_{ij}^n+\frac{\Delta x}{2}(\partial_x\rho)^n_{ij},\qquad
\rho_{i+1/2,j}^{R_xn}= \rho^n_{i+1,j}-\frac{\Delta x}{2}(\partial_x\rho)^n_{i+1,j} \,, \nonumber
\\ && \rho_{i,j+1/2}^{L_yn}= \rho^n_{ij}+\frac{\Delta y}{2}(\partial_y\rho)^n_{ij},\qquad
\rho_{i,j+1/2}^{R_yn}= \rho^n_{i,j+1}-\frac{\Delta y}{2}(\partial_y\rho)^n_{i,j+1} \,,
\end{eqnarray*}
with the slope determined by \eqref{eqn:slopelimiter}. 

At last, the correction step takes the form
\begin{equation} \label{sch:u2D2}
u^{n+1}_{i,j} =  - \frac{m}{m-1}  \frac{ ( \rho^{n+1}_{i+1,j})^{m-1} -( \rho^{n+1}_{i-1,j})^{m-1} }{2\Delta x} \,, \qquad
v^{n+1}_{i,j} =  - \frac{m}{m-1}  \frac{ ( \rho^{n+1}_{i,j+1})^{m-1} -( \rho^{n+1}_{i,j-1})^{m-1} }{2\Delta y}\,.
\end{equation}

\subsection{Properties}
Since we have adopted a classical staggered grid scheme for  the spatial discretization, which is similar to a recent work \cite{BF}, we expect our fully discrete scheme to have similar properties. Nevertheless, it is worth pointing out, unlike the gradient flow model in \cite{CCH} or nonlinear degenerate parabolic equations in \cite{BF}, the tumor growth model is not associated with a decaying free energy/relative entropy. 

\subsubsection{Positivity preserving} Positivity preserving is desired for simulating tumor-growth models since lack of such property may result in nonphysical oscillations at the moving support of the cell density. Thanks to the nonnegative reconstruction by the minmond limiter, we can prove the positivity preserving property of the proposed scheme.  For simplicity, we make the following technical assumption, for a given $m$ and finite time $t$, there exists a uniform $U \in \mathbb R$ for all time steps, such that
\[
\max_{n,j} \{ u^{n*}_{j+1/2}\} \le U,
\]
and $G$ is nonnegative with an upper bound, namely
\[
G(c) \le G_{\max}.
\]
We state the time step constraints for the positivity preserving property in the following, whose proof roughly follows from that for Theorem 2.3 in \cite{CCH}.

\begin{theorem}
Consider the the fully discrete scheme \eqref{sch:u1D}--\eqref{sch:u1D2} to the cell density model with initial data $\rho_0 (x) \ge 0$. Then, the cell averages $\rho^n_i \ge 0$, $\forall n \in \mathbb N^+$ and $\forall i$, if the following CFL condition is satisfied
\begin{equation}\label{cond:dt1}
\Delta t  \le \frac{\Delta x}{2 U}, \quad \text{and} \quad 1-G_{\max} \Delta t >0.
\end{equation}
\end{theorem}
\begin{proof}
Assume that at time level $t_n=n\Delta t$, the cell averages are nonnegative: $\rho^n_j \ge 0$. After the prediction step for the velocity $u^{n*}$, we update the cell averages via the following formula
\begin{equation}\label{eq:up1}
(1-\Delta t \, G^n_i)  \rho^{n+1}_i =   \rho^{n}_i - \lambda \left[  F^n_{i+1/2}-F^n_{i-1/2} \right],
\end{equation}
where $\lambda = \Delta t/\Delta x$. From the definition of $\rho^{Ln}_{i+1/2}$ and $\rho_{i-1/2}^{Rn}$ in \eqref{eqn:rhoLR}, $\rho_i^n=(\rho^{Ln}_{i+1/2}+\rho_{i-1/2}^{Rn})/2$.
 Then, by substituting the definition of the numerical flux $F_{i\pm 1/2}^n$ in \eqref{eq:flux1d} into \eqref{eq:up1}, we get
\begin{align}
(1-\Delta t \, G^n_i)  \rho^{n+1}_i = &
 \frac{1-\lambda u^{n*}_{i+1/2}-\lambda| u^{n*}_{i+1/2}|}{2}\rho^{Ln}_{i+1/2}  -  \frac{ \lambda u^{n*}_{i+1/2}-  \lambda |u^{n*}_{i+1/2}|}{2}\rho^{Rn}_{i+1/2} \nonumber \\
 &+ \frac{\lambda u^{n*}_{i-1/2}+\lambda| u^{n*}_{i-1/2}|}{2}  \rho^{Ln}_{i-1/2} +\frac{1+\lambda u^{n*}_{i-1/2}-  \lambda |u^{n*}_{i-1/2}| }{2} \rho^{Rn}_{i-1/2}.\label{eq:up2}
\end{align}
It is obvious that the coefficients of $\rho^{Rn}_{i+1/2}$ and $\rho^{Ln}_{i-1/2}$ are always nonnegative, $\forall j$. And when condition  \eqref{cond:dt1} is satisfied, the coefficients of $\rho^{Ln}_{i+1/2}$ and $\rho^{Rn}_{i-1/2}$ are  also nonnegative. Hence, with condition \eqref{cond:dt1}, $\rho^{n+1}_i $  is  clearly  a linear combination of the nonnegative reconstructed  values at the cell boundaries, $\rho^{Ln}_{i+1/2}$, $\rho^{Ln}_{i-1/2}$, $\rho^{Rn}_{i+1/2}$  and $\rho^{Rn}_{i-1/2}$. Therefore, we conclude that given  \eqref{cond:dt1}, $\rho^{n+1}_i $ is nonnegative. The theorem thus follows by induction.
\end{proof}

%
%
%

%
\subsubsection{Energy estimate}
We show in the following that the spatial discretization is superior in the sense that it respects the energy estimate of the original equation. For simplicity, we assume $G$ is nonnegative and has an upper bound denoted by $G_{\max}$.

Let us look at the continuous setting first. Denote $H(\rho)= \frac{1}{m-1} \rho^m$ the density of the internal energy, it satisfies
\[
H'(\rho)=p(\rho).
\] 
Then the cell density model \eqref{eq:n} is equivalent to
\begin{equation}\label{eq:nH}
\frac{\partial}{\partial t} \rho - \nabla \cdot \left( \rho \nabla H'(\rho) \right)=\rho G(c).
\end{equation}
Then we define the free energy of equation \eqref{eq:nH} by \[
E(\rho) = \int_{\mathbb R^d} H(\rho) dx.
\]
Clearly, the free energy only consists of the internal energy and may grow in time due to the growth factor. For classical solutions, we have
\[
\frac{d}{dt}E(\rho) = - \int_{\mathbb R^d} \rho |u|^2 dx+\int_{\mathbb R^d} H'(\rho) \rho G dx = -I(\rho)+ mG E(\rho),
\]
where $I(\rho)=\int_{\mathbb R^d} \rho |u|^2 dx$ is the dissipation function and we have used the fact that
\[
H'(\rho) \rho = m H(\rho).
\]
Hence, we arrive at the following energy estimate 
\begin{equation}
E(\rho) (t) \le e^{mGt} E(\rho) (0).
\end{equation}

Next we turn to the semi-discrete (continuous in time) form of our scheme
\begin{equation}
\left\{
\begin{split}
\frac{d\rho_i}{dt}+\frac{F_{i+1/2}(t)-F_{i-1/2}(t)}{\Delta x}=\rho_iG(c_i),\\
u_{i+1/2} =  - \frac{m}{m-1}  \frac{ ( \rho_{i+1})^m -( \rho_{i})^m }{\Delta x}.
\end{split}
\right.
\end{equation}
Here when time variable is continuous, the prediction-correction scheme for the velocity $u$ reduces to an algebraic equation. 
The semi-discrete Energy takes the following form,
\[
E_\Delta (t) = \sum_{j} \Delta x H( \rho_j).
\]
By direct calculation and summation by parts, the time derivative of the semi-discrete energy is given by
\begin{align*}
\frac{d}{dt}E_\Delta  & =- \sum_j H'( \rho_j) (F_{i+1/2}(t)-F_{i-1/2}(t) ) + \sum_j \Delta x  H'( \rho_j) \rho_j G(c_j) \\
&= \sum_j (p( \rho_{j+1}) - p( \rho_{j}) )F_{i+1/2}+ \sum_j \Delta x  H( \rho_j)  m G(c_j) \\
& =-\Delta x \sum_j u_{i+1/2} F_{i+1/2} +\sum_j \Delta x  H( \rho_j)  m G(c_j).
\end{align*}
Note that we have the following estimates,
\begin{align*}
  -\Delta x \sum_j u_{i+1/2} F_{i+1/2} &= -\Delta x \sum_j u_{i+1/2} \left[ \frac{u_{i+1/2}+|u_{i+1/2}|}{2}\rho_{i+1/2}^L+ \frac{u_{i+1/2}-|u_{i+1/2}|}{2}\rho_{i+1/2}^R\right] \\
&\le \Delta x \sum_j (u_{i+1/2})^2 \min\{\rho_{i+1/2}^L,\rho_{i+1/2}^R\} \le 0,
\end{align*}
and
\[
\sum_j \Delta x  H( \rho_j)  m G(c_j) \le m G_{\max} E_\Delta 
\]
Therefore, we reach the analogue energy bound for the semi-discrete scheme
\begin{equation}
E_\Delta (t) \le e^{mG_{\max}t} E_\Delta (0).
\end{equation}

\section{Numerical examples} \label{sec:num}
In this section, we perform several numerical tests using our new schemes \eqref{sch:u1D}---\eqref{sch:u1D2}, \eqref{sch:u2D}--\eqref{sch:u2D2} or  their variations (will be specified below when needed) to justify its performance especially in capturing the large $m$ limit. Here most of the examples are picked from \cite{BF,Pbook,LTWZ}. 

We first elaborate the nutrient models in the following:
\begin{equation} \label{eqn:c0}
\tau \frac{\partial}{\partial t}  c - \Delta c + \Psi (\rho,c)=0,
\end{equation} 
where $\Psi(\rho,c)$ is the consumption function which takes different forms and $\tau$ is a time scaling constant. When $\tau=0$, the nutrient distribution $c(x,t) $ effectively adjusts to its local equilibrium according to the cell density $\rho(x,t)$.  
As in \cite{PTV}, we consider two specific models: the \emph{in vitro} model and the \emph{in vivo} model. For the \emph{in vitro} model, one assumes that the nutrient is constant outside the tumoral region; while the consumption is linear in $c$ inside, thus equation \eqref{eqn:c0} reads
\begin{align} 
-\Delta c + \psi(n)c=0, & \quad \mbox{for} \, x\,\in\, D;  \label{eqn:c0-invitro} \\
c=c_B, & \quad   \mbox{for} \, x\,\in\,\mathbb{R} \backslash D, \label{eqn:c1-invitro}
\end{align}
where 
\begin{equation} \label{eqn:D000}
D = \{n(x)>0 \} = \{ p(n)>0 \}\,.
\end{equation}
Here $\psi(n)$ satisfies
\begin{equation} \label{eqn:psi000}
\psi(n)\geq 0  \quad \text {for } n\geq 0, \qquad \text{and} \qquad  \psi(0)=0\,.
\end{equation}
For the \emph{in vivo} model, the nutrient is brought by the vasculature network away from the tumor and diffused to the tissue. In this case, equation \eqref{eqn:c0} writes
\begin{equation}\label{eq:c}
- \Delta c + \psi(n)c=\chi_{\{ n=0\}} (c_B-c)d \,,
\end{equation}
where $\psi$ is the same as in (\ref{eqn:psi000}). We remark that, the authors proved in \cite{LTWZ} that, both models satisfy the boundedness estimate \eqref{est:cbound}.

\subsection{Convergence test}
In the first example, we consider the 1D porous media equation 
\begin{equation}
\partial_t \rho = \partial_x \rho^{m}\,, \quad m>1\,. 
\end{equation}
with a computational domain $[-5,5]$. The initial condition takes the Barenblatt form 
\begin{equation} \label{eqn:Barenblatt}
\rho(x,0) = \frac{1}{t_0^\alpha} \left( C - \alpha \frac{m-1}{2m} \frac{x^2}{t_0^{2\alpha}} \right)_{+}^{\frac{1}{m-1}}
\end{equation}
with $t_0 = 0.01$ and $\alpha = \frac{1}{m+1}$. $C$ is a constant to be defined for different tests. Then the solution $\rho(x,t)$ remains the shape of \eqref{eqn:Barenblatt} with $t$ in place of $t_0$.  To check convergence, we calculate the errors compared to the analytical solution with decreasing $\Delta x = 1/16, 1/32, 1/64, 1/128, 1/256$:
\begin{equation} \label{error}
error = \sum_{j=1}^{N_x} \sum_{n=1}^{N_t} |\rho_j^n - \rho(x_j, n\Delta t)| \Delta x \Delta t \,.
\end{equation}
We first choose $\Delta = 0.01 \Delta x$. In Fig.~\ref{fig:porous} on the left, we display both the numerical solution and the analytical solution, wherein good agreement is seen. On the right, we plot the error \eqref{error} versus $\Delta x$ and observe a first order accuracy for varying $m$. Here for $m=3$, we choose $C=1$, and $m=15, ~60$, $C=0.1$ just to make sure the solution support is within our computational domain. In Fig.~\ref{fig:porous2}, we use $\Delta t= \Delta^2$. Then for small $m=3$, a second order accuracy can be observed (left) whereas for larger $m$, since the boundary becomes shaper and accuracy degrades. To show the boundary effect, we compute the error in time (right), i.e.,
\[
error(k\Delta t) =   \sum_{j=1}^{N_x}  |\rho_j^k - \rho(x_j, k\Delta t)| \Delta x 
\]
and the oscillation in this error demonstrates the effect of sharp boundary. 
\begin{figure}[!ht]
\centering
\includegraphics[width = 0.48\textwidth]{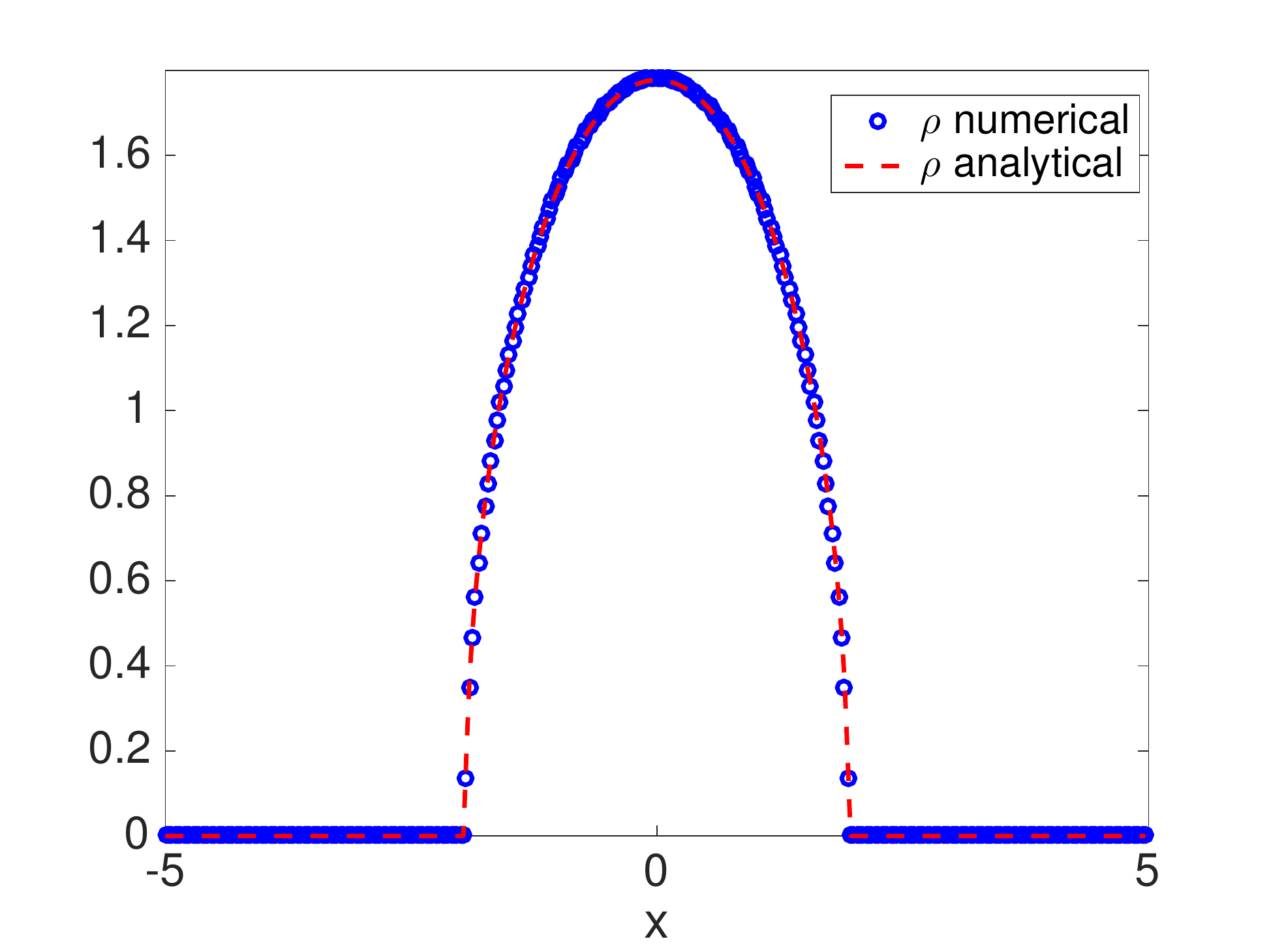}
\includegraphics[width = 0.48\textwidth]{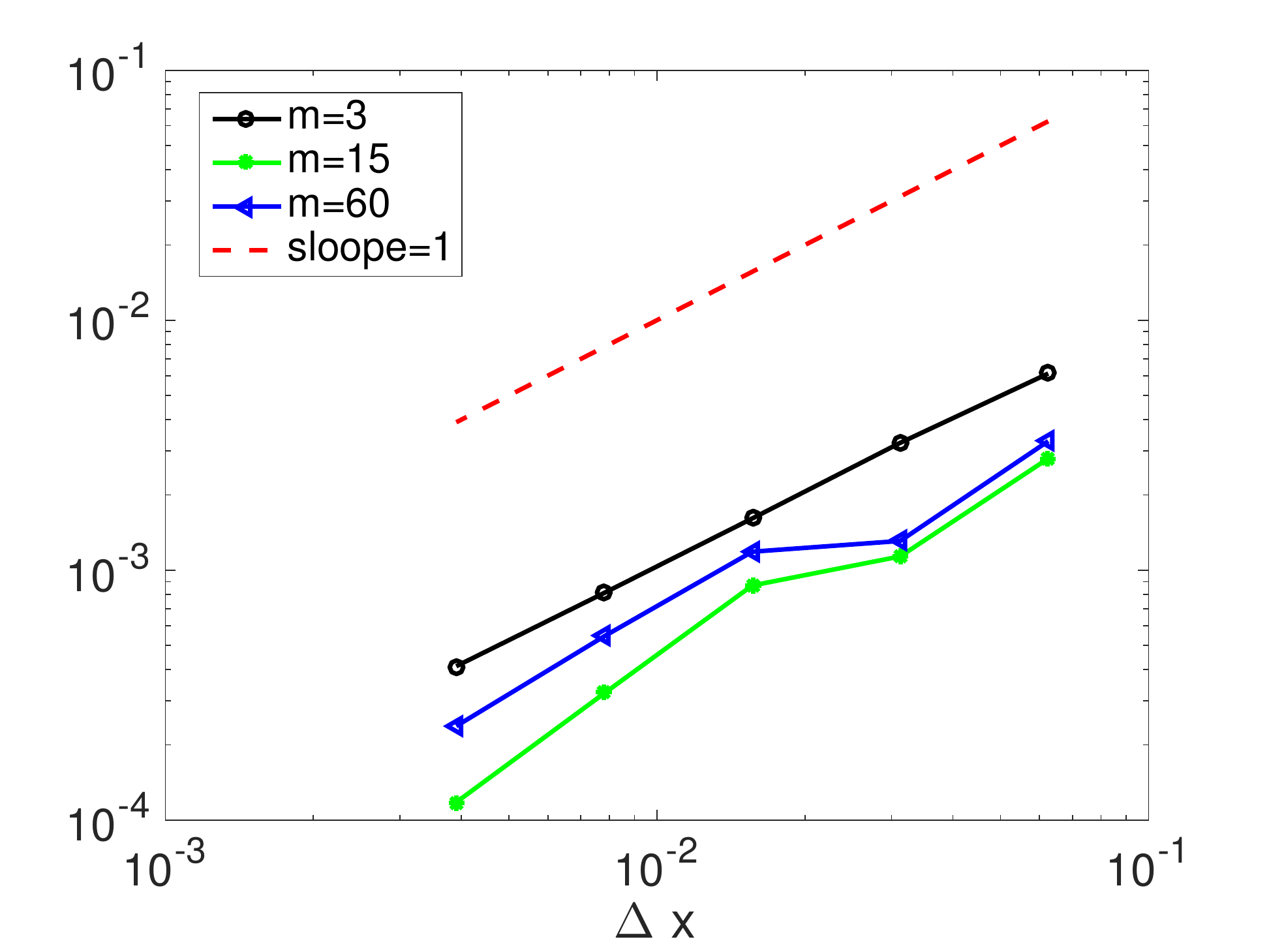}
\caption{ 1D porous media equation using $\Delta t = 0.01\Delta x$. The computed time is $T=0.1$. Left: plot of solution with $m=3$ and $N_x=320$. Right: Plot of error between analytical solution and numerical solution versus $\Delta x$. }
\label{fig:porous}
\end{figure}

\begin{figure}[!ht]
\centering
\includegraphics[width = 0.48\textwidth]{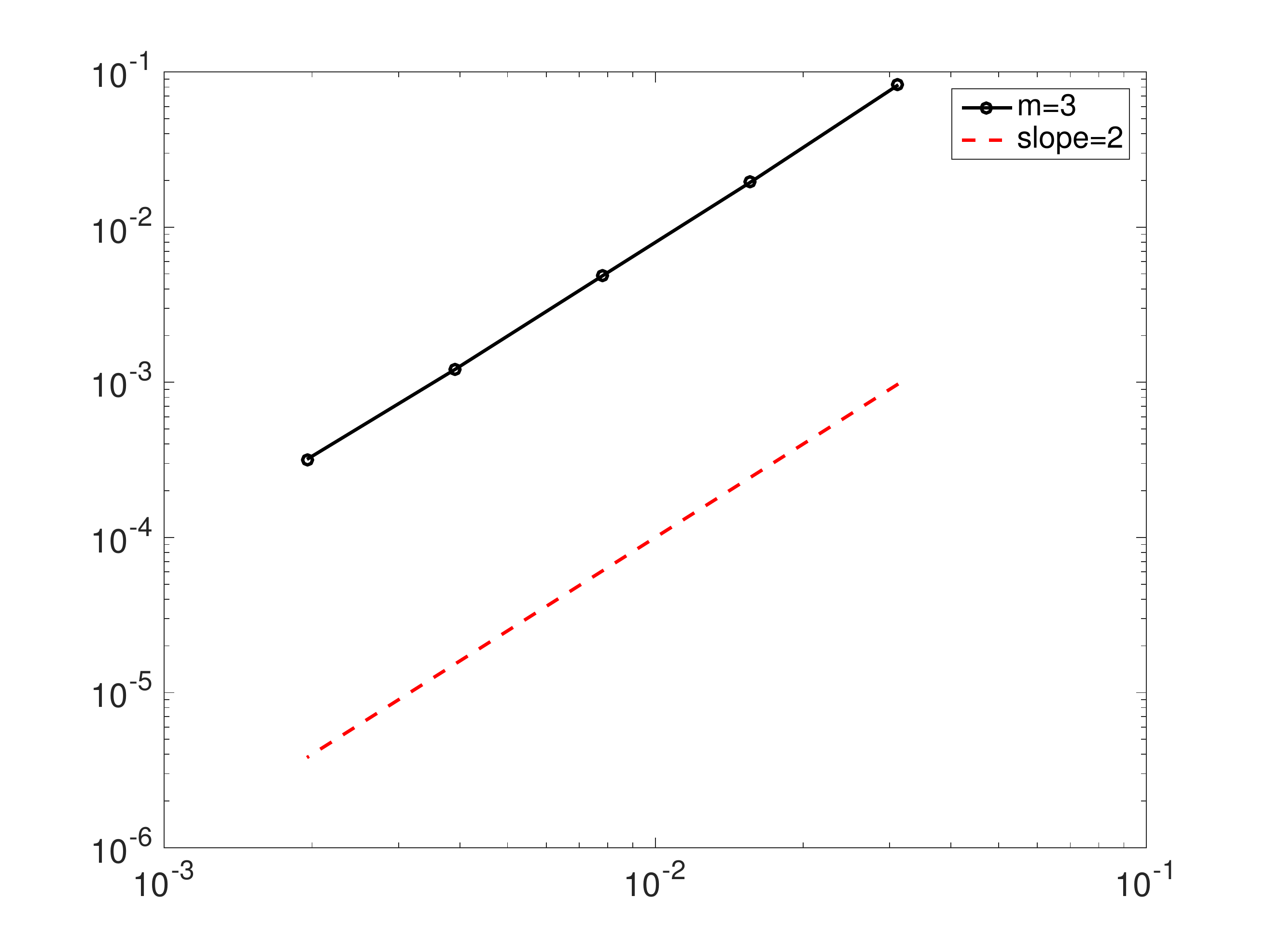}
\includegraphics[width = 0.48\textwidth]{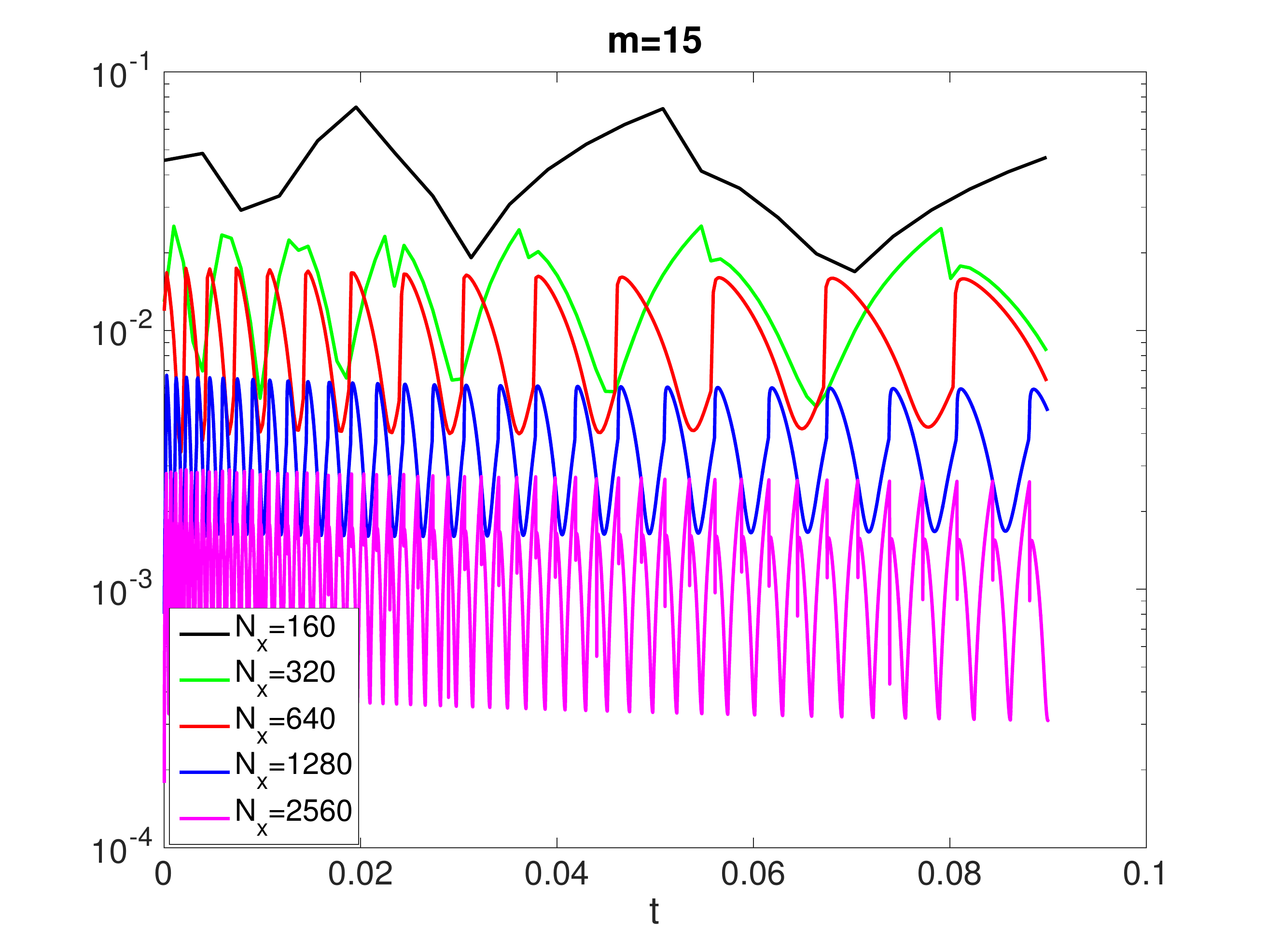}
\caption{ 1D porous media equation using $\Delta t = \Delta x^2$. Left: $m=3$. Plot of error versus $\Delta x$. Initial condition is \eqref{eqn:Barenblatt} with $t_0=0.01$ and $C=1$. Right: $m=15$. Plot of error along time. Initial condition is \eqref{eqn:Barenblatt} with $t_0=0.01$ and $C=0.1$.}
\label{fig:porous2}
\end{figure}

\subsection{1D model with linear growth}
Our second test is devoted to 1D model with linear growth, i.e., $G(c) = c$. The specific model we are computing reads
\begin{eqnarray}
&&\frac{\partial}{\partial t} \rho - \partial_x  \left( \rho \partial_x p(\rho) \right)=\rho c, \qquad p(\rho) = \frac{m}{m-1} \rho^{m-1}\,, \label{eqn:1Drho}
\\ && -\partial_{xx} c + \Psi(\rho,c) = 0\,,  \label{eqn:1Dc0}
\end{eqnarray}
where the nutrient \eqref{eqn:1Dc0} takes the form 
\begin{equation} \label{eqn:vitro1D}
\left\{ \begin{array}{cc} - \partial_{xx}c + \rho c = 0  & x\in D=\{\rho>0\} \\ c = 1 & {\text otherwise} \end{array} \right. 
\end{equation}
in the {\it in vitro} model and
\begin{equation} \label{eqn:vivo1D}
- \partial_{xx} c + \rho c = \chi_{\rho=0} (1-c)
\end{equation}
in the {\it in vivo} model. 

If initially we consider $\rho(x,0)$ as a characteristic function with support $[-R(0), R(0)]$, then sending $m$ to infinity, the cell density $\rho(x,t)$ remains a characteristic function with support $[-R(t), R(t)]$, i.e., $\rho_\infty = \chi_{[-R(t), R(t)]}$. In the {\it in vitro} model $R^\vitro(t)$ evolves as $\partial_t R^\vitro = \text{tanh}(R^\vitro)$, 
and the limiting pressure $p_\infty^\vitro(x,t)$ takes the form
\begin{equation} \label{eqn:pinfvitro}
p_\infty^{\it vitro} = \left\{  \begin{array}{cc} - \text{cosh}(x)/\text{cosh}(R(t)) + 1, & x \in[-R^\vitro(t), R^\vitro(t)]; \\ 0 & \text{ otherwise}, \end{array} \right. 
\end{equation}
In the {\it in vivo} model, $R^{\vivo}(t)$ obeys $\partial_t R^\vivo = \text{sinh}(R^\vivo) / e^{R^\vivo}$
and the limiting pressure $p_\infty^\vivo(x,t)$ is 
\begin{equation} \label{eqn:pinfvivo}
p_\infty^\vivo = \left\{ \begin{array}{cc}  [\cosh(R^\vivo(t))-\text{cosh}(x)]/e^{R^\vivo(t)},
& x\in[-R^\vivo(t), R^\vivo(t)]; \\ 0 & \text{otherwise}. \end{array} \right. 
\end{equation}
Details of the derivation of the above formulas can be found in \cite{LTWZ}.

Here we use computational domain $[-5,5]$ and choose the initial data to be
\begin{equation}
\rho(x,0) = \left( \frac{m-1}{m} p_\infty(x,0) \right) ^{\frac{1}{m-1}}\,, \quad R(0) = 1\,,
\end{equation} 
where $p_\infty$ takes either \eqref{eqn:pinfvitro} or \eqref{eqn:pinfvivo}. We compare our numerical solution with the above analytical solution with $m = 80$ and the results are collected in Fig.~\ref{fig:1Dvitro} and Fig.~\ref{fig:1Dvivo} with $\Delta x = 0.025$ and $\Delta t = 0.00125$. For both the density $\rho$ and pressure $P$, our solution matches well with the analytical solution of the Hele-Shaw model, indicating that our scheme performs well in capturing the free boundary limit. 
\begin{figure}[!ht]
\centering
\includegraphics[width = 0.48\textwidth]{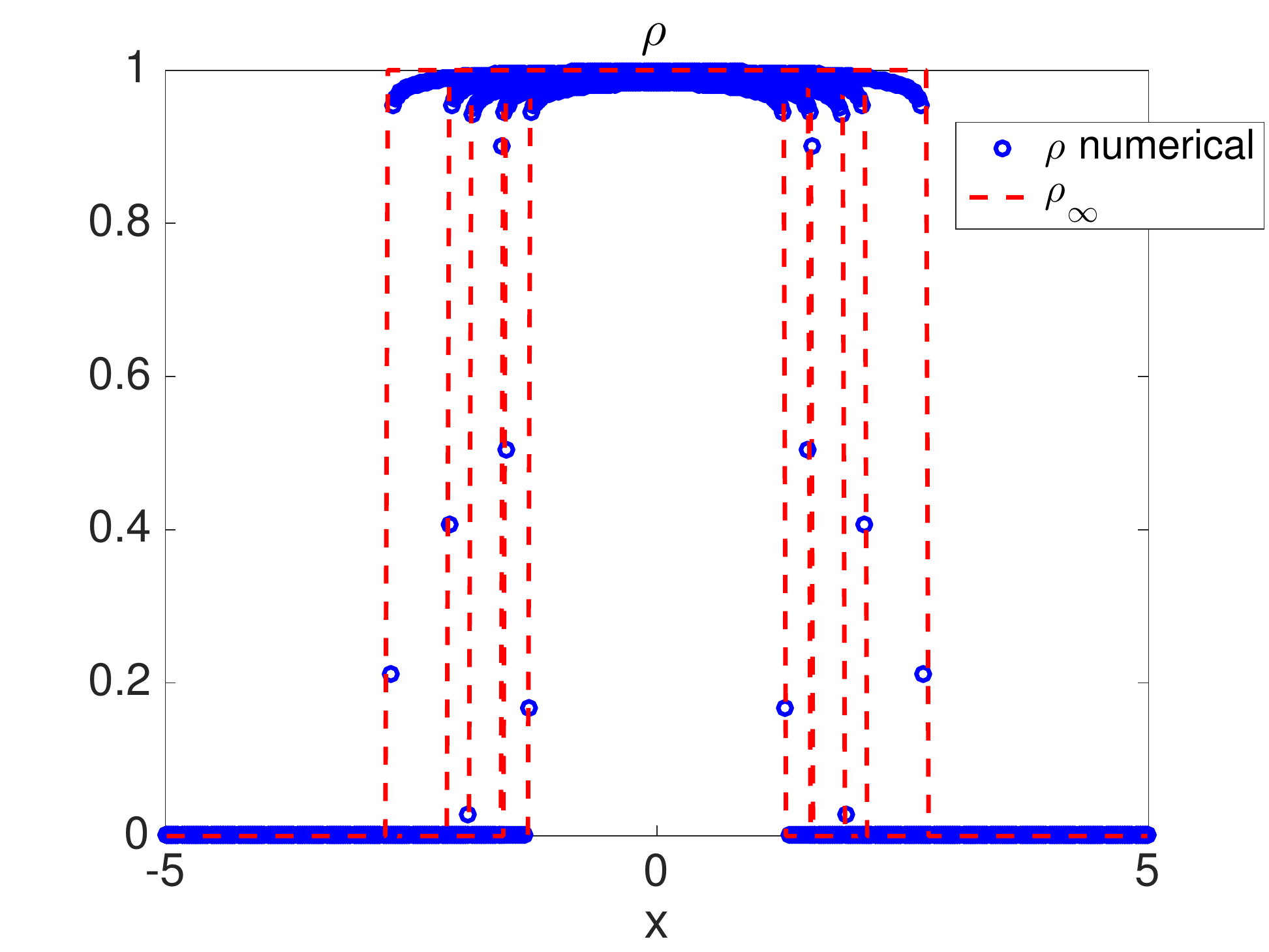}
\includegraphics[width = 0.48\textwidth]{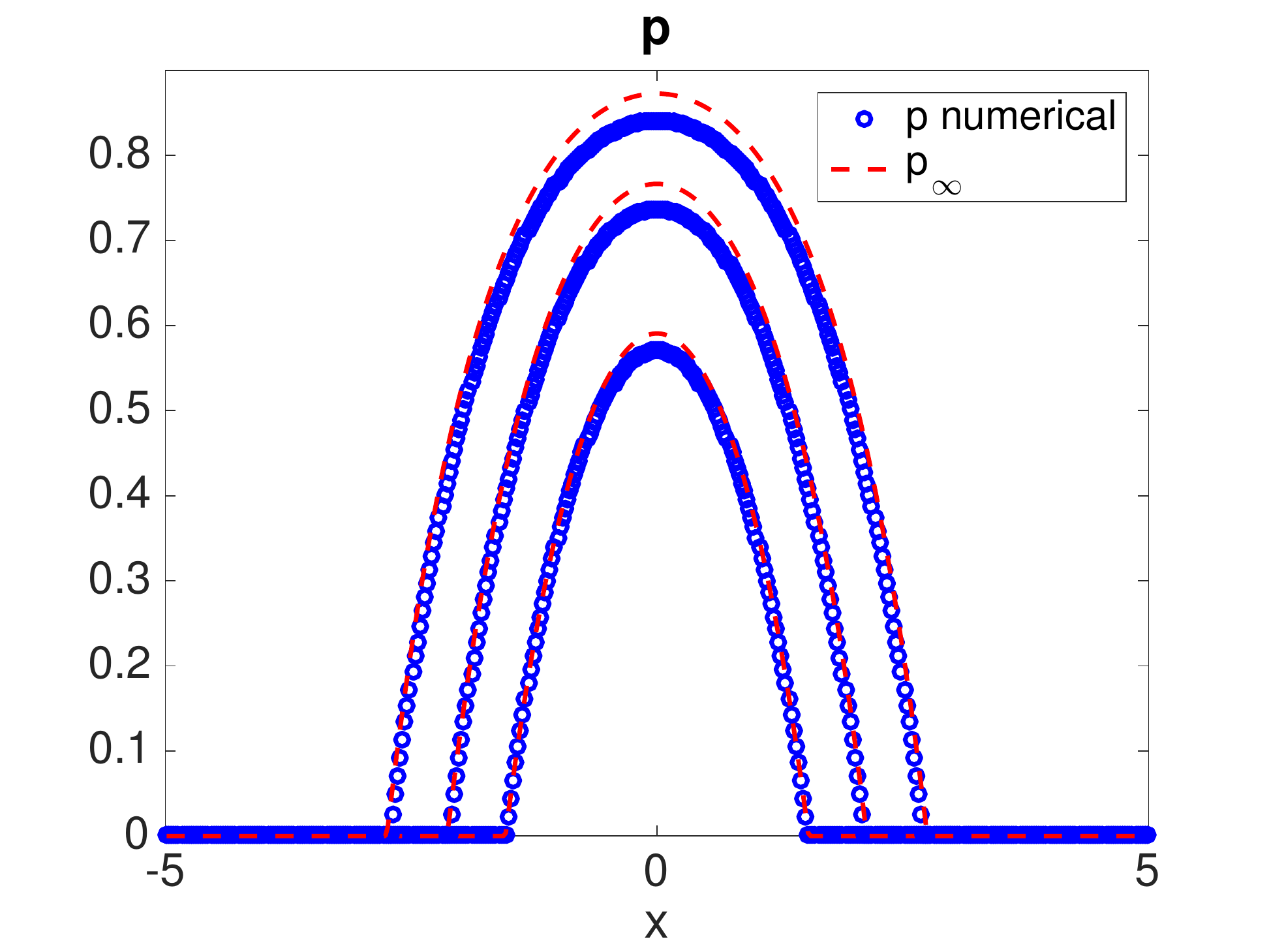}
\caption{1D {\it in vitro} model with $m=80$. We compare the numerical solution (blue circle) and analytical solution (red dashed curve) at different times $t=0.6237$, $t=1.2487$ and $t=1.8737$. The left figure is for the cell density $\rho$ and the right one is for pressure $p(\rho)$. Here we use $\Delta x = 0.025$ and $\Delta t = 0.00125$. }
\label{fig:1Dvitro}
\end{figure}

\begin{figure}[!ht]
\centering
\includegraphics[width = 0.48\textwidth]{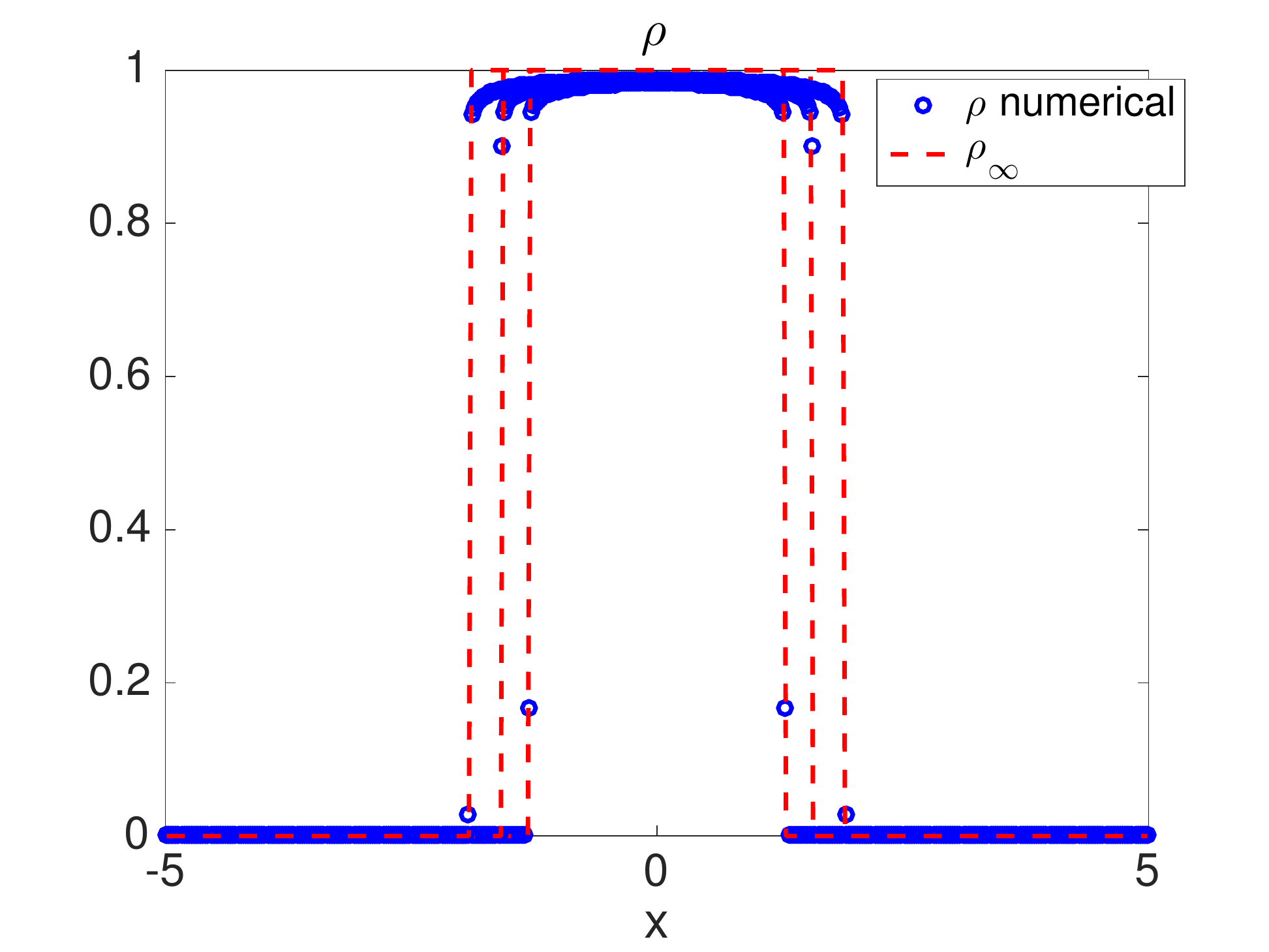}
\includegraphics[width = 0.48\textwidth]{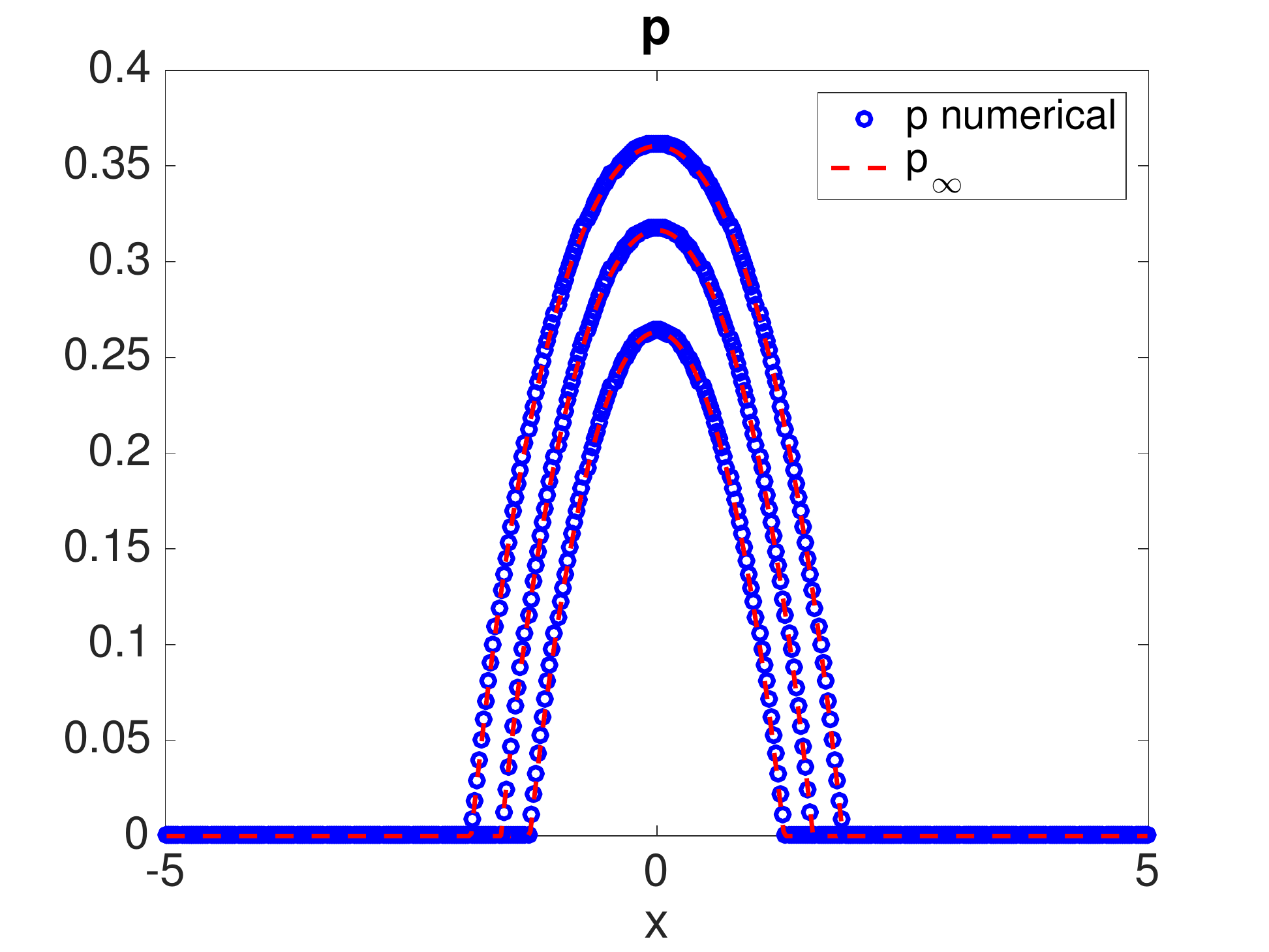}
\caption{1D {\it in vivo} model with $m=80$. We compare the numerical solution (blue circle) and analytical solution (red dashed curve) at different times $t=0.6237$, $t=1.2487$ and $t=1.8737$. The left figure is for the cell density $\rho$ and the right one is for pressure $p(\rho)$. Here we use $\Delta x = 0.025$ and $\Delta t = 0.00125$. }
\label{fig:1Dvivo}
\end{figure}

\subsection{2D radial symmetric model with constant nutrient}
Thirdly, we test the 2D radial symmetric case with constant nutrient $c \equiv1$ and $G(c)\equiv 1$. Then our new form of system \eqref{eqn:dae3} rewrites
\begin{equation*}
\left\{\begin{aligned}
&\p_t \rho+ \frac{1}{r}\frac{\partial}{\partial r}( r\rho  u )=\rho ,\\
&\p_t u=m   \frac{\partial}{\partial r} \left[\rho^{m-2}\left(\frac{1}{r} \frac{\partial}{\partial r} (\rho u r)-\rho G(c)\right)\right],
\end{aligned}\qquad
u(x,t)=-m \rho^{m-2} \frac{\partial \rho}{\partial r}\,,
\right.		
\end{equation*}
where the left system is again the prediction step and the right relation is the correction step. The discretization is similar to that in Section 4 and we omit the details here.  

If we start with an indicator function for the cell density, i.e., $$\rho(0,t)= \mathbf{1}_{r_{-}(0) \leq r \leq r_+(0)}\,,$$ then as $m$ goes to infinity, $\rho$ remains a characteristic function and the pressure has the form \cite{LTWZ}
\begin{equation}\label{eqn:pinf1ann}
p_\infty(r,t) = - \frac{r^2}{4} + \frac{r_+^2 - r_-^2 }{4 (\ln r_+ - \ln r_-)} \ln r  - \frac{r_+^2 \ln r_- - r_-^2 \ln r_+ }{ 4 (\ln r_+ - \ln r_-)}\,,
\end{equation}
where $r_-(t)$ and $r_+(t)$ satisfies
\begin{eqnarray*}
\partial_t r_- = \frac{1}{2} r_- - \frac{ r_+^2 - r_-^2}{ 4  r_-(\ln r_+ - \ln r_- )}\,,
\qquad \partial_t r_+ = \frac{1}{2} r_+ - \frac{r_+^2 - r_-^2 }{ 4  r_+(\ln r_+ - \ln r_- )}\,.
\end{eqnarray*}
Numerically, we choose the initial data as
\begin{equation} \label{IC:1ann}
\rho(r,0) = \left( \frac{m-1}{m} p_\infty (r,0)\right)^{\frac{1}{m-1}}, \quad r_+(0) = 1, ~ r_-(0)=0.6\,,
\end{equation}
where $p_\infty$ is defined in \eqref{eqn:pinf1ann}, and compare our numerical solution with the analytical one specified above. The results are gathered in Fig.~\ref{fig:1ann}. 
\begin{figure}[!ht]
\centering
\includegraphics[width = 0.48\textwidth]{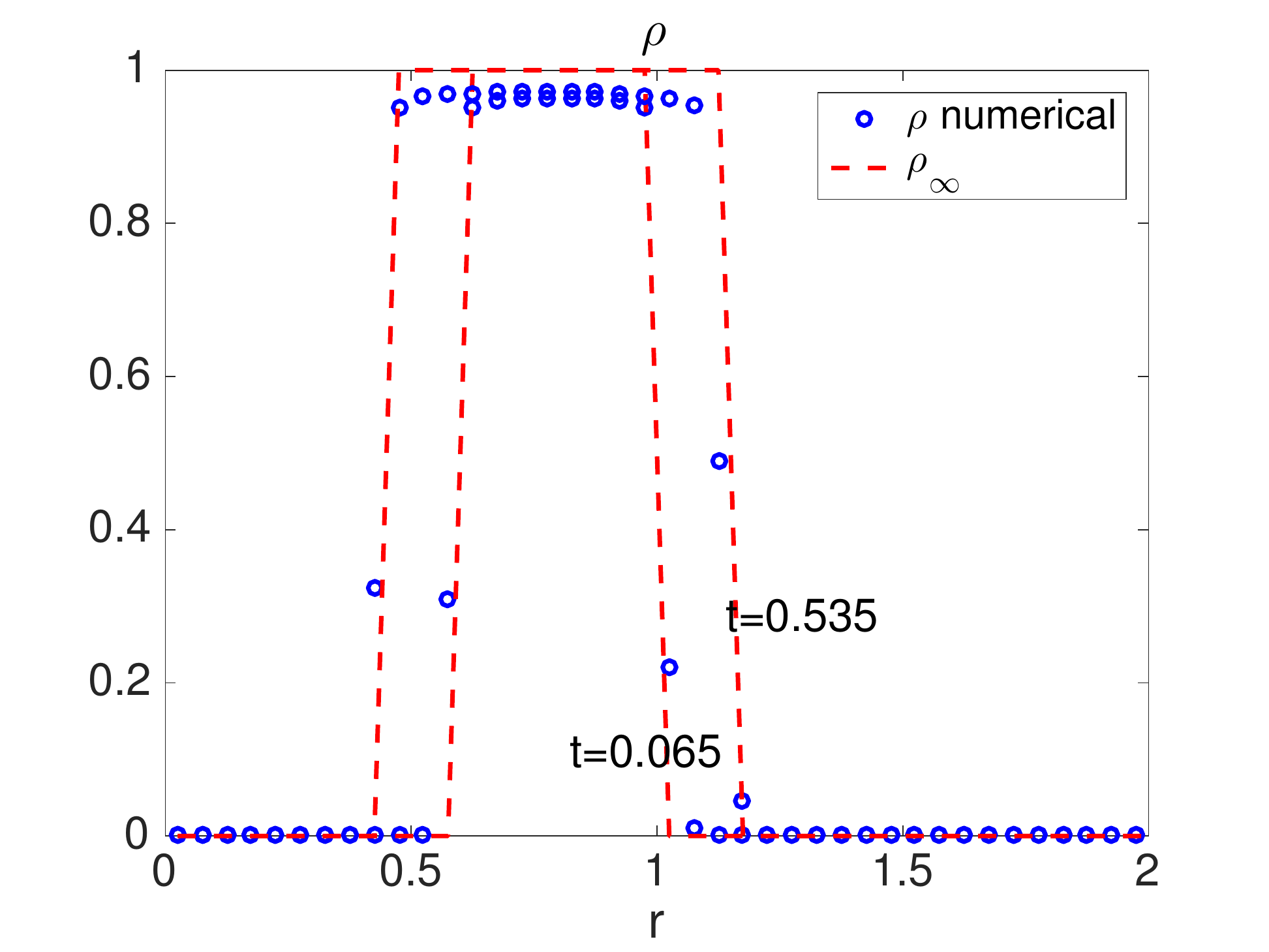}
\includegraphics[width = 0.48\textwidth]{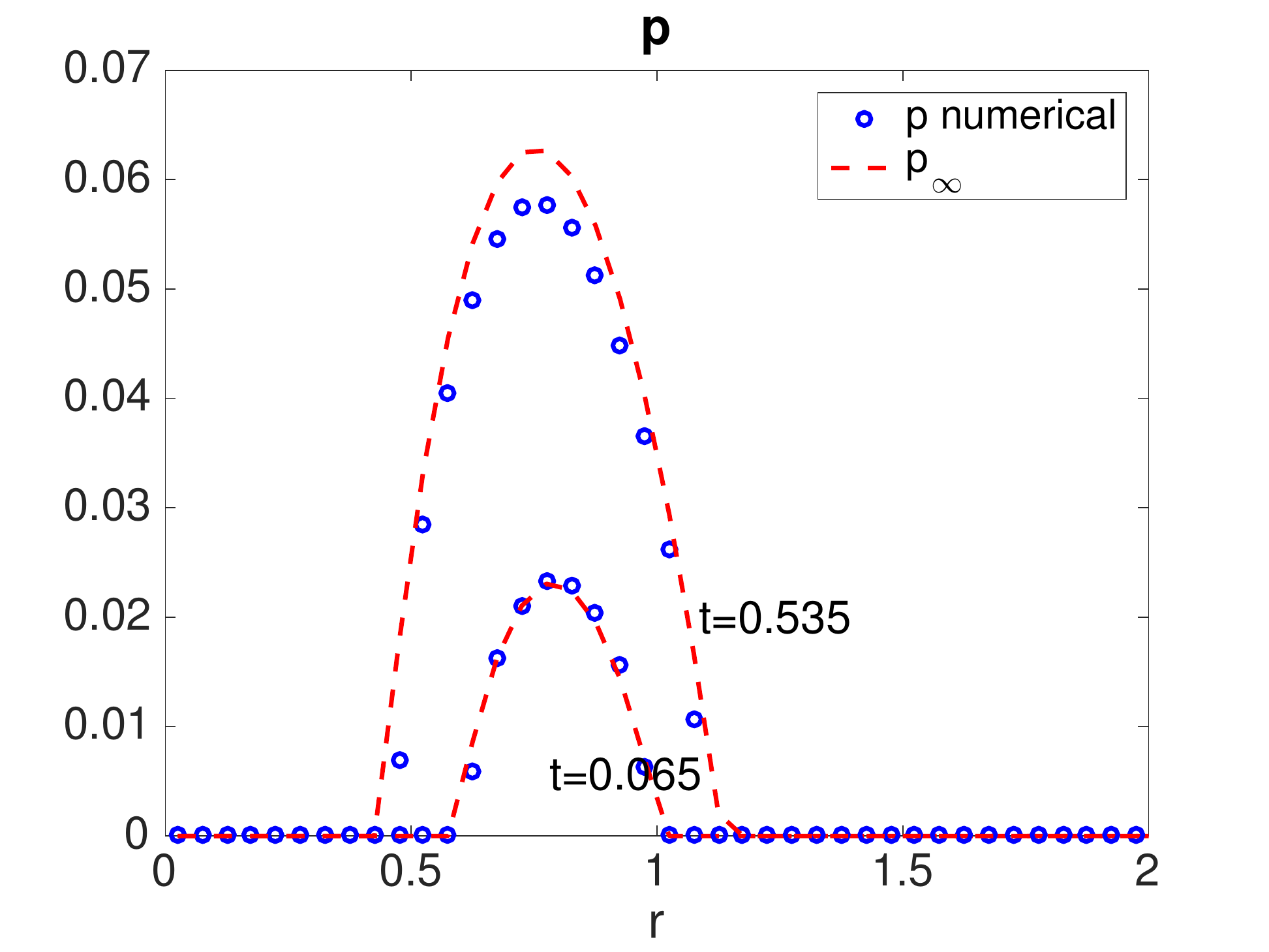}
\caption{2D radial symmetric case with one annular initial data \eqref{IC:1ann} and $m=100$. Here $\Delta r =0.05$, and $\Delta t=0.0025$.  }
\label{fig:1ann}
\end{figure}

\begin{figure}[!ht]
\centering
\includegraphics[width = 0.48\textwidth]{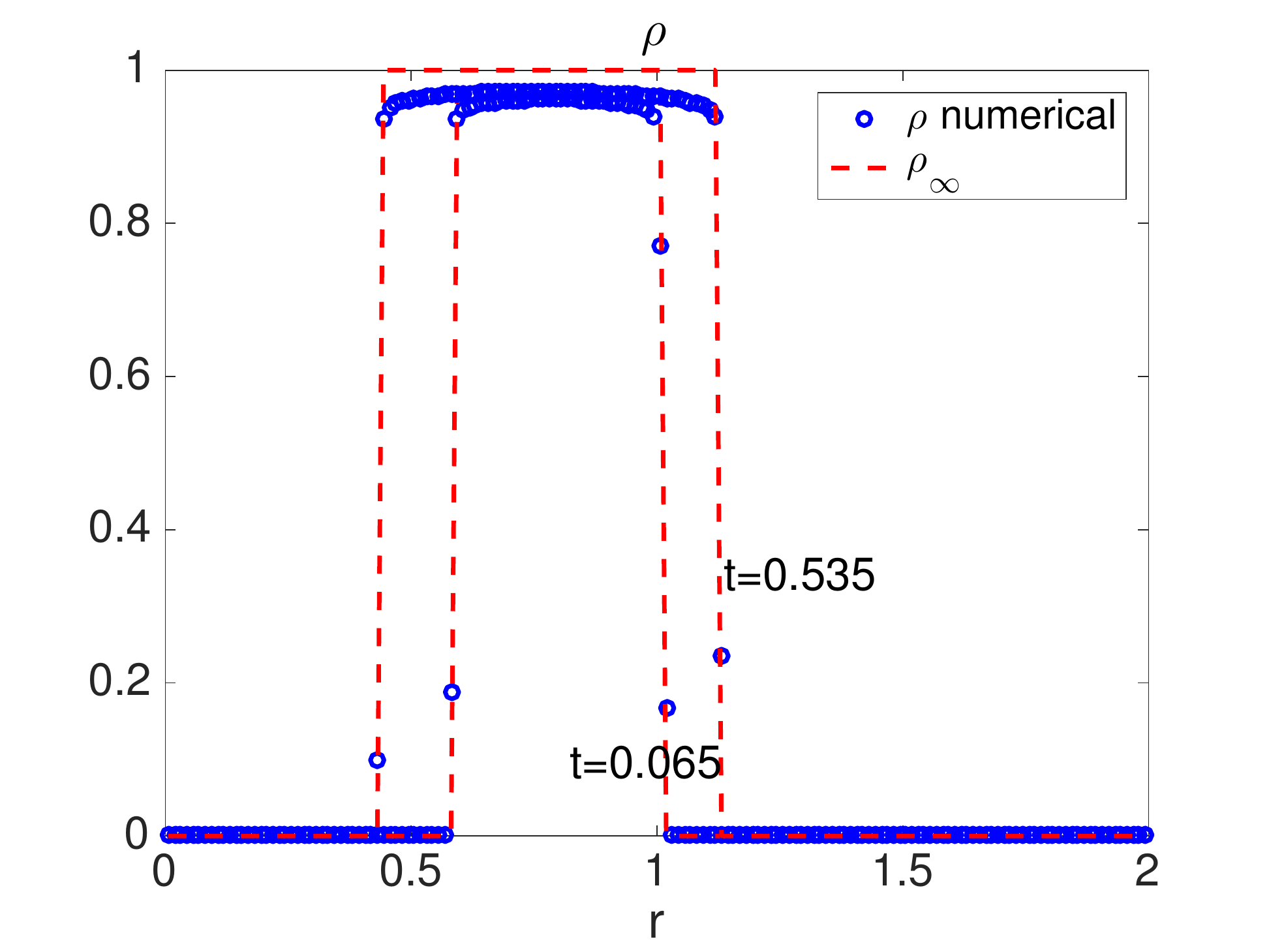}
\includegraphics[width = 0.48\textwidth]{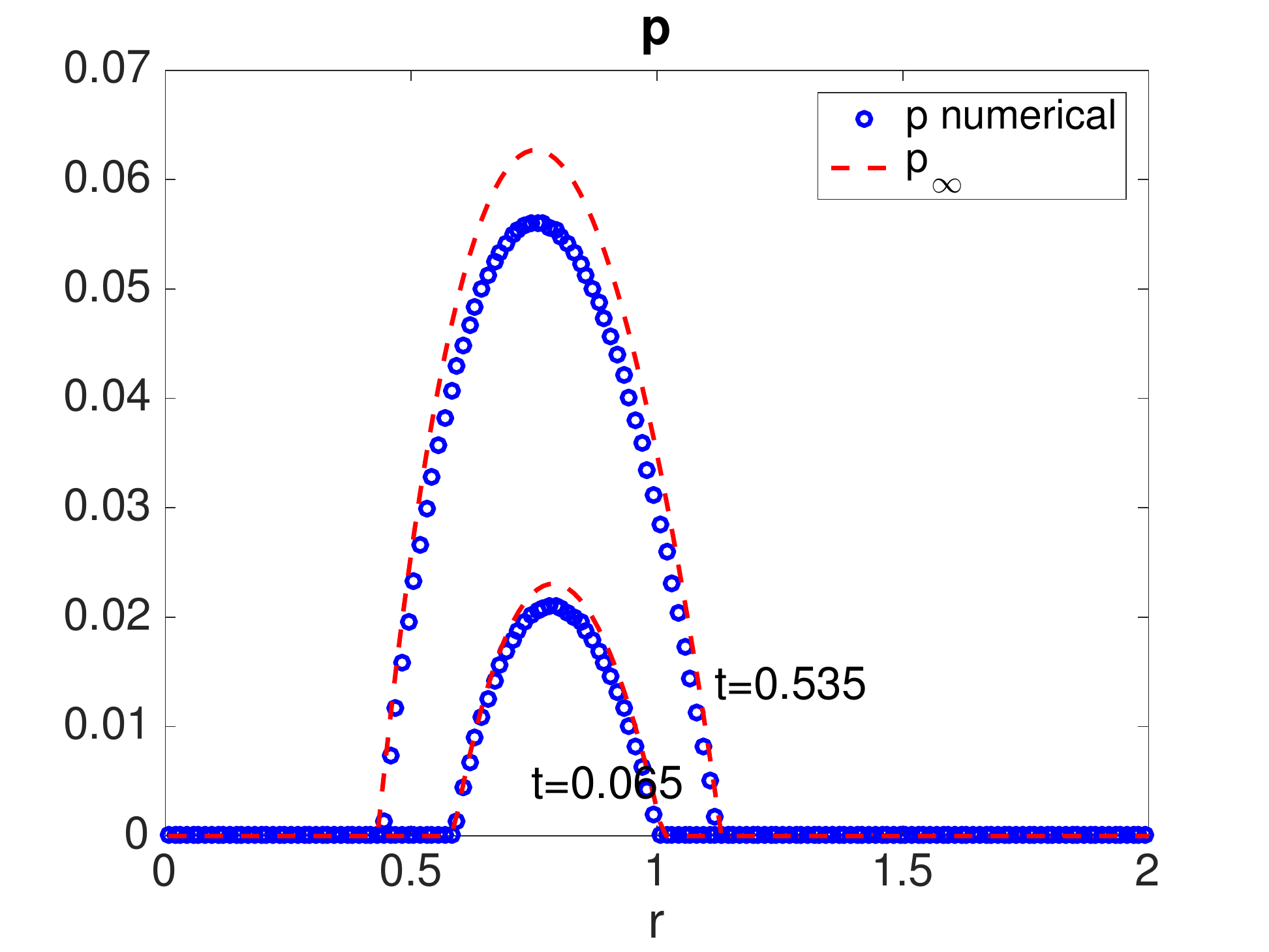}
\caption{The same example as in Fig. \ref{fig:1ann} but with finer mesh $\Delta r =0.0125$, and $\Delta t=0.000625$.  }
\label{fig:1ann}
\end{figure}

Next we take the initial condition with two annulus:
\[
\rho(r,0) = \mathbf{1}_{r_1(0)\leq r \leq r_2(0)} +\mathbf{1}_{r_3(0)\leq r \leq r_4(0)} \, 
\qquad r_1(0)< r_2(0) <r_3(0) < r_4(0)\,,
\]
then the pressure, as $m$ goes to infinity, takes the form 
\begin{eqnarray}
p_\infty(r,t) = \left\{ \begin{array}{cc} 
-\frac{r^2}{4} + \frac{r_2^2 - r_1^2}{4 (\ln r_2 - \ln r_1)} \ln r - \frac{r_2^2 \ln r_1 - r_1^2 \ln r_2}{4(\ln r_2 - \ln r_1)} \,, & r_1 \leq r \leq r_2\,;  \label{pinf_2ann0}
\\ -\frac{r^2}{4} + \frac{r_4^2 - r_3^2}{4 (\ln r_4 - \ln r_3)} \ln r - \frac{r_3^2 \ln r_3 - r_4^2 \ln r_4}{4(\ln r_4 - \ln r_3)} & r_3\,, \leq r \leq r_4\,,  \label{pinf_2ann1}
\end{array} \right.
\end{eqnarray}
where $r_1(t) \sim r_4(t)$ changes according to 
\begin{eqnarray*}
&& \partial_t r_1 = \frac{r_1}{2} - \frac{r_2^2 - r_1^2}{4r_1 (\ln r_2 - \ln r_1)}\,, \quad 
\partial_t r_2 = \frac{r_2}{2} - \frac{r_2^2 - r_1^2}{4r_2 (\ln r_2 - \ln r_1)}\,, 
\\&&
\partial_t r_3 = \frac{r_3}{2} - \frac{r_4^2 - r_3^2}{4r_3 (\ln r_4 - \ln r_3)}\,, \quad 
\partial_t r_4 = \frac{r_4}{2} - \frac{r_4^2 - r_3^2}{4r_4 (\ln r_4 - \ln r_3)}\,.
\end{eqnarray*}
To compute, we set initial cell density to be 
\begin{equation}
\rho(r,0) = \left( \frac{m-1}{m} p_\infty (r,0) \right)^{\frac{1}{m-1}},  \quad 
r_1(0) = 0.6,~ r_2(0) = 0.9, ~ r_3(0) = 1.5, ~ r_4(0) = 1.8\,,
\end{equation}
where $p_\infty $ is defined in \eqref{pinf_2ann0} \eqref{pinf_2ann1}. The solutions are collected in Fig.~\ref{fig:2ann}.

In both examples, the time evolution of both cell density and pressure are displayed, overlaid with analytical solution, wherein good agreement can be observed.  
\begin{figure}[!ht]
\centering
\includegraphics[width = 0.48\textwidth]{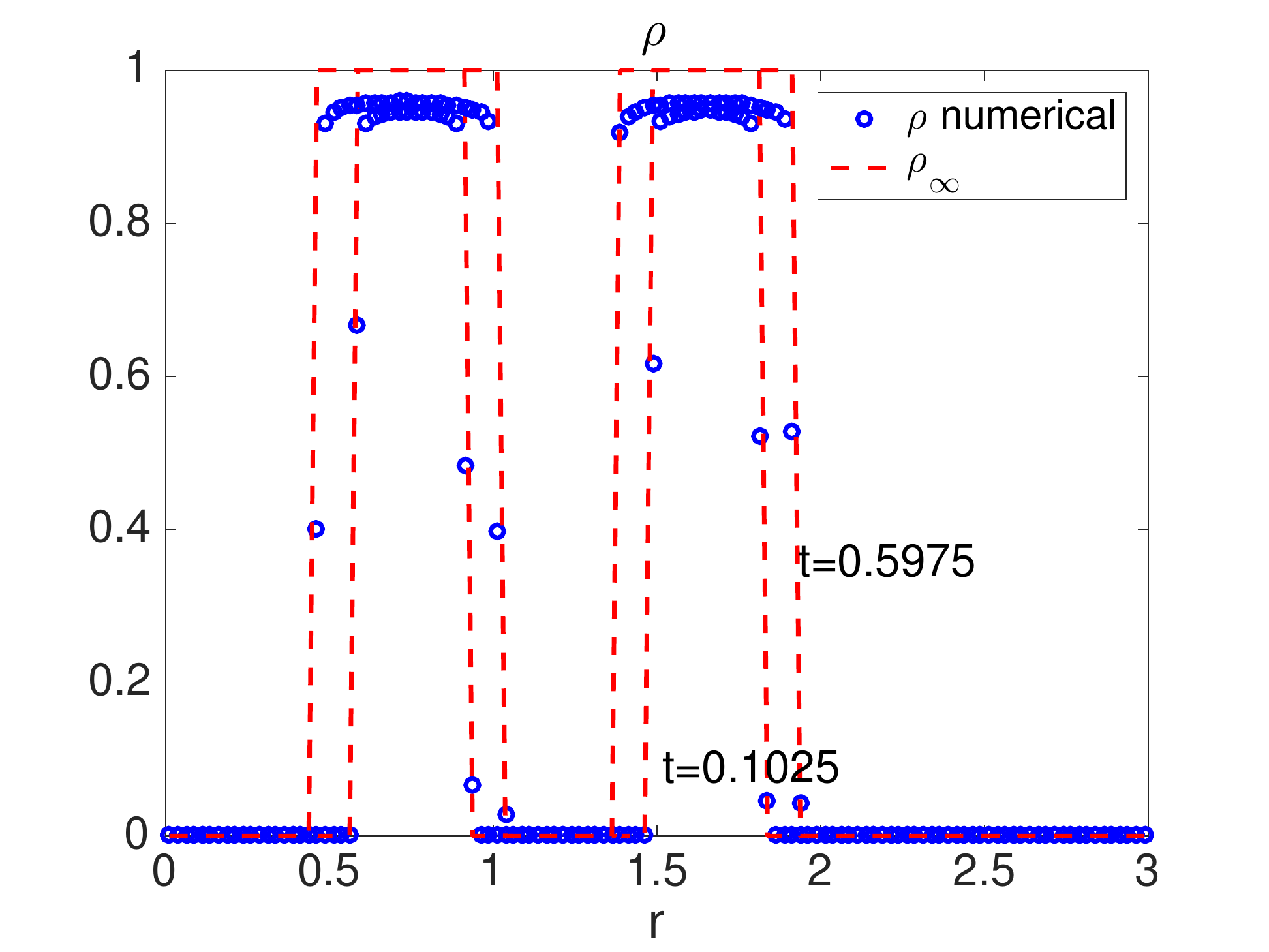}
\includegraphics[width = 0.48\textwidth]{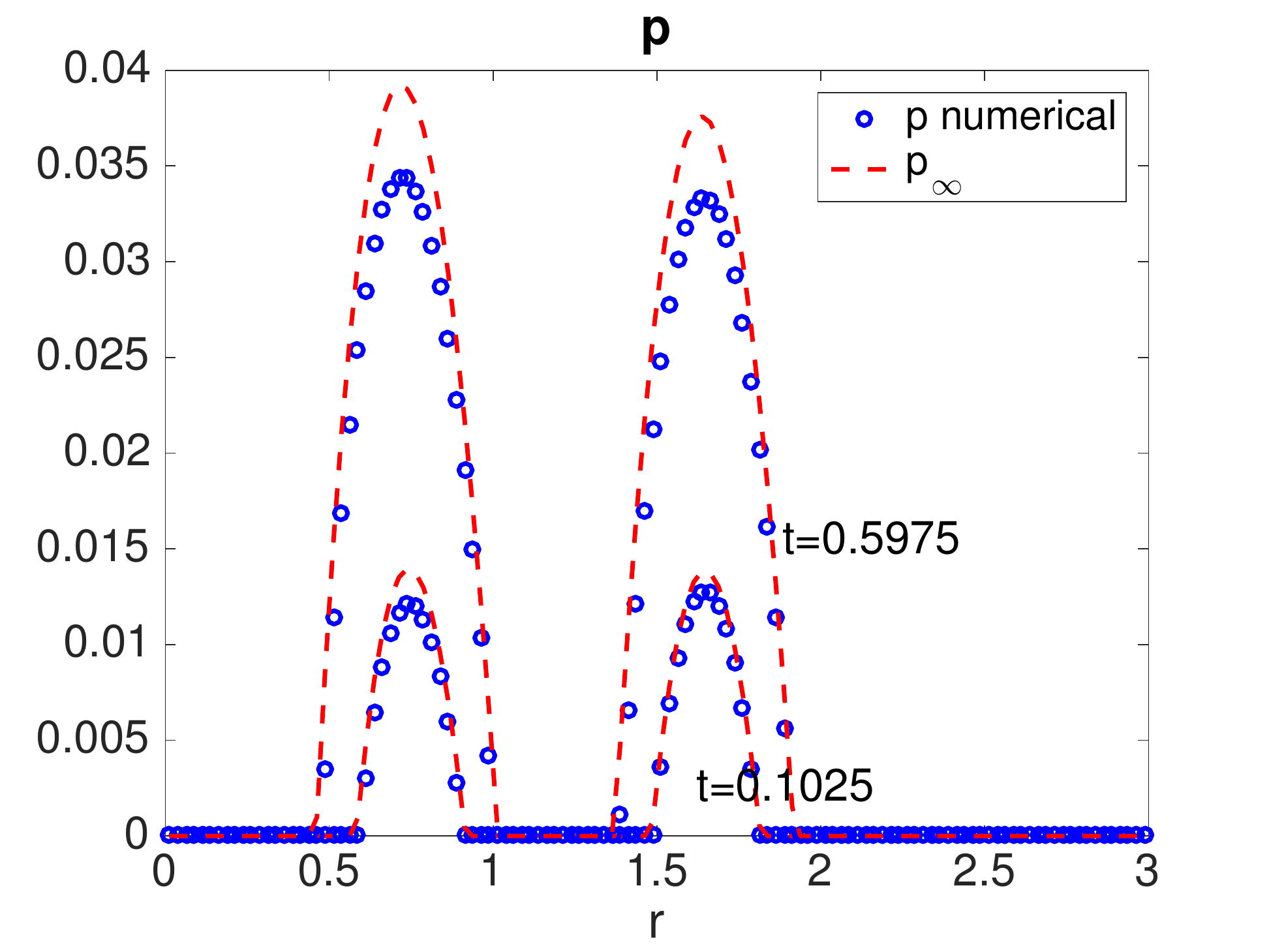}
\caption{2D radial symmetric case with double annulus initial data \eqref{IC:1ann} and $m=80$. Here $\Delta r =0.025$, and $\Delta t=0.00125$.  }
\label{fig:2ann}
\end{figure}

\begin{figure}[!ht]
\centering
\includegraphics[width = 0.48\textwidth]{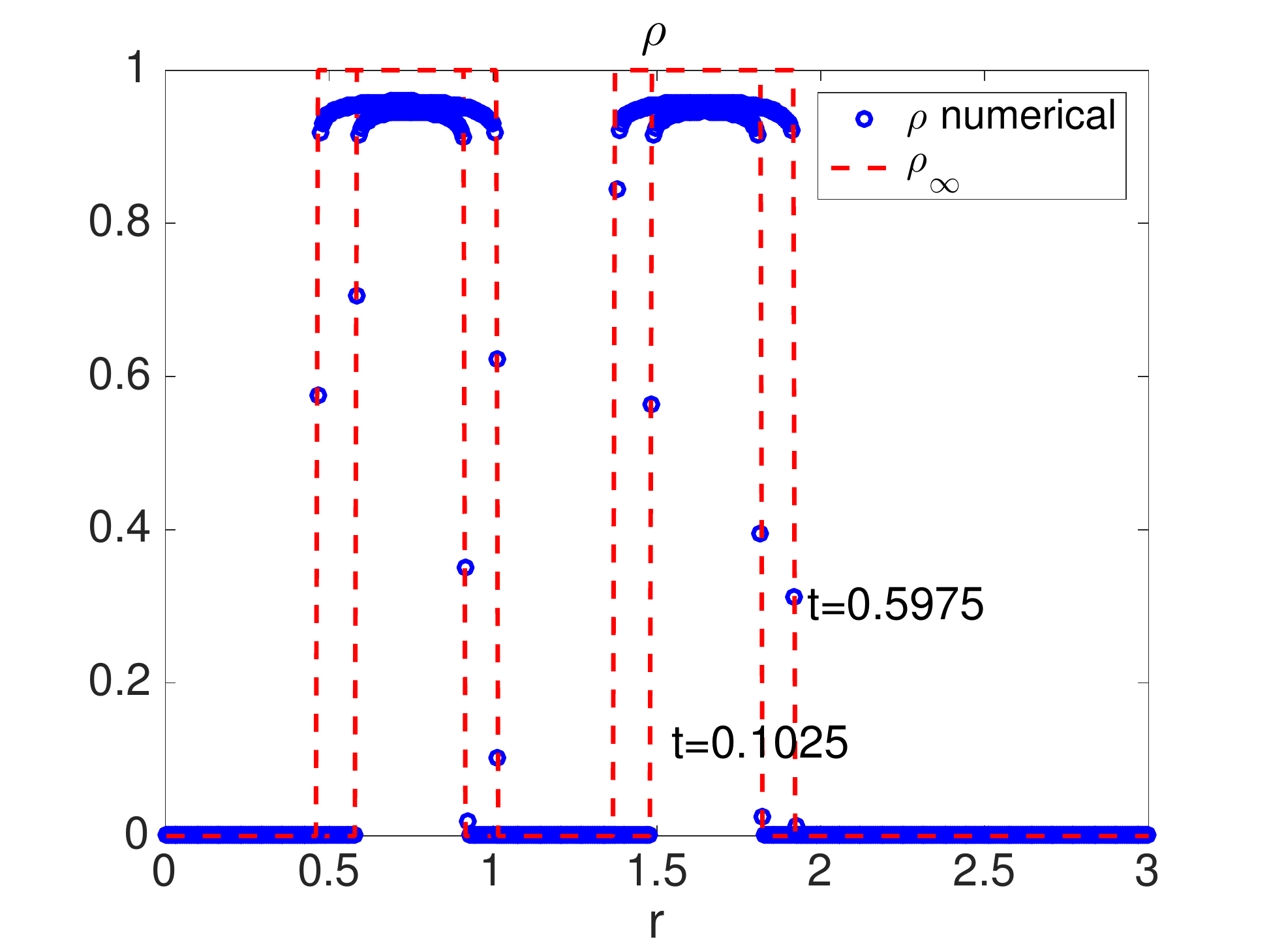}
\includegraphics[width = 0.48\textwidth]{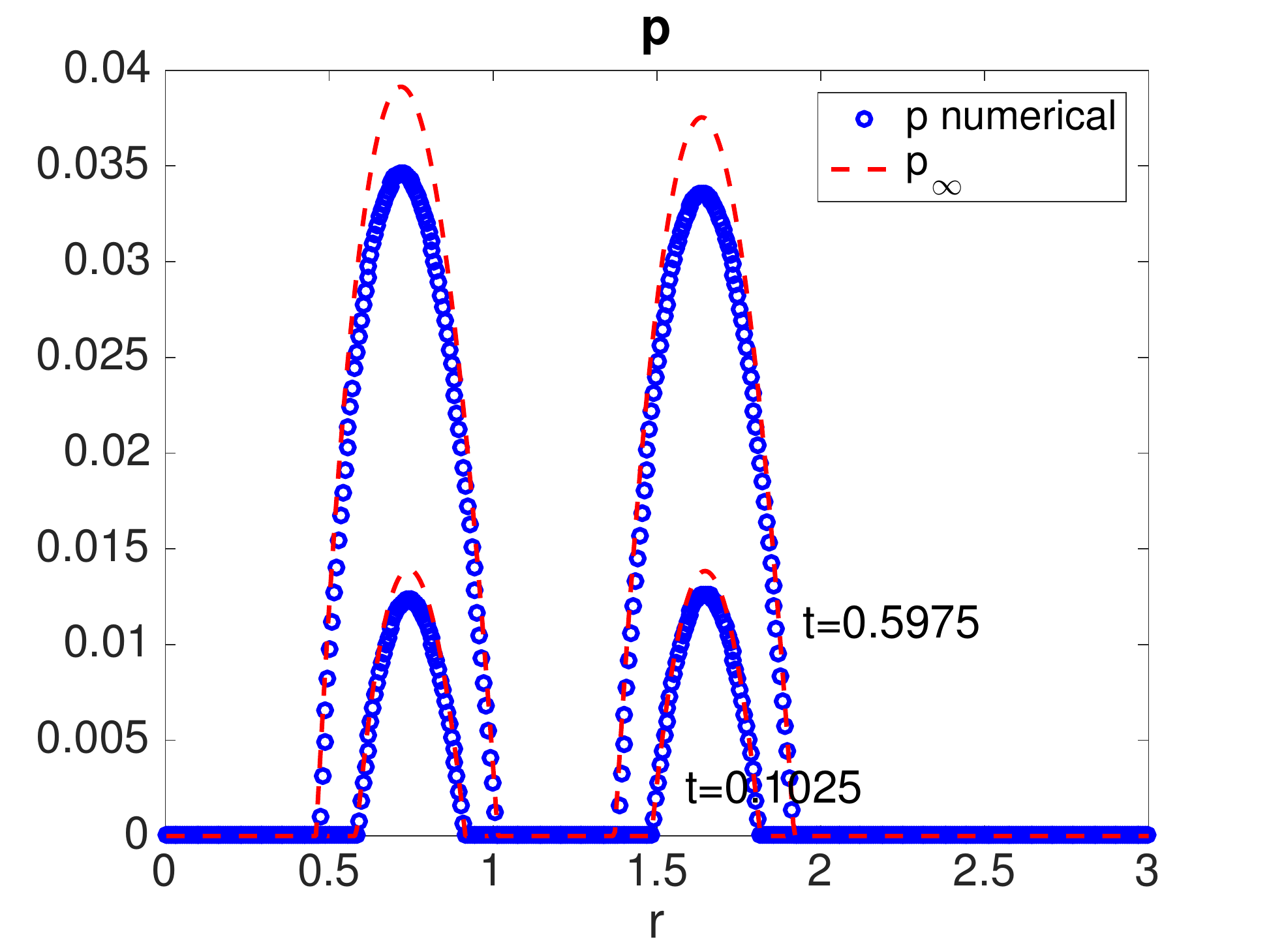}
\caption{The same example as in Fig.\ref{fig:2ann} but with finer mesh $\Delta r =0.00625$, and $\Delta t=3.125e-4$.  }
\label{fig:1ann}
\end{figure}

\subsection{Proliferating, quiescent and deal cells model}
In this section, we consider a biologically more realistic model \cite{BLM,Pbook}. Let $\rho_P(x,t)$, $\rho_Q(x,t)$, and $\rho_D(x,t)$ be the cell densities for proliferating, quiescent and dead cells, then the cell movements obey:
\begin{eqnarray*}
&& \partial_t \rho_P + \partial_x ( u \rho_P) = G(\rho_P, \rho_Q, c) - a \rho_P + b \rho_Q \,,
\\ && \partial_t \rho_Q + \partial_x (u \rho_Q) = a \rho_P - b\rho_Q - d\rho_Q  \,, 
\\ && \partial_t \rho_D + \partial_x (u \rho_D) = d \rho_Q - \mu \rho_D \,,
\\ && u = - \frac{m}{m-1} \partial_x \rho^{m-1}, \qquad \rho = \rho_P + \rho_Q + \rho_D\,,
\end{eqnarray*}
where $a$, $b$, and $\mu$ are three constants, and $c$ is again the nutrient that satisfies \eqref{eqn:c0}. To solve it numerically, we first reformulate it into a prediction-correction framework, i.e.,
\begin{align}\label{eqn:PQD}
&\left\{\begin{aligned}
& \p_t u = m \partial_x \left[  \rho^{m-2} (\p_x (\rho u) - G + \mu \rho_D) \right]\,,\\
&\p_t \rho_P+\p_x (u \rho_P )= G(\rho_P, \rho-\rho_P - \rho_D, c) - a \rho_P + b (\rho-\rho_P-\rho_D) \,,\\
& \p_t \rho_D + \p_x (u \rho_D)= d (\rho- \rho_P - \rho_D) - \mu \rho_D\,, \\
&\p_t \rho  + \p_x (u \rho) = G(\rho_P, \rho-\rho_P - \rho_D, c) - \mu \rho_D\,. \\
\end{aligned}
\right.
\\
& \quad  u(x,t)=-\f{m}{m-1}\p_x \rho^{m-1}(x,t)\,,	
\end{align}
and then discretize it in the same as in Section 4. Here we write an equation for the total density $\rho$ in place of $\rho_Q$ so that the equation for $u$ and $\rho$ is a reminiscent of the single species model. 

Our first test is when $c$ takes on an {\it in vitro} model \eqref{eqn:vitro1D}. We choose $a = b=d=1$, and $\mu=0$, and let $G(\rho_P, \rho_Q, c) = c\rho_P$. The initial data is taken as 
\begin{equation}\label{IC:PQD}
\rho_P = (1-\text{cosh}(x)/\text{cosh}(1.025)) \chi_{-1.025\leq x\leq 1.025}, \quad  \rho_Q = \rho_D = 0;
\end{equation}
and the computational domain is set as $[-8,8]$. In Fig.~\ref{fig:PQD-vitro}, we plot different cell densities at various times $t=0, ~ 2, ~ 4$. One can observe that the total density marches outside just like the single species case, but the proliferate and quiescent cell densities decreases in the center due to the lack of nutrient, which in turn leads to the increase of dead cells in the center. 

\begin{figure}[!ht]
\centering
\includegraphics[width = 0.31\textwidth]{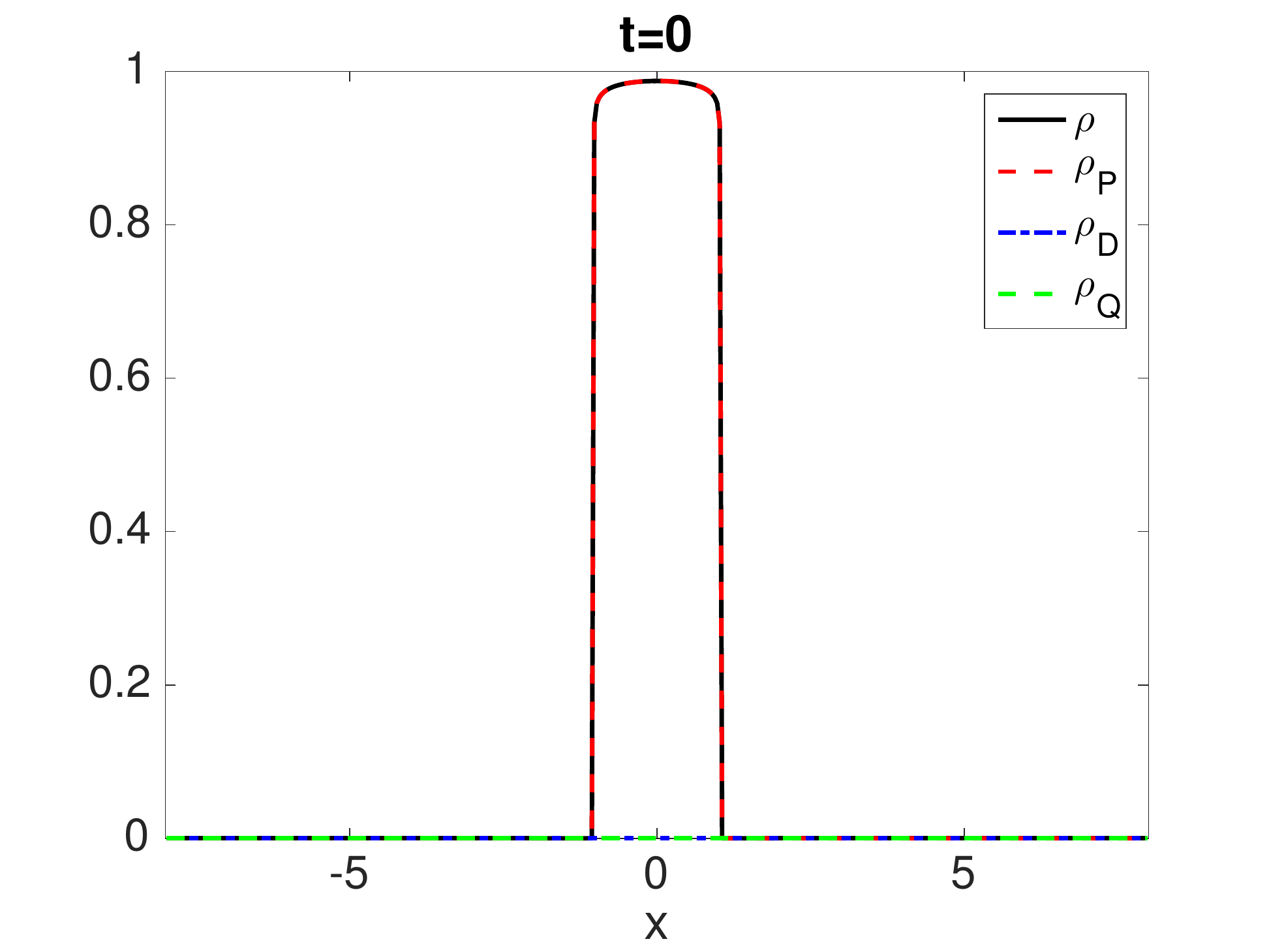}
\includegraphics[width = 0.31\textwidth]{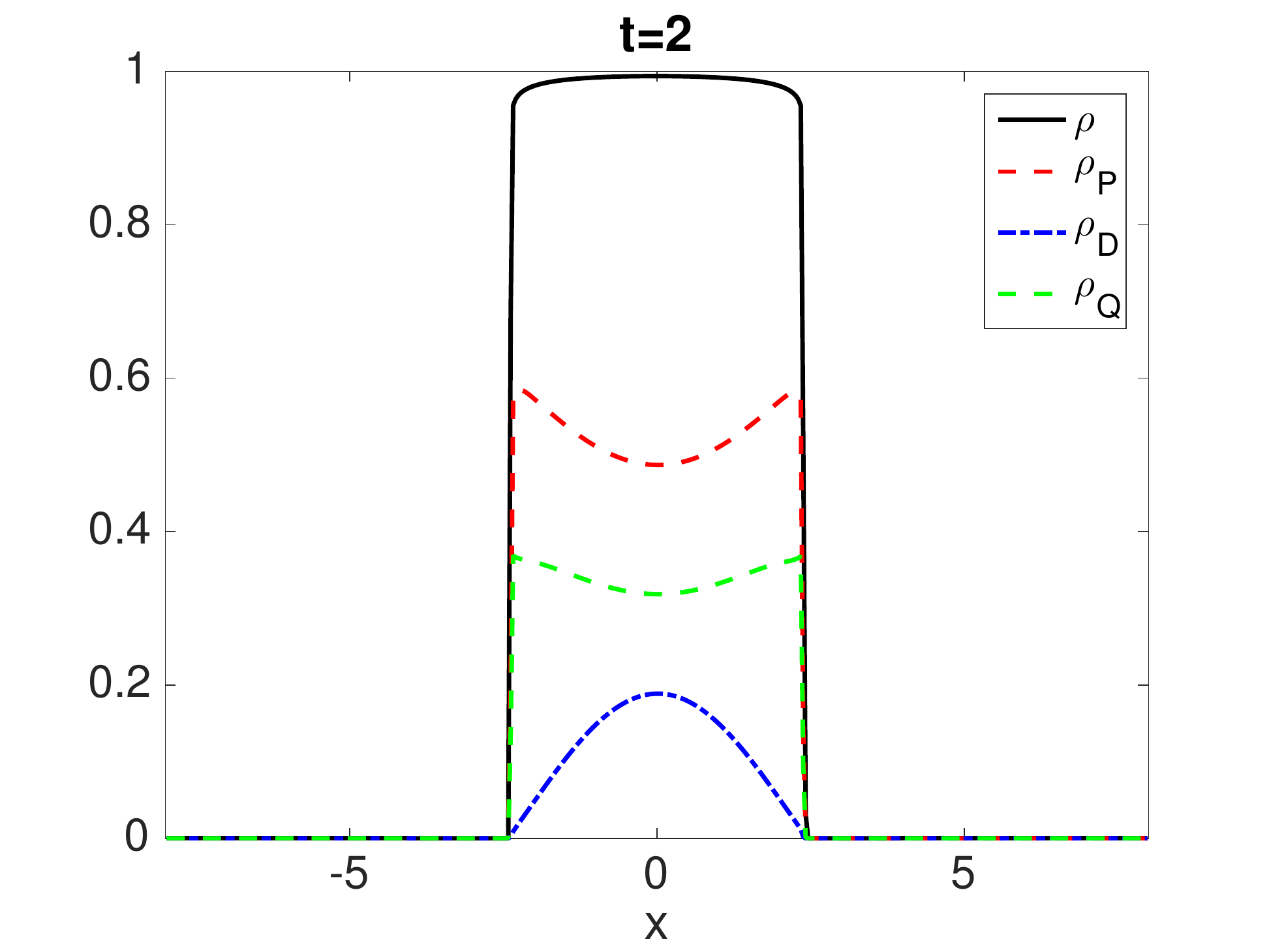}
\includegraphics[width = 0.31\textwidth]{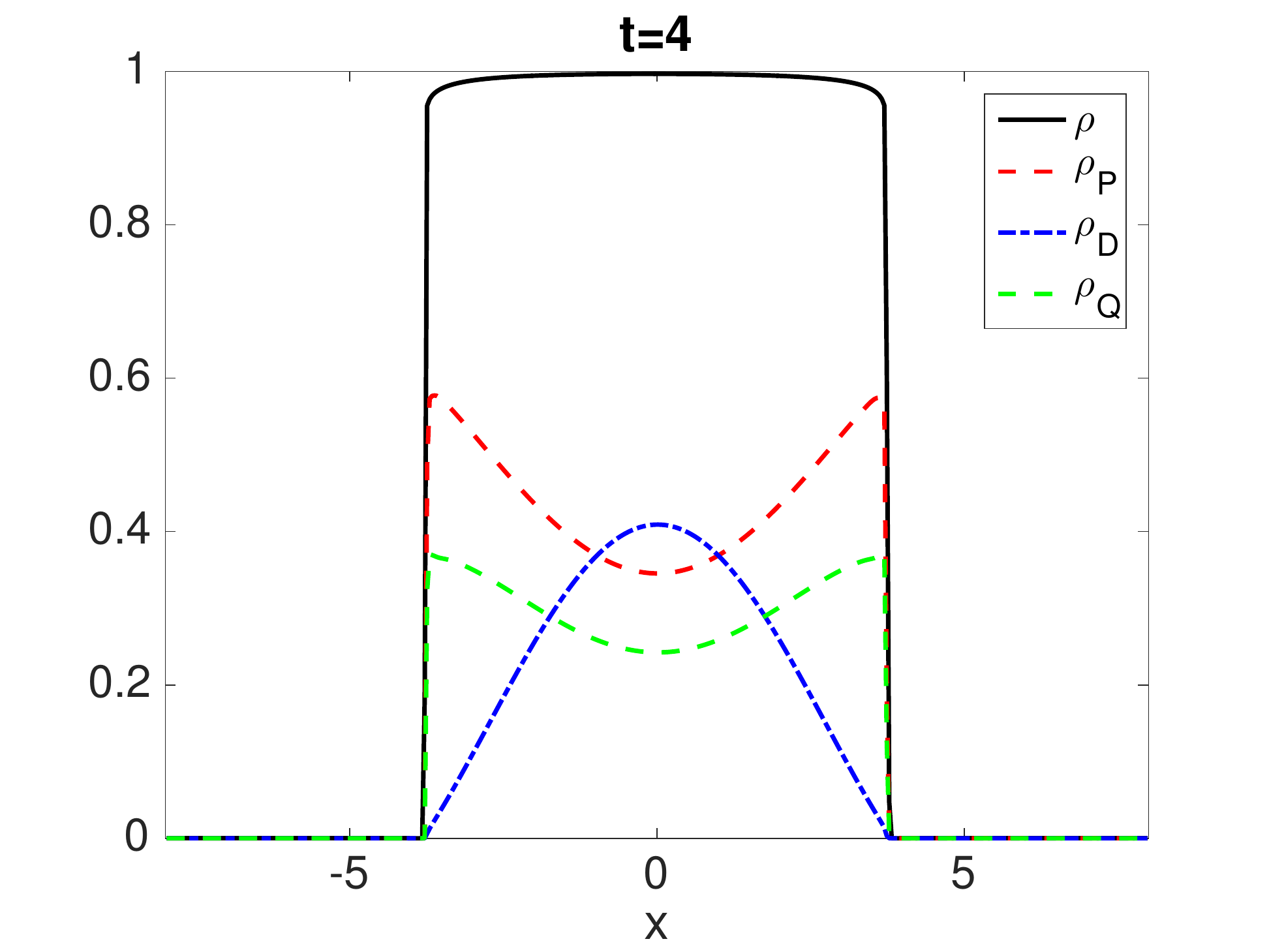}
\caption{Proliferating, quiescent and deal cells model with {\it in vitro} environment. From left to right: plots of $\rho$, $\rho_P$, $\rho_Q$, $\rho_D$ at time $t=0$ (left), $t=2$ (middle) and $t=4$ (right). Here $\Delta x = 0.04$, $\Delta t = 0.002$, $m=80$. }
\label{fig:PQD-vitro}
\end{figure}

Next we consider an {\it in vivo} model for the nutrient \eqref{eqn:vivo1D}. In this case, let $a=2$, $b=d =1$, $\mu=0$, and $G(\rho_P, \rho_Q, c) = \rho_P c$. The initial data is taken the same as \eqref{IC:PQD}, where $x\in[-8,8]$. The plots of different cell densities along time in displayed in Fig.~\ref{fig:PQD-vivo}, wherein similar trend as in the {\it in vitro} case is observed. However, there noticeable differences in the shape of dead cells, the front moving speed, and the relationship between the proliferate cell and quiescent cells. 

\begin{figure}[!ht]
\centering
\includegraphics[width = 0.31\textwidth]{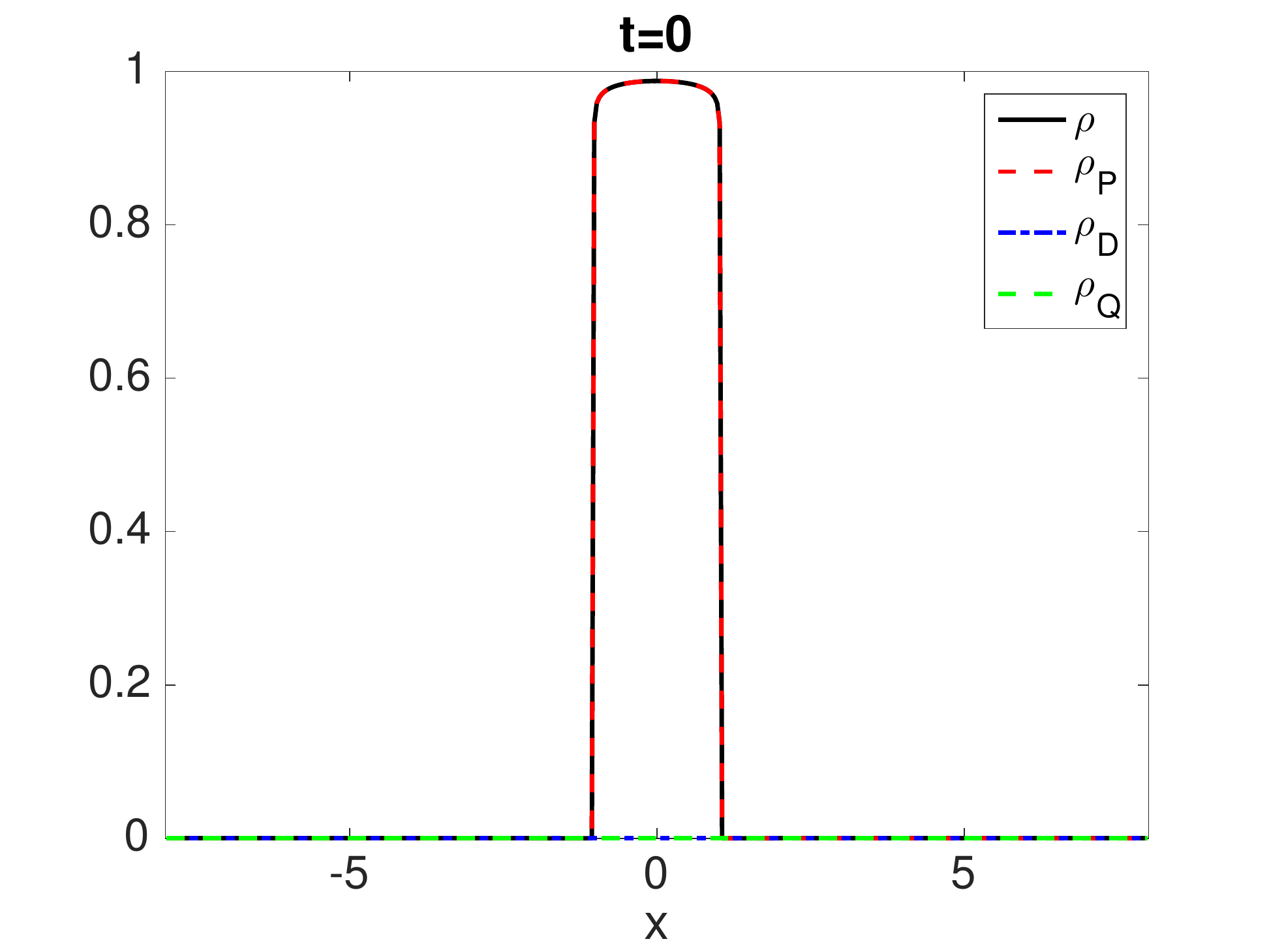}
\includegraphics[width = 0.31\textwidth]{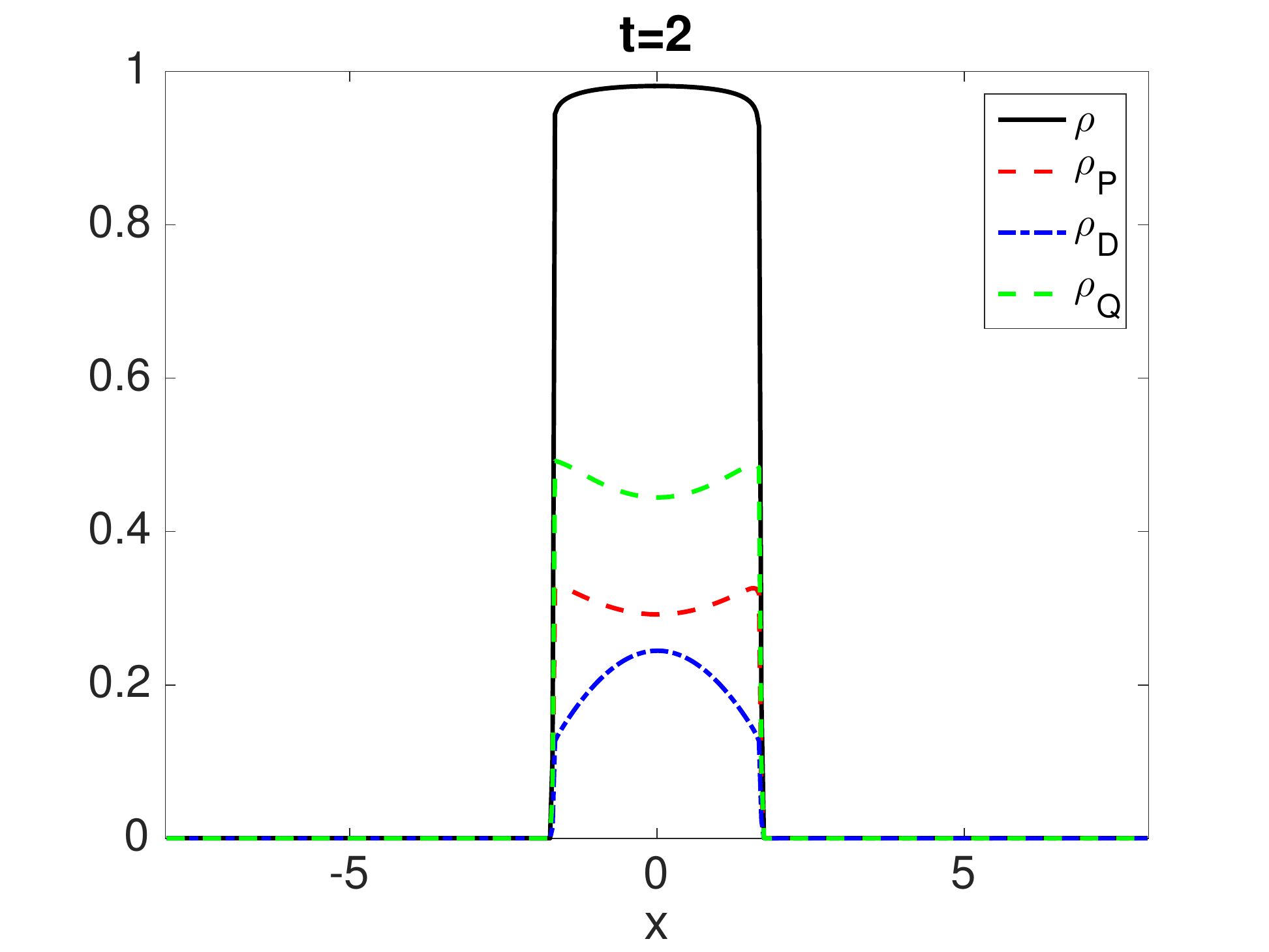}
\includegraphics[width = 0.31\textwidth]{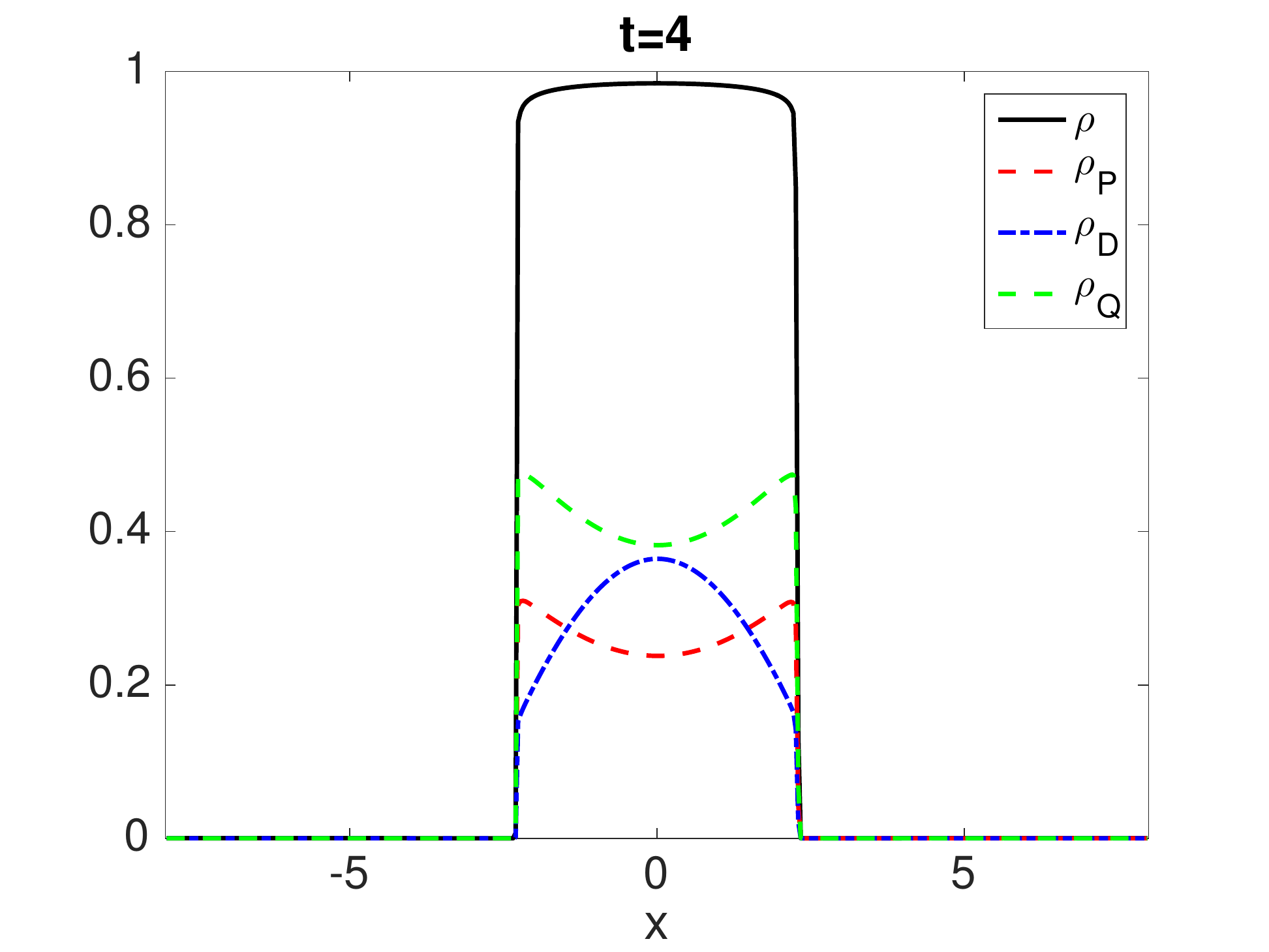}
\caption{Proliferating, quiescent and deal cells model with {\it in vitro} environment. From left to right: plots of $\rho$, $\rho_P$, $\rho_Q$, $\rho_D$ at time $t=0$ (left), $t=2$ (middle) and $t=4$ (right). Here $\Delta x = 0.04$, $\Delta t = 0.002$, $m=80$. }
\label{fig:PQD-vivo}
\end{figure}

In the third and forth tests, we change the coefficient $a$, $b$, $d$ and $\mu$. For the {\it in vitro} case, let $a =1$, $b=2$, $d=0.5$, $\mu=0$, and plot the solution at time $t=4$ (Fig.~\ref{fig:PQD22} upper). For the {\it in vivo} case, choose $a =2$, $b=1$, $d=0.5$, $\mu=0$ (Fig.~\ref{fig:PQD22} lower). In both examples, the growth rate of the dead cells is smaller than the previous cases, and lead to a boundary separation of the cells. When time is large enough, a necrotic core appears in the center. 
\begin{figure}[!ht]
\centering
\includegraphics[width = 0.45\textwidth]{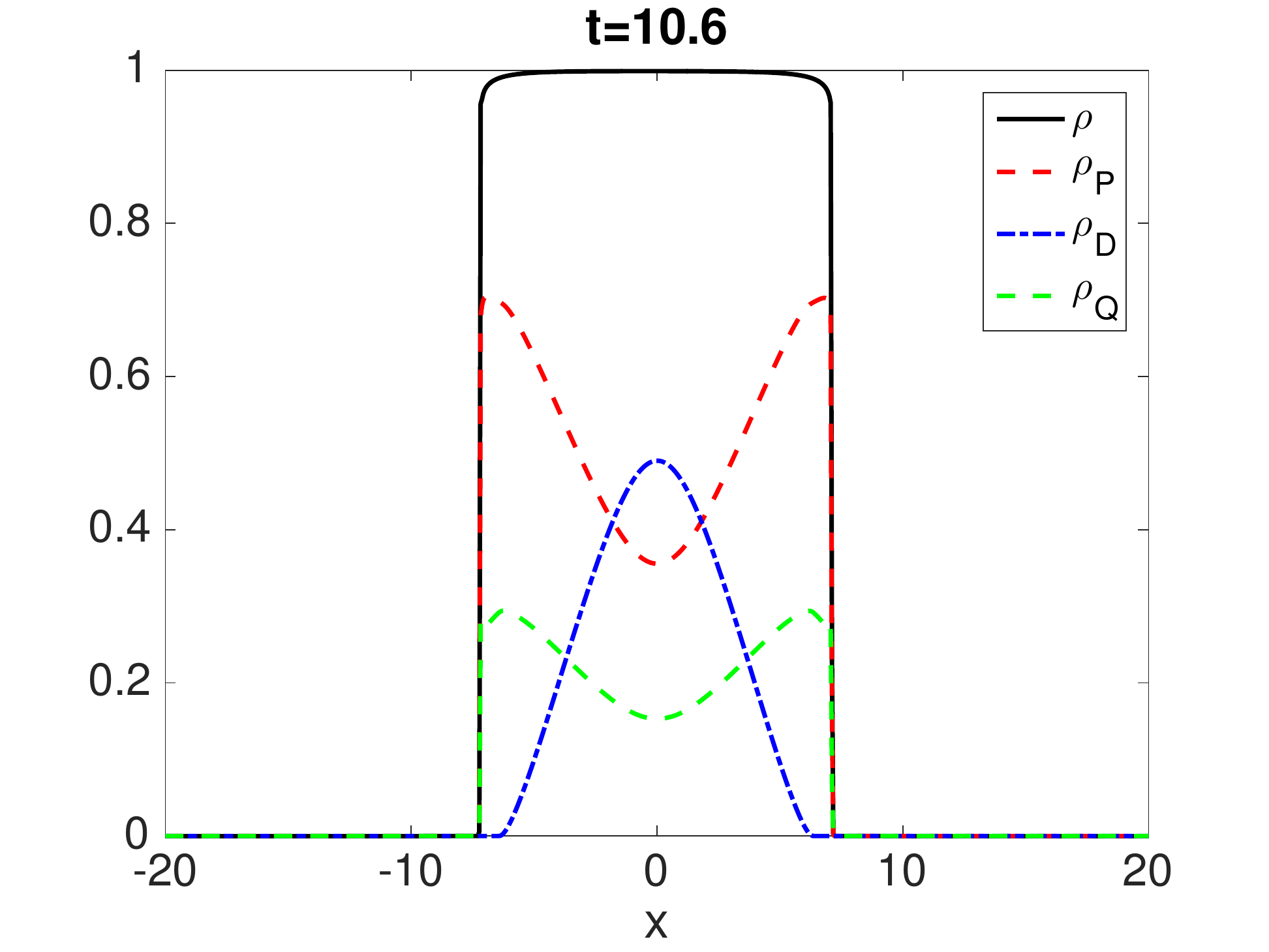}
\includegraphics[width = 0.45\textwidth]{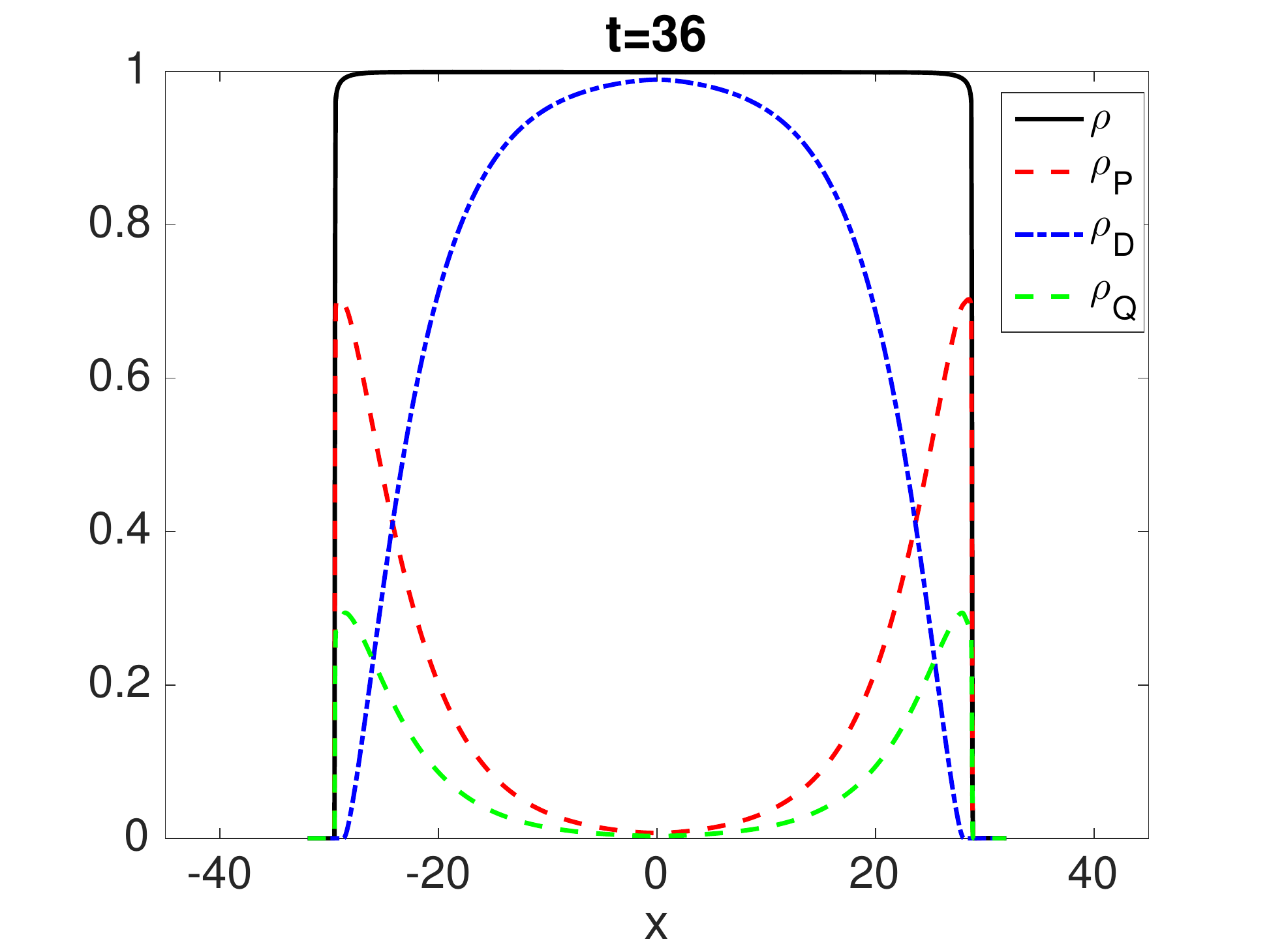}
\\
\includegraphics[width = 0.45\textwidth]{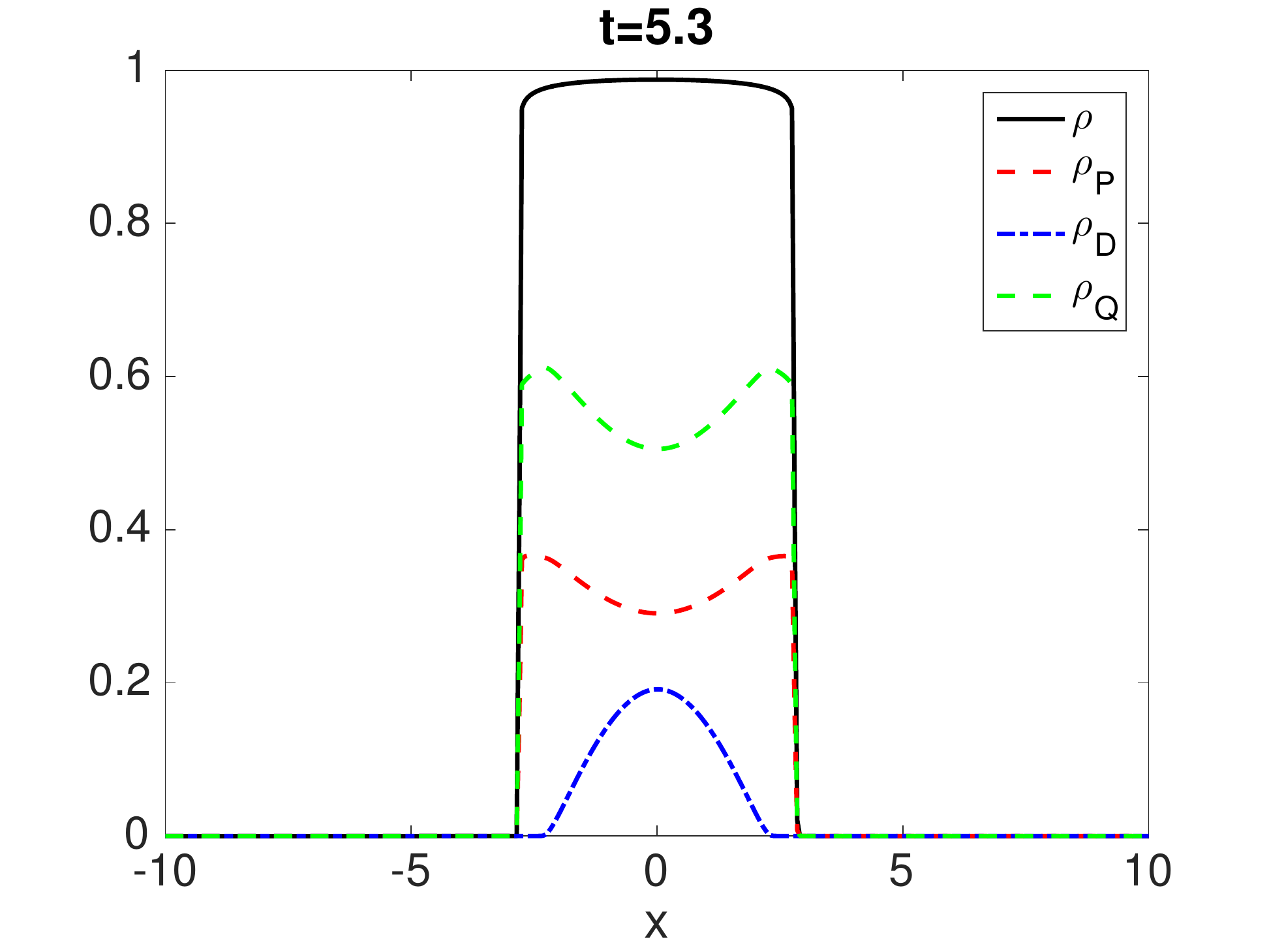}
\includegraphics[width = 0.45\textwidth]{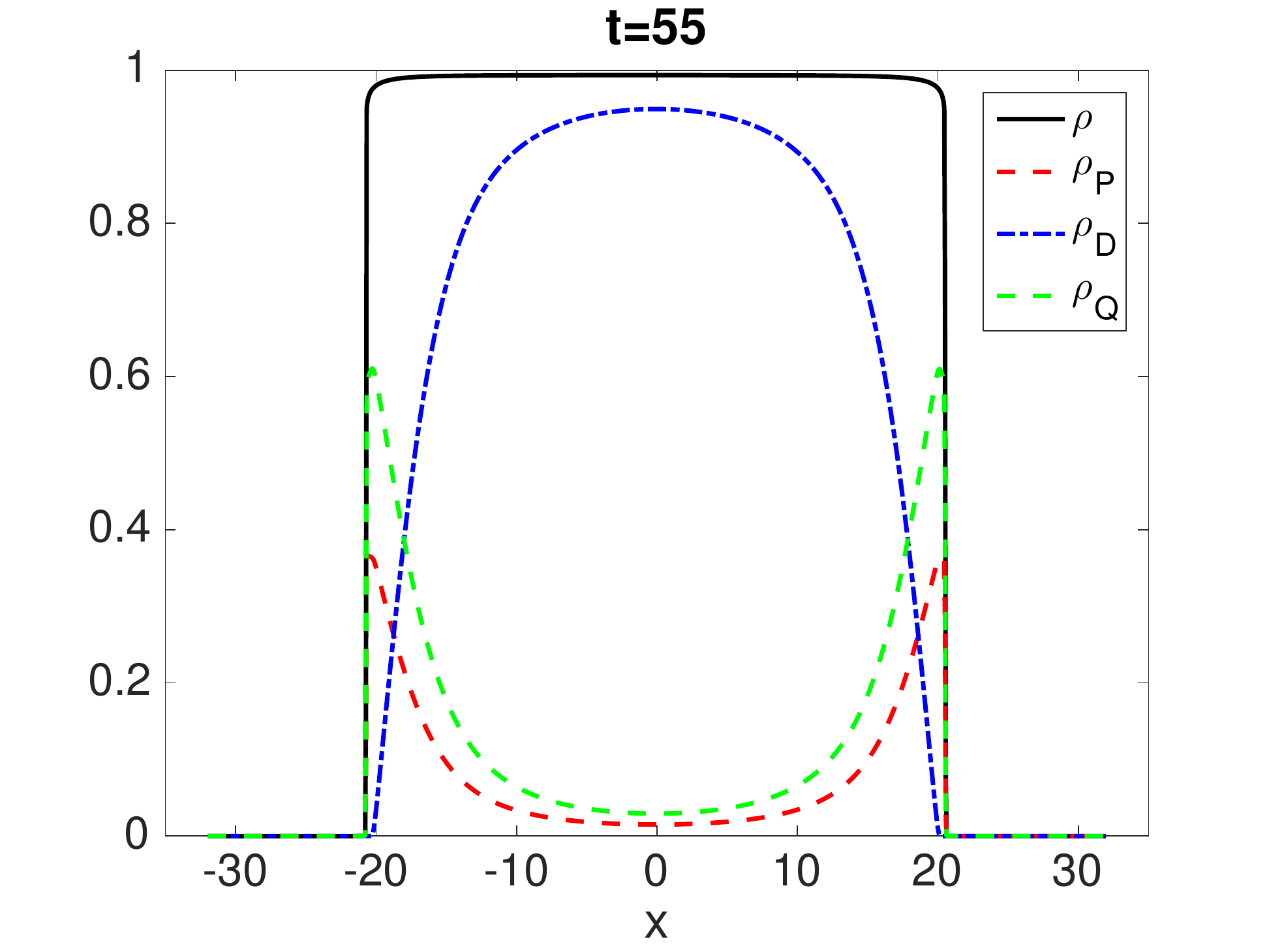}
\caption{Proliferating, quiescent and deal cells model with $\Delta x = 0.053$, $\Delta t = 0.0027$ at time $t=4$. Upper: {\it in vitro} environment. lower: {\it in vivo} environment. Here $m=80$.}
\label{fig:PQD22}
\end{figure}

\subsection{General 2D model}
At last, we conduct two tests in 2D with two different initial conditions. The nutrient is set to be a constant, i.e., $c\equiv 1$, $G(c) \equiv 1$. The computational domain is chosen as $(x,y)\in [-2,2]\times [-2,2]$, and the mesh is discretized with $\Delta x = \Delta y = 0.1$, $\Delta t = 0.005$. 

The first initial condition we considered is 
\begin{equation} \label{IC:2D-square}
\rho(x,y,0) = \left\{ \begin{array}{cc} 0.99 & (x,y)\in[0,0.5]\times[0,0.5] ~ \text{or} ~ [-0.6,-0.2]\times[-0.2,0.8]    \\ 0 & \text{otherwise} \end{array} \right. \,.
\end{equation} 
and evolution of $\rho(x,y)$ is plotted at different times in Fig.~\ref{fig:2D-square}.
\begin{figure}[!ht]
\centering
\includegraphics[width = 0.45\textwidth]{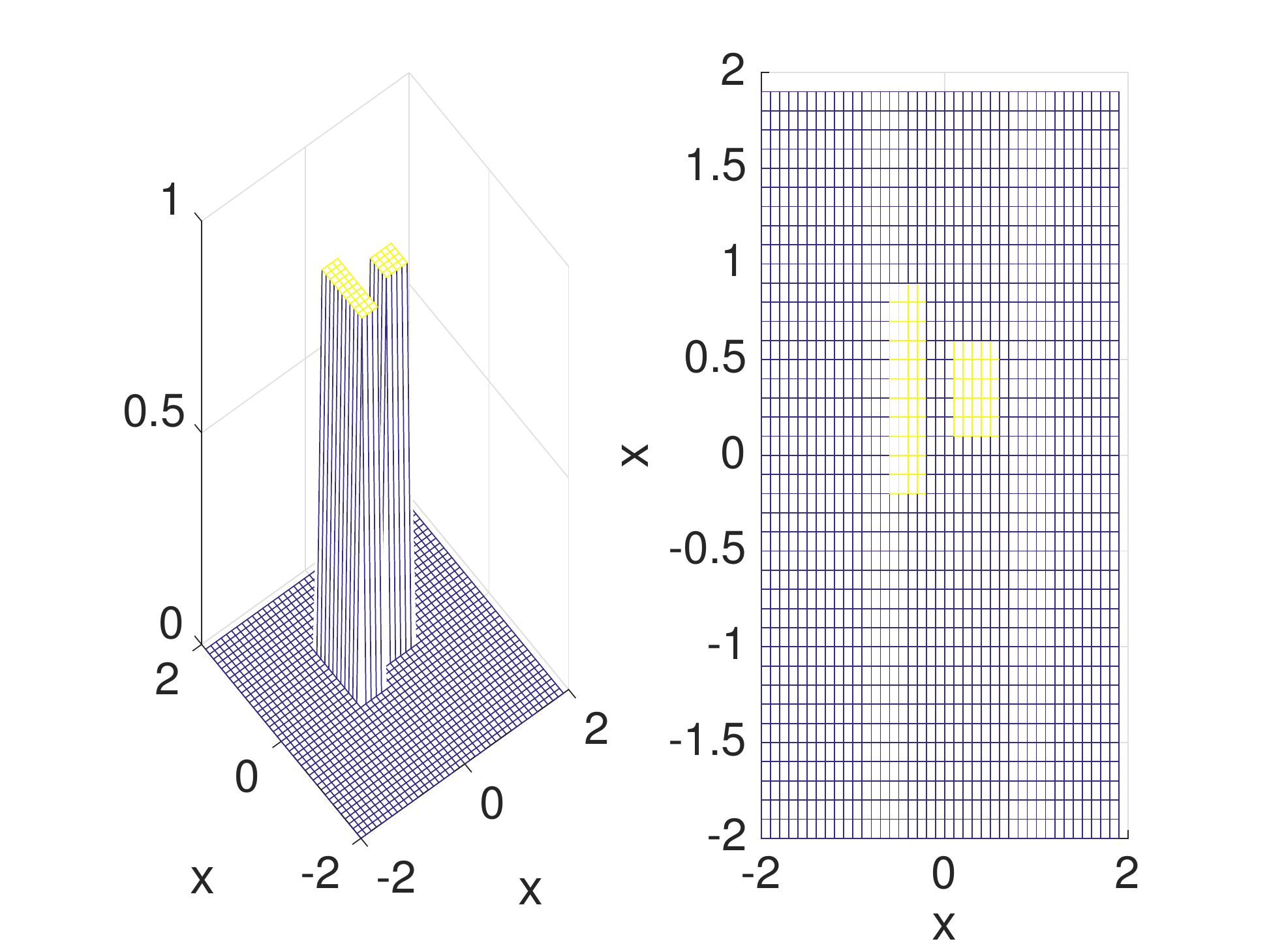}
\includegraphics[width = 0.45\textwidth]{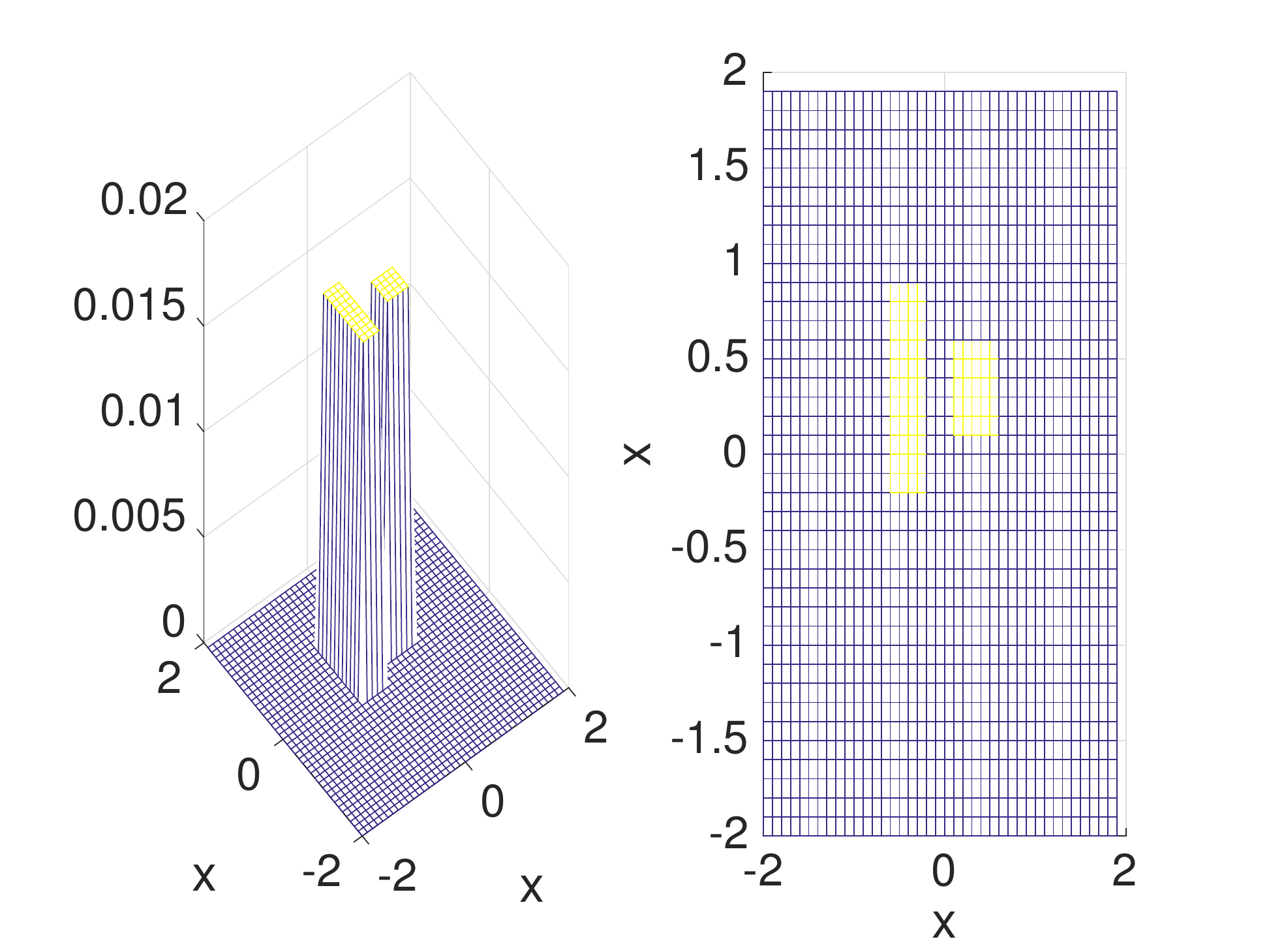}
\\
\includegraphics[width = 0.45\textwidth]{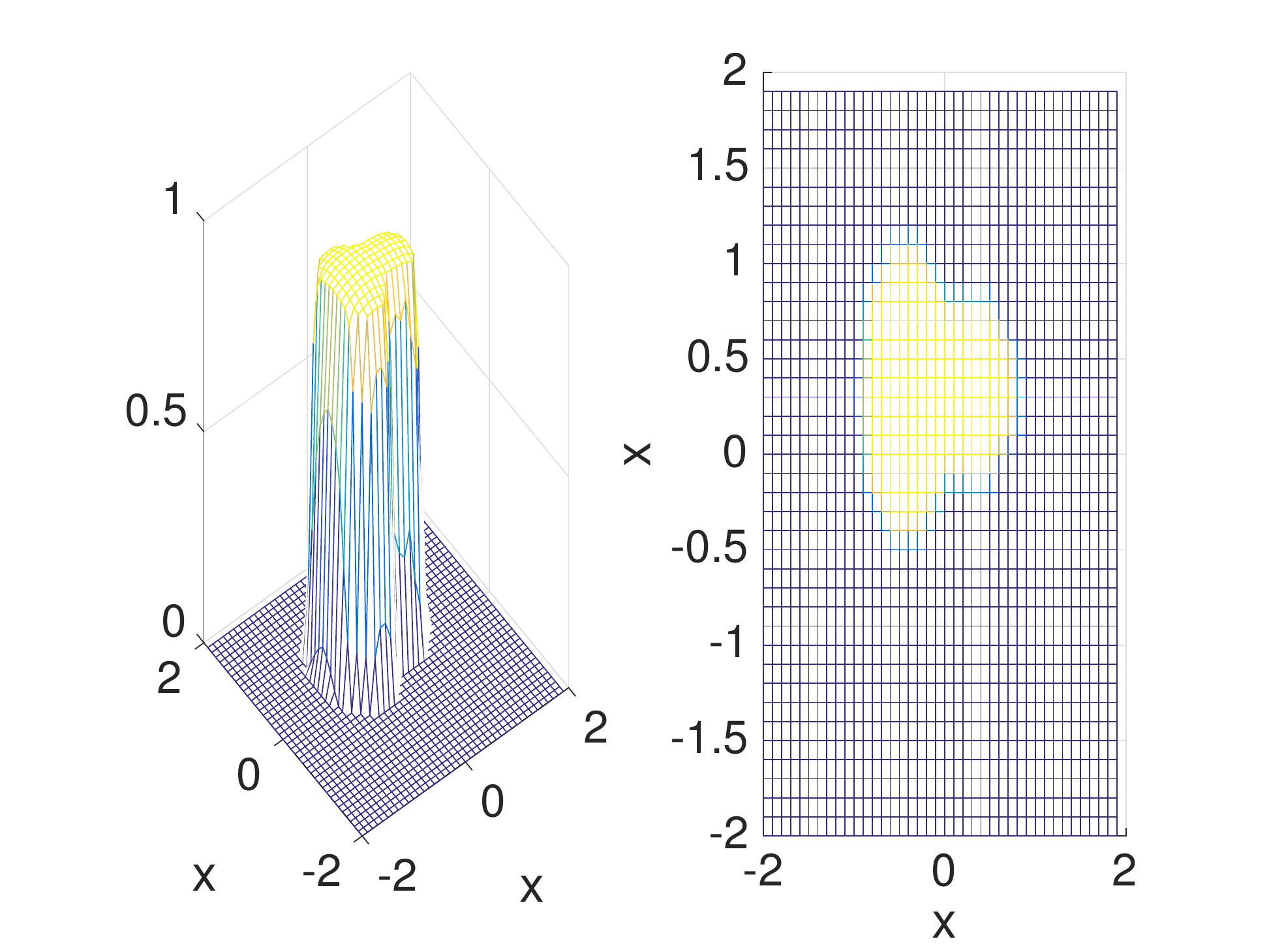}
\includegraphics[width = 0.45\textwidth]{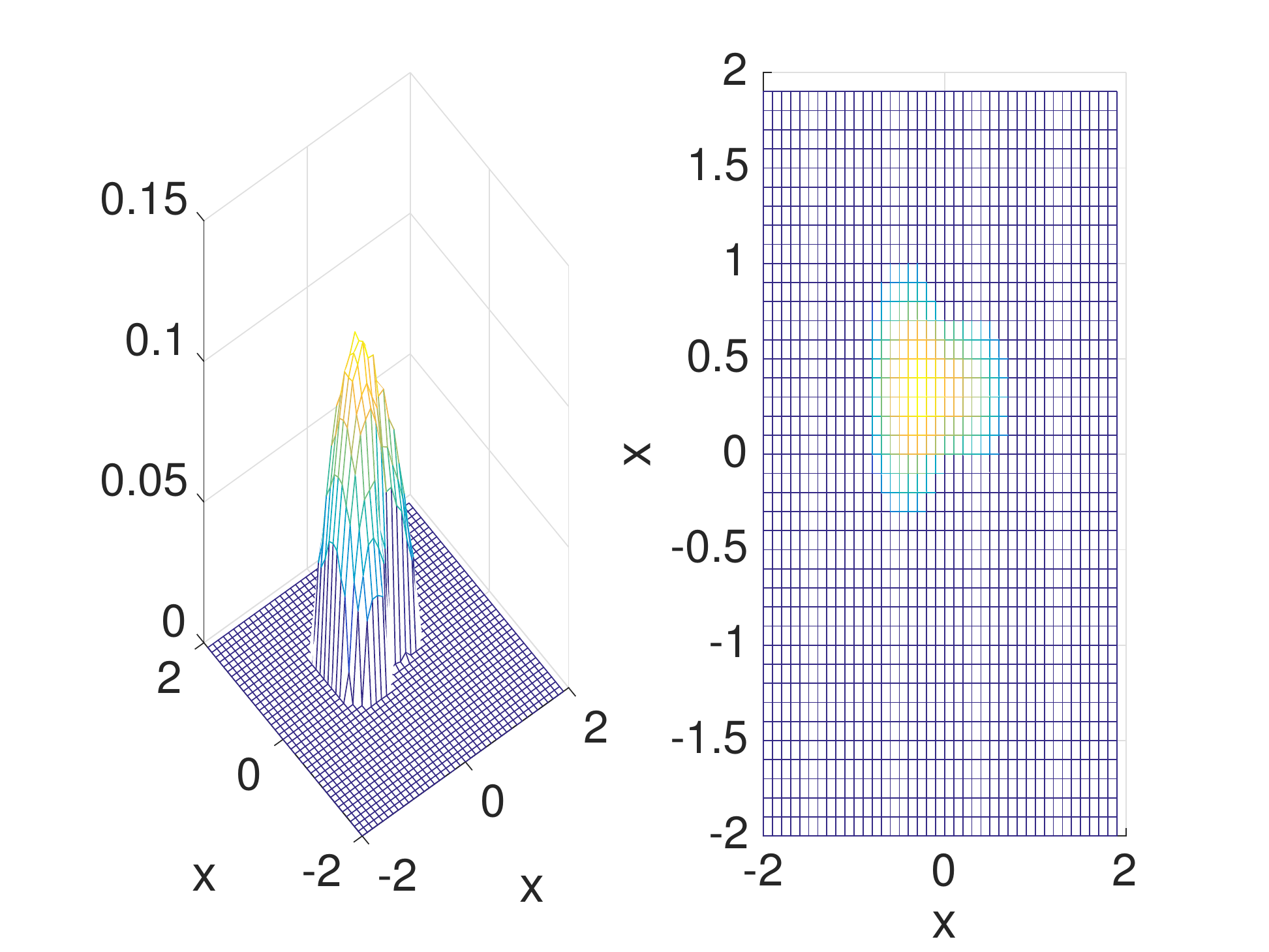}
\\
\includegraphics[width = 0.45\textwidth]{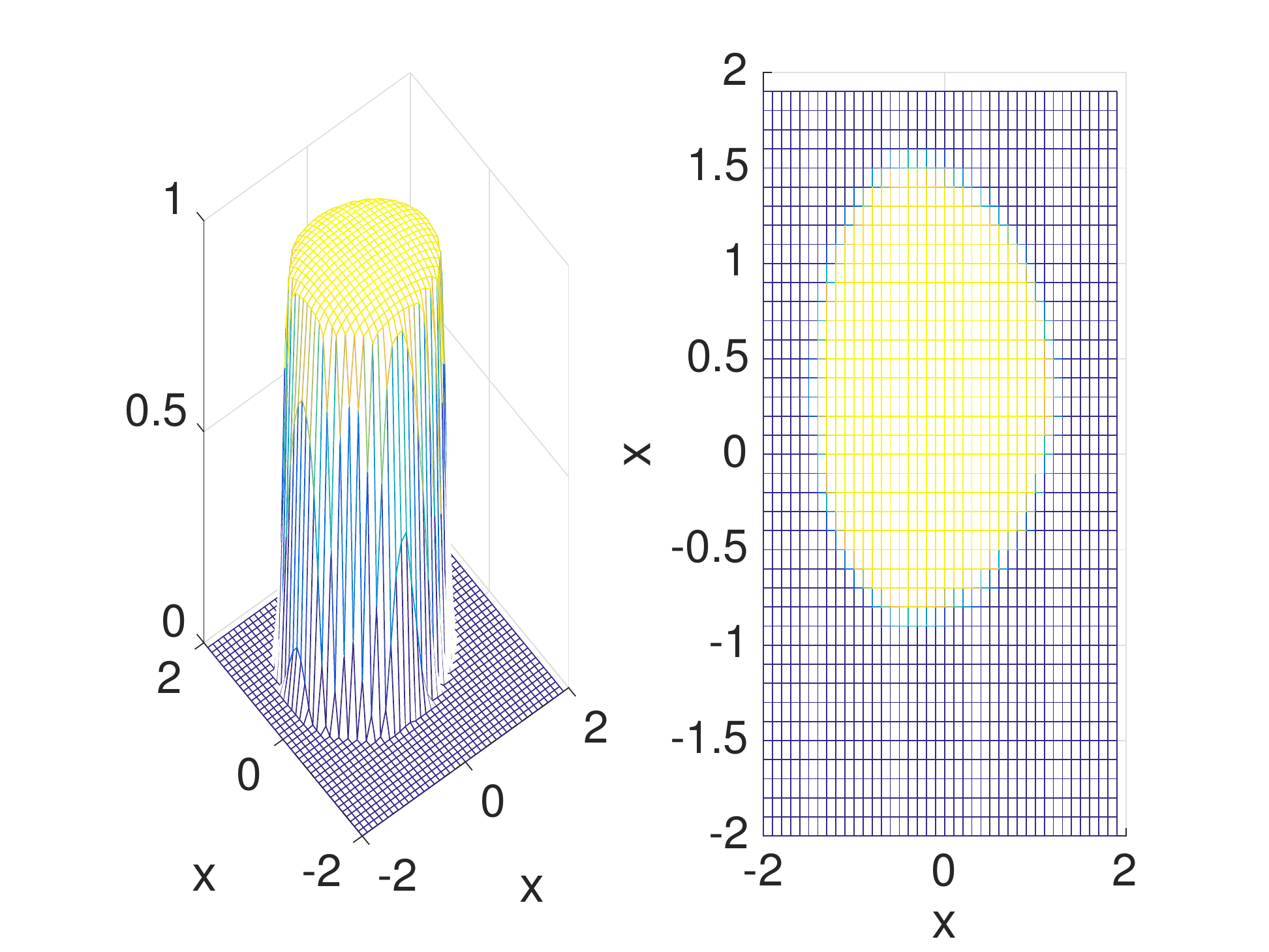}
\includegraphics[width = 0.45\textwidth]{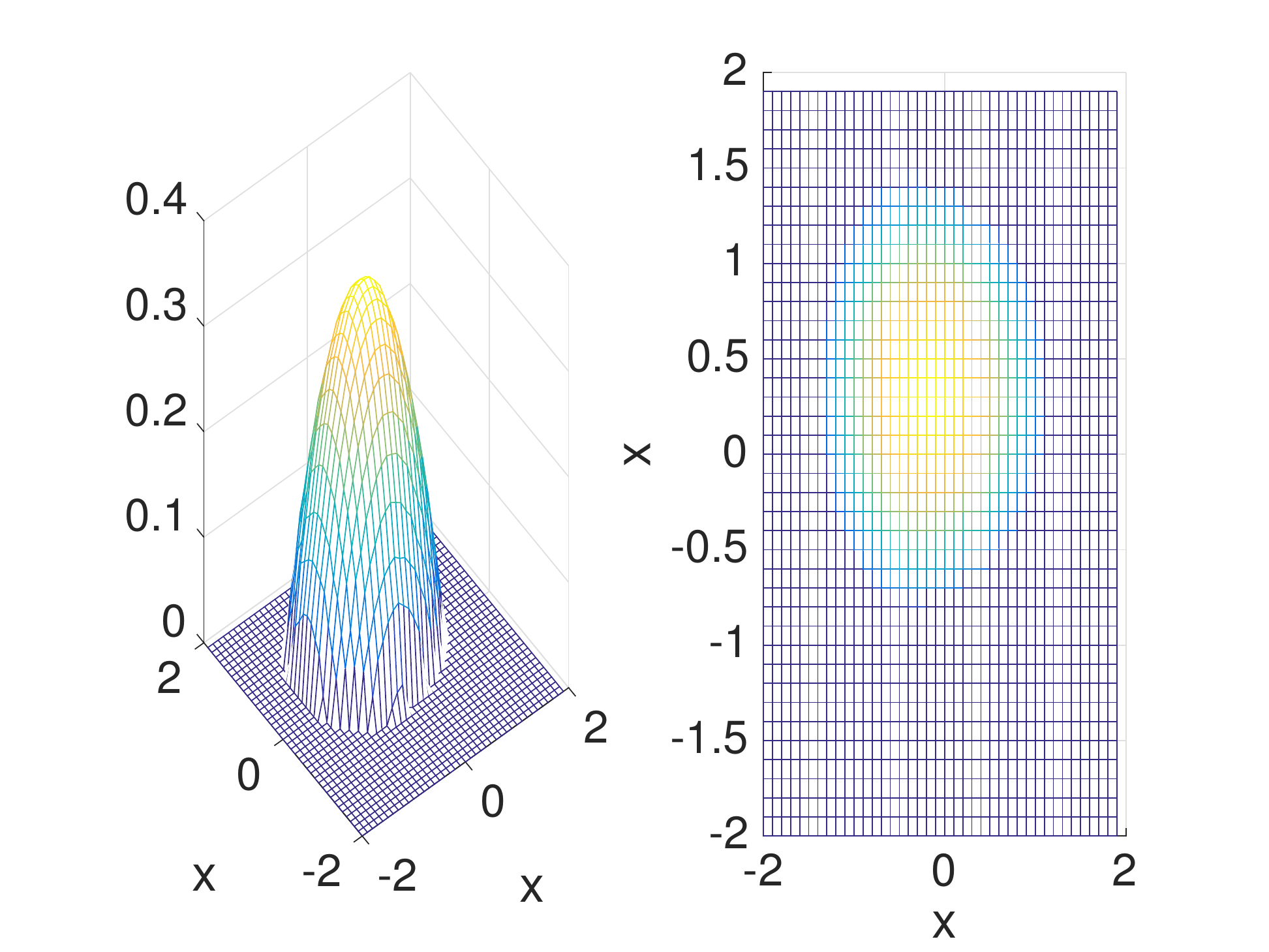}
\caption{ Plot of $\rho(x,y)$ with initial data \eqref{IC:2D-square} at different times: $t=0$ (top), $t=1$ (middle), $t=2$ (bottom). $m=40$. Left: $\rho$, right: $p$.}
\label{fig:2D-square}
\end{figure}

The second initial condition takes the form 
\begin{equation}\label{IC:2D-flower}
n(x,y,0) = \left\{ \begin{array}{cc} 0.9 & \sqrt{x^2 + y^2}-0.5-\sin(4\arctan(y/x))/2<0   \\ 0 & \text{otherwise} \end{array} \right. \,.
\end{equation}
and the evolution is gathered in Fig.~\ref{fig:2D-flower}.
These two cases indicate that, for big enough $m$, no matter what the initial condition is, the density $\rho$ will converge to a radially symmetric characteristic function.
\begin{figure}[!ht]
\centering
\includegraphics[width = 0.45\textwidth]{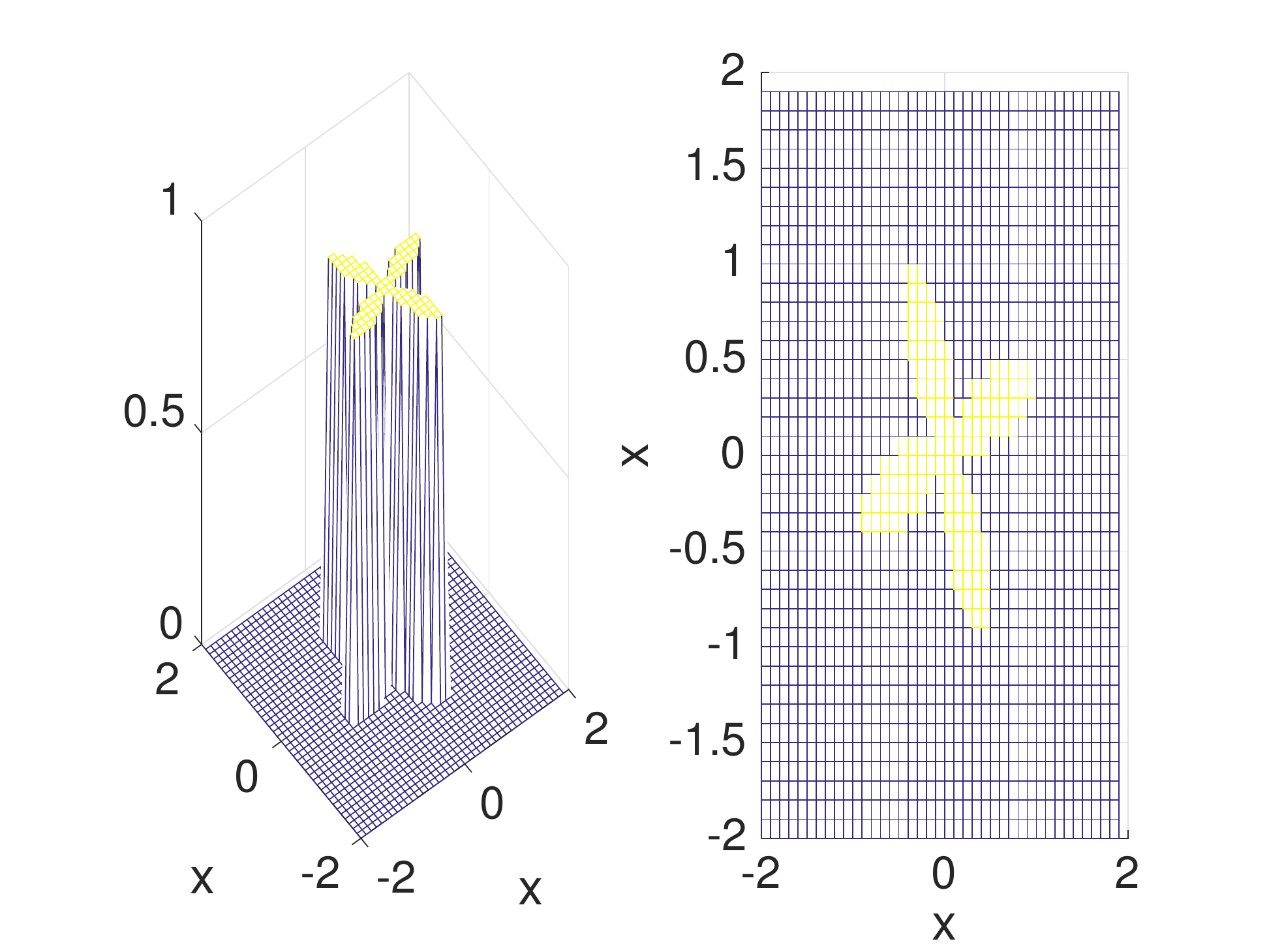}
\includegraphics[width = 0.45\textwidth]{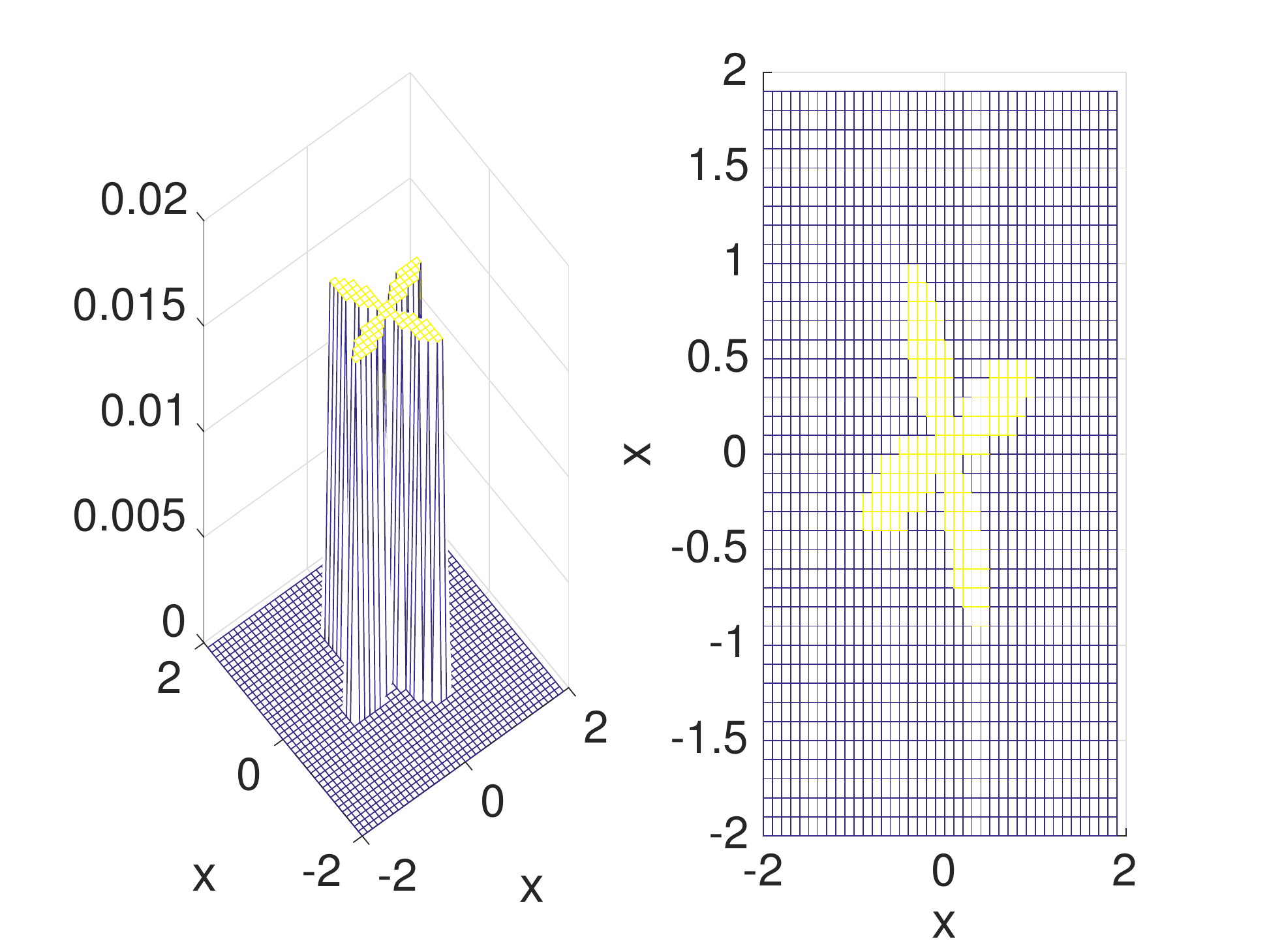}
\\
\includegraphics[width = 0.45\textwidth]{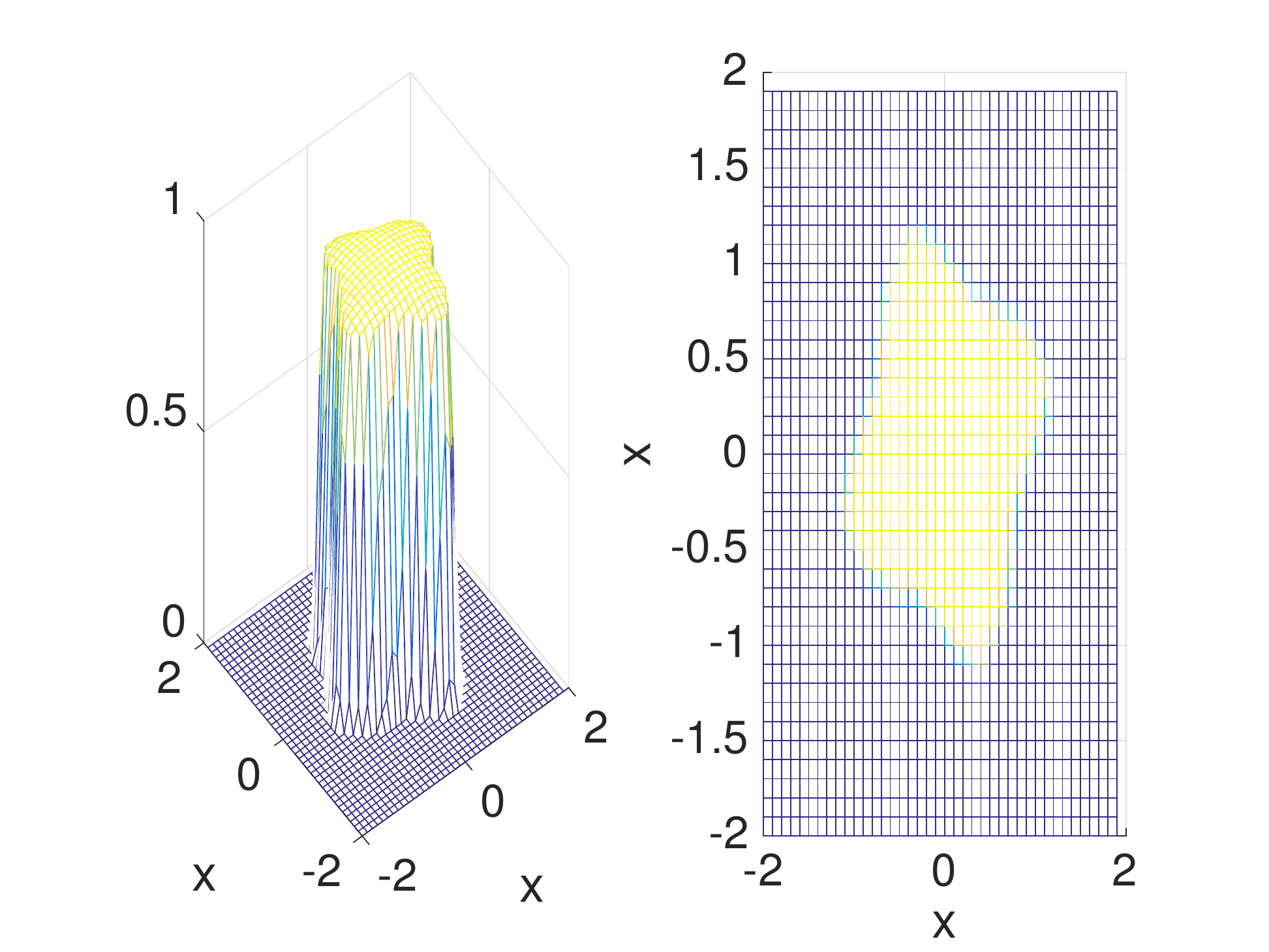}
\includegraphics[width = 0.45\textwidth]{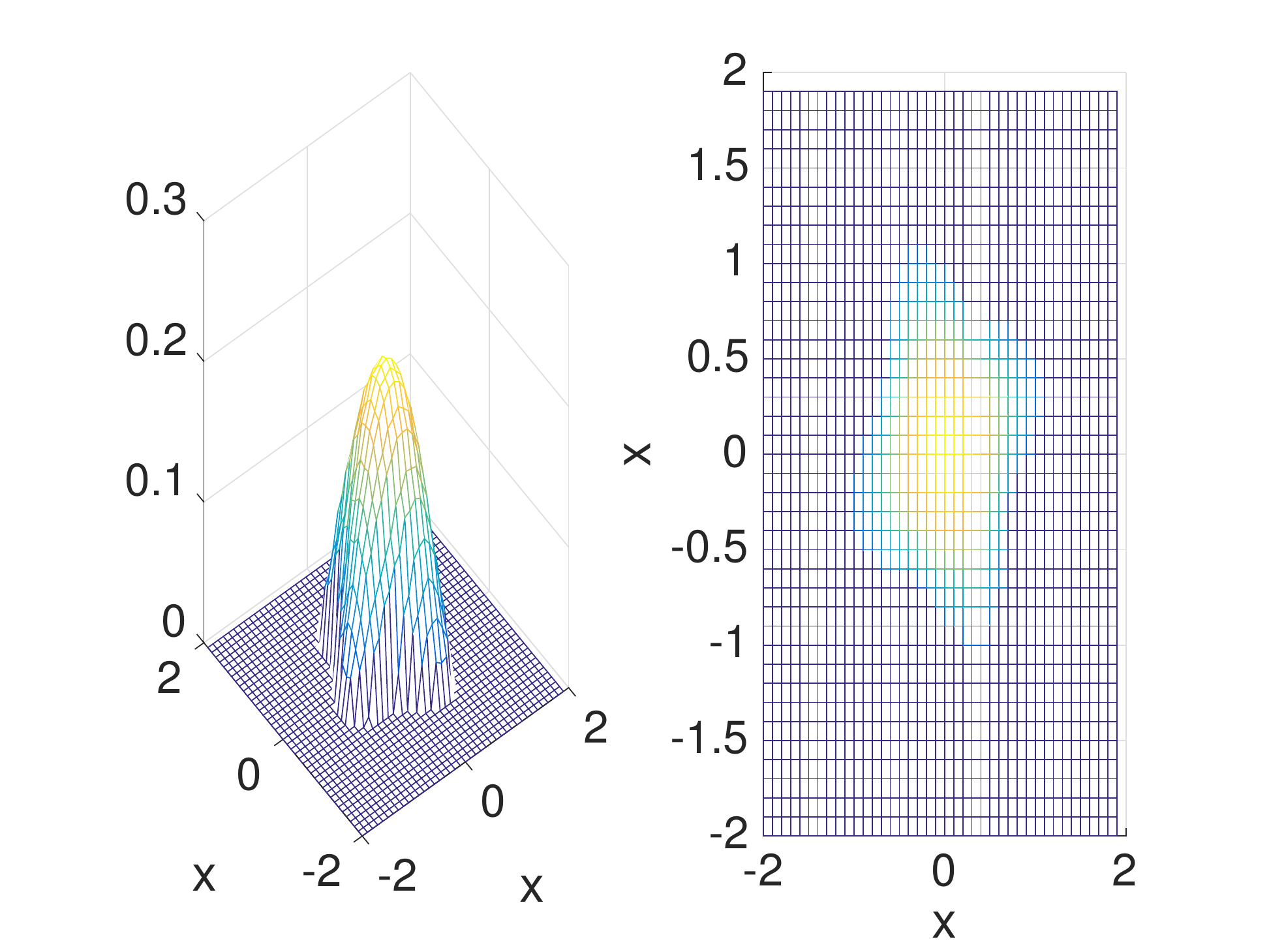}
\\
\includegraphics[width = 0.45\textwidth]{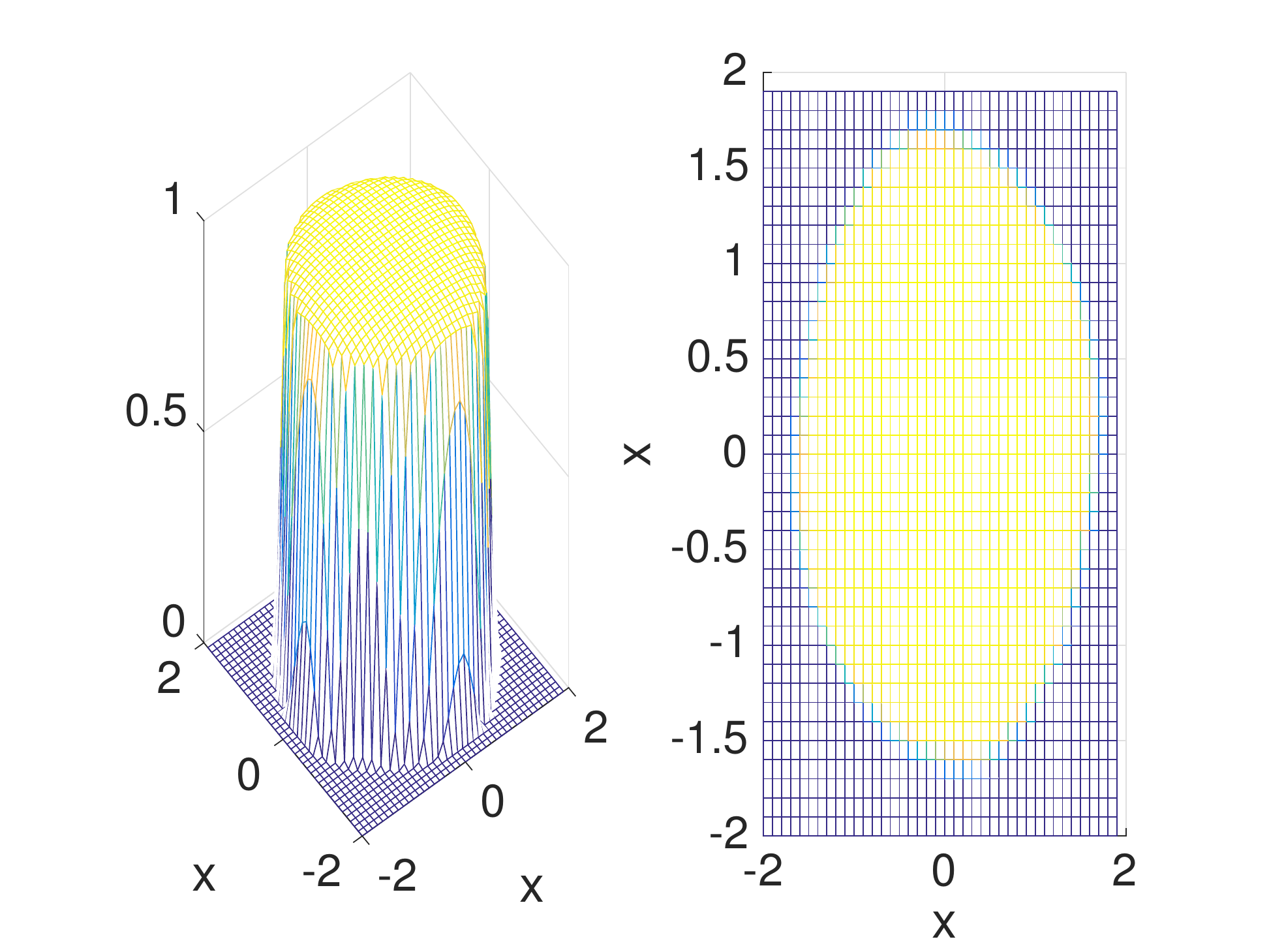}
\includegraphics[width = 0.45\textwidth]{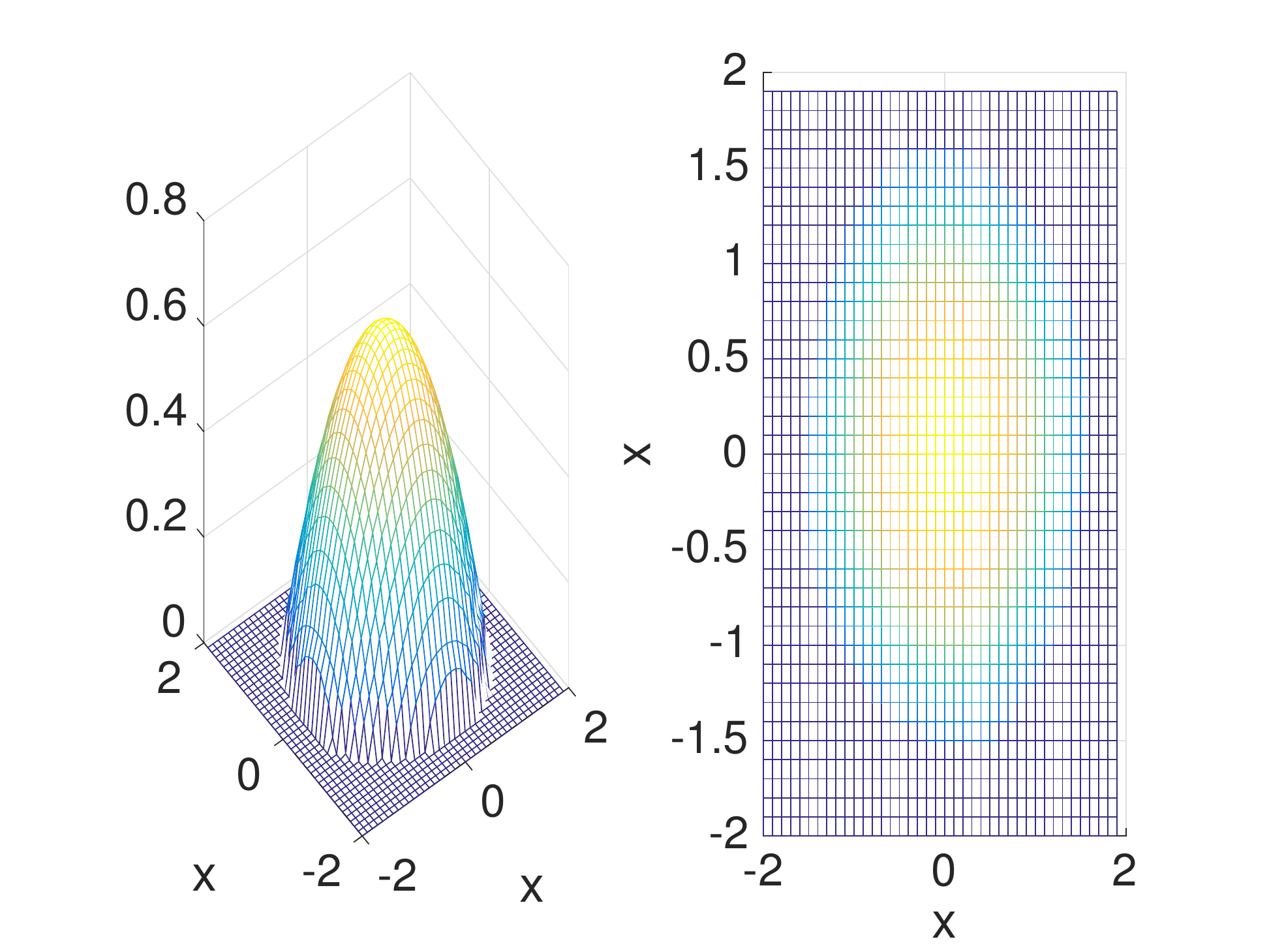}
\caption{ Plot of $\rho(x,y)$ with initial data \eqref{IC:2D-flower} at different times: $t=0$ (top), $t=1$ (middle), $t=2$ (bottom). $m=40$. Left: $\rho$, right: $p$.}
\label{fig:2D-flower}
\end{figure}

\section*{Acknowledgments}
 J. Liu is partially supported by KI-Net NSF RNMS grant No. 11-07444 and NSF grant DMS 1514826. M. Tang is supported by Science Challenge Project No. TZZT2017-A3-HT003-F and NSFC 91330203. Z. Zhou is partially supported by RNMS11-07444 (KI-Net). L. Wang is partially supported by the start up grant from SUNY Buffalo and NSF grant DMS 1620135. M. Tang and L. Wang would like to thank Prof. Jose Carrillo for fruitful discussions.

%
%


\begin{thebibliography}{99}

\bibitem{AKY}
D. Alexander, I. Kim, and Y. Yao, 
\emph{Quasi-static evolution and congested crowd transport}, 
Nonlinearity, {\bf 27} (2014), 823.

\bibitem{AWZ}
T. Arbogast, M. F. Wheeler, and N. Y. Zhang, 
\emph{A nonlinear mixed finite element method
for a degenerate parabolic equation arising in flow in porous media}, 
SIAM J. Numer. Anal., {\bf 33} (1996), 1669--1687.
 
\bibitem{BLM}
N. Bellomo,  N. K. LI, and P. K. Maini,  
\emph{On the foundations of cancer modeling: Selected
topics, speculations, and perspectives},
Math. Models Methods Appl. Sci. {\bf 4} (2008), 593--646. 

\bibitem{Degond}
F Berthelin, P Degond, M. Delitala, M. Rascle,
\emph{A model for the formation and evolution of traffic jams}
Archive for Rational Mechanics and Analysis {\bf 187} (2008), 185--220


\bibitem{BOBAM}  
R. Betteridge, M. R. Owen, H. M. Byrne, T. Alarc\'on, and P. K. Maini, 
{\it The impact of cell crowding and active
cell movement on vascular tumour growth},
Networks Heterogen. Media {\bf 1} (2006)  515--535.

\bibitem{BF}
M. Bessemoulin-Chatard and F. Filbet, 
\emph{A finite volume scheme for nonlinear degenerate parabolic equations}, 
SIAM J. Sci. Comput., {\bf 34} (2012), B559-B583. 

\bibitem{BCW}
M. Burger, J. A. Carrillo, and M. T. Wolfram, 
\emph{A mixed finite element method for nonlinear
diffusion equations}, 
Kinetic and Related Models, {\bf 3} (2010), 59-83.

\bibitem{ByDr} 
H. Byrne and D. Drasdo, 
\emph{Individual based and continuum models of growing cell populations: a comparison}, 
J. Math. Biol. (2009).

\bibitem{CCH} 
J. A. Carrillo, A. Chertock and Y. Huang, 
{\it A Finite-Volume Method for Nonlinear Nonlocal
Equations with a Gradient Flow Structure}, 
Communications in Computational Physics, {\bf 17} (2015)  233--258.

\bibitem{CLZ}
J. Chen, J.-G. Liu and Z. Zhou, 
\emph{On a Schrodinger-Landau-Lifshitz system: Variational structure and numerical methods}, 
Multiscale Modeling and Simulation, {14} (2016), 1463--1487.


\bibitem{CYZ}
R. Chen, X. F. Yang and H. Zhang, 
\emph{Second order, linear and unconditionally energy stable schemes for a hydrodynamic model of smectic-A liquid crystals}, submitted.

\bibitem{Chorin1}
A. J. Chorin, 
{\it The numerical solution of the Navier-Stokes equations for an incompressible fluid}, 
Bull. Am. Math. Soc., {\bf 73} (1967), 928--931.

\bibitem{Chorin2}
A. J. Chorin, 
{\it Numerical solution of the Navier-Stokes equations}, 
Math. Comput., {\bf 22} (1968), 745--762.


\bibitem{CKY}
K. Craig, I. Kim and Y. Yao, 
\emph{Congested aggregation via Newtonian interaction},  
arXiv preprint, arXiv:1603.03790 (2016).


\bibitem{EGHM}
R. Eymard, T. Gallout, R. Herbin, and A. Michel, 
\emph{Convergence of a finite volume scheme
for nonlinear degenerate parabolic equations}, 
Numer. Math., {\bf 92} (2002), 41--82.


\bibitem{Fredman}
A. Friedman, 
\emph{A Mathematical analysis and challenges arising from models of tumor growth},
Math. Model. Methods Appl. Sci. {\bf 17} (2007), 1751--1772.


\bibitem{Greenspan}
H. P. Greenspan,  
\emph{Models for the growth of a solid tumor by diffusion}, 
Stud. Appl. Math. {\bf 51} (1972), 317--340.

\bibitem{HJL}
J. Haack, S. Jin, and J.-G. Liu, 
{\it An all-speed asymptotic-preserving method for the isentropic Euler and Navier-Stokes equations}, Communications in Computational Physics {\bf 12} (2012), 955--980.

\bibitem{Howison}
S. D. Howison, 
\emph{Fingering in Hele-Shaw cells}, 
J. Fluid Mech. {\bf 167} (1986), 430--453.


\bibitem{JinXin}
S. Jin and Z. P. Xin, 
\emph{The relaxation schemes for systems of conservation laws in arbitrary space dimensions}, 
Commu. Pure Appl. Math., {\bf 48} (1995), 235--277.

\bibitem{KRT}
K. H. Karlsen, N. H. Risebro, and J. D. Towers, 
\emph{Upwind difference approximations for
degenerate parabolic convection-diffusion equations with a discontinuous coefficient}, 
IMA J. Numer. Anal., {\bf 22} (2002), p. 623.


\bibitem{LTWZ}
J.-G. Liu, M. Tang, L. Wang and Z. N. Zhou, 
\emph{Analysis and computation of some tumor growth models with nutrient: from cell density models to free boundary dynamics}, submitted.

\bibitem{LWZ}  
J.-G. Liu, L. Wang and Z. Zhou, 
\emph{Positivity-preserving  and  asymptotic preserving method  for 2D Keller-Segal equations}, Mathematics of Computation, accepted, 	arXiv:1610.03016.

\bibitem{LSZ}
Y. Liu, C. W. Shu, and M. Zhang, 
\emph{High order finite difference WENO schemes for nonlinear
degenerate parabolic equations}, 
SIAM J. Sci. Comput., {\bf 33} (2011), 939--965.

\bibitem{NPT}
G. Naldi, L. Pareschi, and G. Toscani, 
\emph{Relaxation schemes for partial differential equations and applications to degenerate diffusion problems}, Surv. Math. Ind. {\bf 10} (2002), 315--343.

\bibitem{Pbook}
B. Perthame, 
\emph{Some mathematical models of tumor growth},
https://www.ljll.math.upmc.fr/perthame/cours M2.pdf.

\bibitem{PQV} 
B. Perthame, F. Quir\`os and J.-L. V\`azquez, 
{\it The Hele-Shaw asymptotics for mechanical models of tumor growth}.

\bibitem{PTV}
B. Perthame, M. Tang  and N. Vauchelet,  
\emph{Traveling wave solution of the Hele-Shaw
model of tumor growth with nutrient}, 
Math. Model. Methods Appl. Sci. {\bf 24} (2014), 2601--2626.

\bibitem{PY}
I. S. Pop and W. A. Yong, 
\emph{A numerical approach to degenerate parabolic equations}, 
Numer. Math., {\bf 92} (2002), 357--381.

\bibitem {TVCVDP} M. Tang, N. Vauchelet, I. Cheddadi, I. Vignon-Clementel, D. Drasdo and 
B. Perthame,\emph{Composite waves for a cell population system
modeling tumor growth and invasion},  Chin. Ann. Math.  Ser. B, {\bf 34}
(2013), No.~2, 295--318. 
\bibitem{JJP} 
J. Ranft, M. Basana, J. Elgeti, J.-F. Joanny, J. Prost and F. J\"ulicher,  
\emph{Fluidization of tissues by cell division and apoptosis}. 
Proc. Natl. Acad. Sci. USA  {\bf 107} (2010)  20863--20868.

\bibitem{Vazquez}
J. L. Vazquez, 
{\it The Porous Medium Equation: Mathematical Theory}, 
Published to Oxford Scholarship Online,  2007, 
DOI:10.1093/acprof:oso/9780198569039.001.0001


\bibitem{WGE}
X.-P. Wang, C. J. Garc\i­a-Cervera and W. E, 
{\it A Gauss-Seidel projection method for micromagnetics stimulations}, 
J. Comput. Phys., {\bf 171} (2001), 357-372.

\bibitem{ZW}
Q. Zhang and Z. L. Wu, 
\emph{Numerical simulation for porous medium equation by local discontinuous
Galerkin finite element method}, J. Sci. Comput., {\bf 38} (2009), 127--148.




\end{thebibliography}
\end{document}